\documentclass[11pt]{article}

\usepackage{amsmath,amsthm,amssymb,mathrsfs}
\usepackage{graphicx,mdframed}
\usepackage{bbm}
\usepackage{subcaption}
\usepackage{float}

\font\tencmmib=cmmib10 \skewchar\tencmmib '60
\newfam\cmmibfam
\textfont\cmmibfam=\tencmmib

\def\lessim{\ \lower4pt\hbox{$
		\buildrel{\displaystyle <}\over\sim$}\ }
\def\gessim{\ \lower4pt\hbox{$\buildrel{\displaystyle >}
		\over\sim$}\ }

\def\hx{\hat{X}}
\def\ha{\hat{A}}
\def\hg{\hat{G}}

\newcommand{\la}{\langle}
\newcommand{\ra}{\rangle}

\newcommand{\e}{\mathbb{E}}
\newcommand{\p}{\mathbb{P}}

\newtheorem{lemma}{\bf Lemma}[section]
\newtheorem{definition}{\bf Definition}[section]
\newtheorem{theorem}{\bf Theorem}[section]

\newtheorem{remark}{\bf Remark}
\newtheorem{example}{\bf Example}[section]

\newtheorem{proposition}{\bf Proposition}[section]

\newmdtheoremenv{theo}{Theorem}

\numberwithin{equation}{section}

\newenvironment{Proof of lemma}{\noindent{\bf Proof of Lemma}}{\hfill$\Box$\newline}
\newenvironment{Proof of theorem}{\noindent{\it Proof of Theorem}}{\hfill\scriptsize{$\Box$}\newline}
\newenvironment{Proof of theorems}{\noindent{\bf Proof of Theorems}}{\hfill$\Box$\newline}
\newenvironment{Proof of proposition}{\noindent{\bf Proof of Proposition}}{\hfill$\Box$\newline}
\newenvironment{Proof of propositions}{\noindent{\bf Proof of Propositions}}{\hfill$\Box$\newline}
\newenvironment{Proof of exercise}{\noindent{\it Proof of Exercise:}}{\hfill$\Box$}

\begin{document}
	
		\title{Universality of Approximate Message Passing Algorithms}
	
	\author{Wei-Kuo Chen\thanks{University of Minnesota. Email: wkchen@umn.edu.  Partially supported by NSF grant DMS-17-52184} 
		\and Wai-Kit Lam\thanks{University of Minnesota. Email: wlam@umn.edu}}

\maketitle

\begin{abstract}	
We consider a broad class of Approximate Message Passing (AMP) algorithms  defined as a Lipschitzian functional iteration in terms of an $n\times n$ random symmetric matrix $A$. 
We establish universality in noise for this AMP in the $n$-limit and validate this behavior in a number of AMPs popularly adapted in compressed sensing, statistical inferences, and optimizations in spin glasses.  
\end{abstract}


\section{Introduction}

 Motivated by the ideas from  belief propagation algorithms, Approximate Message Passing (AMP) algorithms were initially introduced in the context of compressed sensing, see \cite{donoho2013information, donoho2013accurate, Donoho18914,donoho2010message}. Thereafter they have received great popularity in a number of emerging applications in data science, statistical physics, etc. concerning the development of  efficient algorithms for some randomized estimations and optimizations with large complexity.

One major application has been laid on the subject of matrix estimations, in which one aims to extract the structure of a signal matrix in a randomized environment. A popular setting is the so-called spiked model, where the data arrives as the sum of a noise, an $n\times n$ symmetric random matrix $A_n$, and the signal, an $n\times n$ symmetric low-rank matrix $Z_n$, 
\begin{align*}
\hat{A}_n:=A_n+Z_n.
\end{align*}
The goal is to recover the structure of $Z_n$ from the realization of the matrix $\hat{A}_n.$ A typical example one considered in the literature is when $A_n$ is the normalized Gaussian Wigner ensemble and $Z_n$ is given by
\begin{align}\label{eq1}
Z_n=\frac{1}{n}\sum_{\ell=1}^r\gamma_\ell z^\ell \otimes z^{\ell},
\end{align}
where $z^1,\ldots,z^r$ are non-random column vectors with $\|z^\ell\|_2=\sqrt{n}$ and the parameters $\gamma_1,\ldots,\gamma_r\geq 0$ are the signal-to-noise ratios (SNR's). In probability and statistics, this spiked model has been intensively studied by means of the spectral method, see \cite{BS06,BGN12,BGN11,CDMF09,FP07,John01,KY13,paul2007,Peche}. In Bayesian optimal approach, 
the setting often considered in the literature is to assume that the vectors $z^1,\ldots,z^r$ are randomized and their $i$-th marginal vectors, $(z_i^1,\ldots,z_i^r)$ for $1\leq i\leq n$, are independently sampled from a given prior distribution. It turns out that the corresponding Minimum Mean Square Error Estimator (MMSEE), $\e[(z^1,\ldots,z^r)|\ha_n]$, can be connected to the Gibbs expectation of the famous Sherrington-Kirkpatrick (SK) mean-field spin glass model arising from the statistical physics \cite{SK}.
One peculiar feature within this connection is that this model satisfies the so-called Nishimori identity, namely, the conditional distribution of the vectors $z^1,\ldots,z^r$ given the data $\hat{A}_n$ is equal to the distribution of the vector-valued spin configuration of the corresponding SK model. This allows one to fully understand the behavior of the MMSEE and its phase transition in terms of SNR's, see \cite{BDM+16,DAM+16,DM+14,KXZ+16,LM+16,LKZ17,M+17} for the recent progress. 

The study of the above Bayesian estimation arises a challenging computational problem in searching for polynomial-time algorithms in simulating the MMSEE. To this end,  AMP algorithms have been widely adapted \cite{KKM+16,MR16,MV+18,PSC141,PSC142,RF12,VSM15} and known to achieve a good level of success; in some cases, it allows to obtain the Bayesian-optimal error estimates, see \cite{DM+14,DM15,MV+18}. In addition to being useful in  matrix estimations, AMP has also been applied to a number of randomized optimization problems in mean-field spin glass models in recent years. In particular, it was shown in \cite{EMS+20,MON+19} that AMP allows to implement polynomial-time algorithms in the optimization of the SK Hamiltonian and its variants.

In these applications, the AMP algorithm is formulated as a sequence of $n$-dimensional vectors $(v^{[k]})_{k\geq 0}$ of the form 
\[
v^{[k+1]} = \ha_n f_k(v^{[k]},\ldots,v^{[0]}) - \sum_{j=1}^k b_{k,j} f_{j-1}(v^{[j-1]},\ldots,v^{[0]}),
\]
for \[
b_{k,j} = \frac{1}{n} \sum_{i=1}^n \frac{\partial f_k}{\partial v_i^{[j]}} (v_i^{[k]},\ldots,v_i^{[0]}),
\]
where $f_k\in C^1(\mathbb{R}^{k+1})$ and the two vectors $f_k(v^{[k]},\ldots,v^{[0]})$ and $f_{j-1}(v^{[j-1]},\ldots,v^{[0]})$ above are defined coordinate-wise by $v^{[k]},\ldots,v^{[0]}$ and $v^{[j-1]},\ldots,v^{[0]}$, respectively. The key component here is the initialization $v^{[0]}$; it influences the convergence of the AMP in the large $n$ limit.

When $Z_n$ is a zero matrix and the initialization $v^{[0]}$ is independent of $A_n$, it was known \cite{BLM15,BM+11,BMN+17,JM+13} that under mild assumptions on $f_j$'s, this iterative algorithm converges in the sense that for any Lipschitz function $\phi\in C(\mathbb{R}^{k+1}),$ almost surely,
\begin{align}\label{add:eq---1}
\lim_{n\to\infty}\frac{1}{n}\sum_{i=1}^n\phi(v_i^{[k]},\ldots,v_i^{[0]})= \e \phi(V_k,\ldots,V_0),
\end{align}
where $(V_k,\ldots,V_0)$ is a centered Gaussian random vector with covariance $$
\e V_{a+1}V_{b+1}=\e f_a(V_{a},\ldots,V_0)f_b(V_b,\ldots,V_0),\,\,\forall 0\leq a,b\leq k-1.
$$
When $Z_n$ is of the form \eqref{eq1} and the spectrum of $\ha_n$ exhibit the so-called Baik-Ben Arous-P\'ech\'e (BBP) phase transition \cite{BBP}, namely, the top eigenvalue of $\ha_n$ stays a gap away from the rest of the eigenvalues and the principal eigenvector is correlated to the prior, a recent paper \cite{MV+18} further showed that an analogous convergence remains valid when the AMP is initialized by  the principal eigenvector, see Example \ref{ex2} below. The typical way to use AMP is to select the functions $\phi,f_k,\ldots,f_0$ properly (usually are smooth and with bounded derivatives) so that the limit \eqref{add:eq---1} converges to the desired quantities of interest by adjusting the number of iteration $k$, see, e.g., \cite{Donoho18914,MON+19,MV+18}.


While the above convergences were known to be true when $A_n$ is Gaussian,  in this work we investigate their validity under general randomness. When the signal matrix is not presented, i.e., $Z_n\equiv 0$ (or equivalently, $\ha_n = A_n$) and $u^{[0]}$ is independent of $A_n$, this question was answered earlier in the work \cite{BLM15}, in which they showed that if the evolution functions $f_k,\ldots,f_0$ of the AMP are polynomials, then the AMP converges to the same limit of \eqref{add:eq---1} independent of the choice of the randomness on $A_n.$ In our setting, we consider a generalized AMP with Lipschitz evolution functions and let it iterate in the presence of the signal matrix. Our first main result validates the universality of the AMP. As a consequence, this implies that the universality established in \cite{BLM15} also holds for Lipschitz functions and under the presence of the signal matrix $Z_n$. (We note here that the work in \cite{BLM15} can be actually extended to Lipschitz evolution functions by \cite[Proposition~6]{BLM15} and its proof, but only when $Z_n\equiv 0$.) Furthermore, we show that universality of AMP with spectral initialization remains valid when the system exhibits the BBP phase transition.

Our approach is based on a Gaussian interpolation argument. In doing so, the central ingredient relies on a novel control on the moments of the partial derivatives of the AMP orbit with respect to the entries of the noise matrix $A_n.$ While our argument are formulated for the purpose of this paper, the same strategy is expected to be applicable in more general settings.

\section{Main results}

We begin with some notations.  For any column vectors $u^0,u^1,\ldots,u^k\in \mathbb{R}^n$ and a function $f:\mathbb{R}^{k+1}\to \mathbb{R}$, we define $f(u^k,u^{k-1},\ldots,u^0)$ as a column vector by 
$$
f(u^k,u^{k-1},\ldots,u^0)_i=f(u_i^k,u_i^{k-1},\ldots,u_i^0).
$$
For $x,y\in \mathbb{R}^n,$ set $\|x\|_2=\bigl(\sum_{i=1}^nx_i^2\bigr)^{1/2}$ and $\la x,y\ra=\sum_{i=1}^nx_iy_i.$ Let $M_n(\mathbb{R})$ be the collection of all  $n\times n$ real-valued symmetric matrices. For $X,Z\in M_n(\mathbb{R})$, denote
$$
X_n=\frac{X}{\sqrt{n}}
$$
and
$$
\hat{X}_n=\frac{X}{\sqrt{n}}+\frac{Z}{n}.
$$
The generalized approximate message passing is formulated as follows. 

\begin{definition}\label{def1}
	Let $u^{[0]}:M_n(\mathbb{R})\to \mathbb{R}^n$ be a measurable function. For any $k\geq 0$, let $F_k\in C(\mathbb{R}^{k+1})$ be Lipschitz. The generalized AMP orbit corresponding to $(X,Z)$, $(F_k)_{k\geq 0},$ and $u^{[0]}$ is the sequence of vector-valued functions $u^{[k]}:M_n(\mathbb{R})\to \mathbb{R}^n$ for $k\geq 0$ defined iteratively by 
	\begin{align*}
	u^{[k+1]}(X)&=F_k(\hat{X}_nu^{[k]}(X),u^{[k-1]}(X),u^{[k-2]}(X),\ldots,u^{[0]}(X)).
	\end{align*}
\end{definition}

We now specify the randomness on $X$, $Z$, and $u^0$ Let $\sigma>0$ be fixed.  For any $n\geq 1,$ let ${u}^0=(u_i^0)_{i\in [n]}$ be an $n$-dimensional random vector and $Z=(z_{ii'})_{i,i'\in [n]}$ be an $n\times n$ random  symmetric matrix. Assume that there exists a constant $C(\sigma)>0$ such that
\begin{align}\label{add:eq--1}
\sup_{n\geq1}\Bigl(\e \exp\Bigl(\frac{\|u^0\|_2^2}{\sigma n}\Bigr),\max_{i\in[n]}\e \exp\Bigl(\frac{|u_i^0|}{\sigma}\Bigr),\max_{i,i'\in[n]}\e \exp\Bigl(\frac{|z_{ii'}|}{\sigma}\Bigr)\Bigr)\leq C(\sigma).
\end{align} 
Suppose that $A=(a_{ii'})_{i,i'\in[n]}$ is an $n\times n$ random symmetric matrix, whose upper triangular entries are independent with zero mean and unit variance and are $\sigma$-subgaussian, i.e., $\e e^{\lambda a_{ii'}}\leq e^{\lambda^2\sigma^2/2}$ for all $\lambda\in \mathbb{R}$.

We further assume that $A$ is independent of $u^0$ and $Z$, but allow $u^0$ and $Z$ to be dependent on each other. An important example of $A$ is when the entries $a_{ii'}$'s are standard normal. In this case, we denote $A$ by $G$ and we call Definition \ref{def1} associated to $X=G$ a Gaussian AMP. Our main result shows that if we initialize $u^{[0]}(X)=u^0$, then the AMP corresponding to any $A$ is essentially the same as the Gaussian AMP.

\begin{theorem}\label{thm0}
	Let $u^{[0]}(X)=u^0$. For any $k\geq 0$ and Lipschitz function $\phi$ on $\mathbb{R}^{k+1}$, we have that in probability
	\begin{align*}
	\lim_{n\to\infty}\bigl|\Phi_{k,n}(A)-\Phi_{k,n}(G)\bigr|=0,
	\end{align*}
	where
	\begin{align}\label{add:eq2}
	\Phi_{k,n}(X):=\frac{1}{n}\sum_{i=1}^n\phi(u_i^{[k]}(X),\ldots,u_i^{[0]}(X)),\,\,X\in M_n(\mathbb{R}).
	\end{align}
\end{theorem} 

Next we introduce the AMP used in the matrix estimation and some optimization problems in mean-field spin glasses.
 
 \begin{definition}\label{def2}
 	Let $f_{-1}\equiv 0$. For $k\geq 0,$ assume that $f_k\in C^{1}(\mathbb{R}^{k+1})$ is Lipschitz and its first-order partial derivatives are also Lipschitz. Let $X\in M_n(\mathbb{R}).$ Starting from an initialization $v^{[0]}(X)$, define the AMP orbit for $k\geq 0$ iteratively by 
\begin{align}\label{AMP:ord}
v^{[k+1]}(X) = \hx_n f_k(v^{[k]}(X),\ldots,v^{[0]}(X)) - \sum_{j=1}^k b_{k,j}(X) f_{j-1}(v^{[j-1]}(X),\ldots,v^{[0]}(X)),
\end{align}
where
\[
b_{k,j}(X) = \frac{1}{n} \sum_{i=1}^n \frac{\partial f_k}{\partial v_i^{[j]}(X)} (v_i^{[k]}(X),\ldots,v_i^{[0]}(X)).
\]  
\end{definition}
Note that Definition~\ref{def2} is not a direct example of the generalized AMP in Definition \ref{def1} due to the term $b_{k,j}(X)$. Nevertheless, since $b_{k,j}(X)$ is an average of the partial derivatives, this quantity is essentially indistinguishable between different randomness and this allows us to establish the following universality.

\begin{theorem}\label{cor1}
    Let $v^{[0]}(X)=u^0.$ For any $k\geq 0$, if $\phi$ is Lipschitz on $\mathbb{R}^{k+1}$, then in probability,
	\begin{align*}
	\lim_{n\to\infty}\bigl|\phi_{k,n}(A)-\phi_{k,n}(G)\bigr|=0,
	\end{align*}
	where 
	\begin{align*}
	\phi_{k,n}(X):=\frac{1}{n}\sum_{i=1}^n\phi(v_i^{[k]}(X),\ldots,v_i^{[0]}(X)),\,\,X\in M_n(\mathbb{R}).
	\end{align*}
\end{theorem}

\begin{remark}\rm As pointed out in the introduction, it was known \cite{BLM15,BM+11,BMN+17,JM+13} that if $Z=0,$ the Gaussian AMP in Definition \ref{def2} converges, see \eqref{add:eq---1}. If $f_0,\ldots,f_k$ are polynomials and $Z=0$, it was further understood in  \cite{BLM15} that this convergence is independent of the choice of the randomness of $A$. Theorem~\ref{cor1} here  extends this universality to Lipschitz functions and in the presence of $Z$. We refer the reader to check \cite{EMS+20,MON+19} for the usage of this AMP in the optimization of the SK Hamiltonian and related models. 
\end{remark}

\begin{example}\rm \label{ex1}
    For $\gamma\geq 0,$ set $Z=\gamma u^0\otimes u^0$. In this case, $\hat A_n$ is a rank-one spiked matrix, $$
	\ha_n=\frac{A}{\sqrt{n}}+\gamma\frac{u^0\otimes u^0}{n}.
	$$ 
	In matrix estimation, one would like to recover the vector $u^0$ from the realization of $\ha_n$. When $A=G$, the MMSEE, $\e[u^0|\ha_n]$, is popularly adapted for this purpose and it can be simulated via the AMP in Definition \ref{def2} with specifically chosen functions $f_k$'s, see, e.g., \cite{DM+14}. Theorem \ref{cor1} here indicates that in a non-Gaussian noise environment, the AMP in Definition \ref{def2} still allows to implement the same simulation for $u^0$ as the the Gaussian AMP.
\end{example}

Recall that the initialization $u^0$ and the signal matrix $Z$ are assumed to be  independent of the noise. In Example \ref{ex1}, since the MMSEE is a measurable function of the spiked matrix $\ha_n$, it is often more desirable that the initialization depends on $\ha_n$, as it should provide a better estimate for the MMSEE. When $A=G$, an attempt along this line has been successfully carried out in \cite{MV+18}. Our last main result addresses universality towards this direction. For any $X\in M_n(\mathbb{R})$, denote by $\lambda_1(\hx_n)\geq \lambda_2(\hx_n)\geq\cdots\geq\lambda_n(\hx_n)$  the eigenvalues of $\hx_n$ and by $\psi^1(\hx_n)$ the top eigenvector of $\hx_n$ with $\|\psi^1(\hx_n)\|_2=\sqrt{n}.$ Set
\begin{align}\label{psi}
\psi(X)=\mbox{sign}\bigl(\la \psi^1(\hx_n),u^0\ra\bigr)\psi^1(\hx_n)
\end{align}
whenever $
\la \psi^1(\hx_n),u^0\ra\neq 0.
$
Note that although there are two possible choices of $\psi^1(\hx_n)$ up to a sign, the definition $\psi(X)$ here is not influenced by this difference.
 
\begin{theorem}\label{cor2}
 Assume that
	\begin{equation}
	\begin{split}\label{cor2:eq0}
	\liminf_{n\to\infty} \lambda_1(\hg_n)> 	\max\bigl(\limsup_{n\to\infty}\max_{2\leq \ell\leq n}\bigl|\lambda_\ell(\hg_n)\bigr|,1\bigr),\\
		\liminf_{n\to\infty}\lambda_1(\ha_n)> \max\bigl(\limsup_{n\to\infty}\max_{2\leq \ell\leq n}\bigl|\lambda_\ell(\ha_n)\bigr|,1\bigr),
	\end{split}
	\end{equation}
	and
	\begin{align}\label{cor2:eq2}
	\liminf_{n\to\infty}\frac{1}{n}\min\bigl(\bigl|\la \psi^1(\hg_n),u^0\ra\bigr|,\bigl|\la \psi^1(\ha_n),u^0\ra\bigr|\bigr)>0.
	\end{align}
	Consider the AMP orbit $(v^{[\ell]})_{0\leq \ell\leq k}$ defined in Definition \ref{def2}. Let $v^{[0]}(X)=\psi(X)$. If $\phi\in C(\mathbb{R}^{k+1})$ is Lipschitz, then in probability,
	\begin{align*}
	\lim_{n\to\infty} |\phi_{k,n}(A)-\phi_{k,n}(G)|=0.
	\end{align*}
\end{theorem}

The assumptions \eqref{cor2:eq0} and \eqref{cor2:eq2} say that the top eigenvalue stays a gap away from the rest of the eigenvalues and the principal eigenvector is correlated to the prior vector $u^0.$ These behaviors are not only required in our proofs for technical purposes, but also appear to be quite typical in the BBP phase transition, see the following example.

\begin{example}\label{ex2} \rm Recall $\ha_n$ from Example \ref{ex1}. Let $u^0=(u_1^0,\ldots, u_n^0)$  for $u_1^0,\ldots,u_n^0\stackrel{i.i.d.}{\thicksim} w,$ where $w$ is a centered random variable  with compact support and unit variance. Recall from \cite{BGN11} that the BBP transition point is equal to $1$: If $\gamma<1$, 
	\begin{align}
	\begin{split}\label{bbp0}
	\lim_{n\to\infty}\lambda_1(\hg_n)&=2,\,\,a.s.,\\
	\lim_{n\to\infty}\frac{1}{n}\bigl\la \psi^1(\hg_n),u^0\bigr\ra&=0,\,\,a.s.;
	\end{split}
	\end{align} 
	if	 $\gamma>1,$	
	\begin{align}
	\begin{split}\label{bbp}
	\lim_{n\to\infty}\lambda_1(\hg_n)&=\gamma+\gamma^{-1}>2,\,\,a.s.,\\
	\lim_{n\to\infty}\frac{1}{n}\bigl\la \psi^1(\hg_n),u^0\bigr\ra&=\sqrt{1-\gamma^{-2}}>0,\,\,a.s..
	\end{split}
	\end{align} 
	These imply that the spectral method can be used to gain useful information about $u^0$ only if the SNR exceeds the critical threshold, i.e., $\gamma>1,$ as in this case the principal eigenvector is positively correlated to $u^0$. In \cite{MV+18}, the convergence of AMP in Definition \ref{def1} initialized by the top eigenvector was investigated, which states that again when $\gamma>1,$ 
	\begin{align}\label{add:eq---2}
	\lim_{n\to\infty}\frac{1}{n}\sum_{i=1}^n\phi(u_i^{[0]}(G),u_i^{[k]}(G))=\e \phi(w,\mu_kw+\sigma_k g).
	\end{align}
	Here, starting from $\mu_0=\sqrt{1-\gamma^2}$ and $\sigma_0=1/\gamma$, $(\mu_k)_{k\geq 1}$ and $(\sigma_k)_{k\geq 1}$ are defined through
	\begin{align*}
	\mu_{k+1}&=\gamma \e[wf_k(\mu_kw+\sigma_kg)],\\
\sigma_{k+1}^2&=\e [f_k(\mu_k w+\sigma_kg)^2],
	\end{align*}
	where $g\thicksim N(0,1)$ is independent of $w.$ Note that $\hg_n$ is a perturbation of $G_n$ by a rank-one matrix. The eigenvalue interlacing property implies that for any small $\delta>0$, asymptotically
	\begin{align*}
	-\sqrt{2}-\delta\leq \lambda_n(G_n)\leq\lambda_{i}(\hg_n)\leq \lambda_1(G_n)\leq \sqrt{2}+\delta
	\end{align*}
	 for all $2\leq i\leq n,$ where $\lambda_1(G_n)$ and $\lambda_n(G_n)$ are the largest and smallest eigenvalues of $G_n$, respectively. Note that this inequality, \eqref{bbp0}, and \eqref{bbp} are also valid for $A.$ Hence,  the assumptions of \eqref{cor2:eq0} and \eqref{cor2:eq2} are valid and as a result, the convergence of \eqref{add:eq---2} is  universal in probability.
\end{example}

Our approach to proving Theorem \ref{thm0} is to match the first and second moments of $\Phi_{k,n}$ between $A$ and $G$, respectively. To this end, we define a Gaussian interpolation $X=A(t):=\sqrt{t}A+\sqrt{1-t}G$ for $0\leq t\leq 1$ and control the $t$-derivatives of $\e\Phi_{k,n}(A(t))$ and $\e\Phi_{k,n}(A(t))^2.$  The hope is that if the total number of the terms as well as their orders appearing in these derivatives are small enough, then we anticipate that an application of the approximate Gaussian integration by parts would make the derivatives small. However, due to the iteration of the AMP, these derivatives involve highly complicated Hadamard products of a large number of column vectors in terms of the higher order partial derivatives of $u^{[\ell]}(X)$ and $\hx_nu^{[\ell]}(X)$. As a result, the control of their $p$-th moments are extremely delicate, especially for those of $\hx_nu^{[\ell]}(X)$. The novelty of our analysis adapts a Taylor expansion of the derivatives up to the $p$-th order, which allows us to extract the dependence of the $i$-th row of $X$ out of the derivatives. This combining with a subtle moment computation in this expansion perfectly cancels out the majority of the smaller order terms and yields the following moment controls (see Proposition \ref{prop2} and Lemma \ref{lem-1}) that for any $p\geq 1,$ there exists a universal constant $C>0$ such that for any collection $P$ of variables $x_{ii'}$ for $i,i'\in [n]$ counting multiplicities with $|P|=m$, we have
\begin{align*}
\sup_{n\geq 2}\sup_{i\in[n]}\Bigl(\e\bigl|\partial_Pu^{[k]}(A)_i\bigr|^p\Bigr)^{1/p}&\leq \frac{C}{n^{m/2}},\\
\sup_{n\geq 2}\sup_{i\in[n]}\Bigl(\e\bigl|\partial_P\bigl(\hat Au^{[k]}(A)\bigr)_i\bigr|^p\Bigr)^{1/p}&\leq \frac{C}{n^{m/2}},
\end{align*}
where $\partial_P$ is the partial derivatives with respect to the variables in $P.$
Using the Markov inequality and the union bound, these yield a uniform control on the derivatives that for any $P$ with $|P|=m$ and $\delta>0$, with probability at least $1-Cn^{-\delta}$, 
\begin{align*}
\max_{i\in [n]} \bigl|\partial_Pu^{[k]}(A)_i\bigr|\leq \frac{1}{n^{\frac{m}{2}-\delta-\frac{1}{p}}},\\
\max_{i\in [n]} \bigl|\partial_P\bigl(\ha u^{[k]}(A)\bigr)_i\bigr|\leq \frac{1}{n^{\frac{m}{2}-\delta-\frac{1}{p}}}.
\end{align*}
Once Theorem \ref{thm0} is established, the proof of Theorem \ref{cor1} follows essentially by a special choice of the functions $F_0,\ldots,F_k,\ldots$. Although the term $b_{k,j}(X)$ in \eqref{AMP:ord} relies on all coordinates, its form averages out the partial derivatives and consequently,  $b_{k,j}(A)$ and $b_{k,j}(G)$ are asymptotically equal in probability, which is already enough to establish Theorem \ref{cor1} following an induction argument. Lastly to show Theorem \ref{cor2}, recall that while both AMP's in Theorems \ref{cor1} and \ref{cor2} share the same iteration procedure, their initializations are of different kind; the former is initialized independent of $A_n$, but the latter adapts the principal eigenvector of $\hat{A}_n$. We show that this eigenvector can be approximated very well by the power method (see Lemma \ref{pm}). In view of this method, it is essentially a special case of our generalized AMP with the choice $F_k(x_k,\ldots,x_0)=\hx_n x_k$ and an analogous argument as that for Theorem \ref{cor1} allows to establish Theorem \ref{cor2}. One technicality here is that in order to guarantee the convergence of the power method, one would have to choose the initialization carefully and ensure that the principle eigenvalue of $\ha_n$ is well-separated from the other eigenvalues. This explains why the assumptions \eqref{cor2:eq0} and \eqref{cor2:eq2} need to be in position.

We mention that many works in the literature, e.g., \cite{JM+13}, also established high-dimensional version of the AMP, in which the functions $f_k$ are allowed to be of vector-valued. In addition, it was also discussed in \cite{MV+18} that one can initialize the AMP in Definition \ref{def2} via other leading eigenvectors, whose corresponding eigenvalues exhibit the BBP phase transition. In view of the present approach, it is plausible that universality under these settings can be obtained from our results by a similar line of derivation. We do not address these directions here.

The rest of the paper is organized as follows. Sections \ref{sec:con}-\ref{sec:sec5} are the preparation for the proof of Theorem \ref{thm0}. Section \ref{sec:con} establishes a Gaussian concentration inequality for the function $\Phi_{k,n}(X)$ as well as a number of prior controls on the AMP orbit. In Section~\ref{sec:sec3.5}, we show that in proving Theorem~\ref{thm0}, it suffices to assume that $\phi$ and $F_k$'s are smooth with uniformly bounded derivatives. Section~\ref{sec:sec4} provides the main estimates on the moments of the derivatives of the AMP orbits. In Section~\ref{sec:sec5}, we carry out our interpolation argument and present the proof of Theorem~\ref{thm0}. The proofs of Theorems~\ref{cor1} and \ref{cor2} are provided in Sections~\ref{sec:sec6} and \ref{sec:sec7}, respectively. Finally, the Appendix gathers error estimates of some approximate Gaussian integration by parts.

\section{Lipschitz property and concentration inequality}\label{sec:con}

Consider the AMP in Definition \ref{def1} with initialization $u^{[0]}(X)=u^0.$ In this section, we establish a Lipschitz property for this AMP and a concentration inequality for $\Phi_{k,n}(G)$. These will be used later in the proof of Theorem \ref{thm0}. Recall that the functions $F_k$ in Definition \ref{def1} are Lipschitz. Let $\eta_k$ be the Lipschitz constant of $F_k$. For any $X\in M_n(\mathbb{R}),$ denote by $\|X\|_2$ the $\ell_2-\ell_2$ operator norm of $X.$

\begin{proposition}\label{prop0}
	If $\phi\in C(\mathbb{R}^{k+1})$ is Lipschitz with Lipschitz constant $\eta>0,$ then we have that
	\begin{align*}
	|\Phi_{k,n}(X)-\Phi_{k,n}(Y)|&\leq k\eta \|X_n-Y_n\|_2\Bigl(\frac{\|u^{[0]}(X)\|_2}{\sqrt{n}}+1\Bigr)\sum_{\ell=1}^k\Theta_\ell(Y)\Delta_{\ell}(X)\\
	&\qquad\qquad\quad+\frac{k\eta\|u^{[0]}(X)-u^{[0]}(Y)\|_2}{\sqrt{n}}\sum_{\ell=1}^k\Theta_\ell(Y),
	\end{align*}
	where 
	 \begin{align}\label{prop0:eq2}
	\Delta_k(X):=2^k\Bigl((\eta_{0}+\cdots+\eta_{k-1})(\|\hx_n\|_2+1)+|F_{0}(0)|+\cdots+|F_{k-1}(0)|+1\Bigr)^k
	\end{align}
	and
	\begin{align}\label{prop0:eq3}
	\Theta_k(Y)&:=\bigl((1+\|\hat{Y}_n\|_2)\max(\eta_{0},\ldots,\eta_{k-1})+2\bigr)^k.
	\end{align} 
\end{proposition}

This proposition says that the AMP orbits behave stably subject to a small perturbation to the matrix $X$ and the initialization. From this, we show that the Gaussian AMP is concentrated.

\begin{theorem}\label{concentration}
	 Let $u^{[0]}(X)=u^0.$ For any $k\geq 0$, if $\phi\in C(\mathbb{R}^{k+1})$ is Lipschitz, then
	\begin{align*}
    \lim_{n\to\infty}\e \bigl|\Phi_{k,n}(G)-\tilde{\e}\Phi_{k,n}(G)\bigr|^2=0,
	\end{align*}
	where  $\tilde \e$ is the expectation conditionally on $u^{0}$ and $Z$.
\end{theorem}

For the rest of this section, we establish these results.

\subsection{Proof of Proposition \ref{prop0}}

The proof of this proposition relies on two lemmas on the boundedness and the Lipschitz property of the vector $u^{[k]}(X)$ following an iterative argument.

\begin{lemma}\label{add:lem1}
For every $k\geq 1,$
\begin{align}
\label{add:lem1:eq1}
\|u^{[k]}(X)\|_2&\leq \Delta_k(X)(\|u^{[0]}(X)\|_2+\sqrt{n}),
\end{align}
where $\Delta_k(X)$ is defined in \eqref{prop0:eq2}.
\end{lemma}

\begin{proof} 
	Write
	\begin{align*}
		\|u^{[\ell]}(X)\|_2^2&=\sum_{i=1}^n F_{\ell-1}(\hx_nu^{[\ell-1]}(X),u^{[\ell-2]}(X),\ldots,u^{[0]}(X))_i^2.
	\end{align*}
	Using the Lipschitz property of $F_{\ell-1}$ and the trivial bound $(a+b)^2\leq 4(a^2+b^2)$ yields
	\begin{align*}
	\|u^{[\ell]}(X)\|_2^2&\leq \sum_{i=1}^n\Bigl(\eta_{\ell-1}\Bigl(|\hx_nu^{[r]}(X)_i|^2+\sum_{r=0}^{\ell-2}|u_i^{[r]}(X)|^2\Bigr)^{1/2}+|F_{\ell-1}(0)|\Bigr)^2\\
	&\leq 4\sum_{i=1}^n\Bigl(\eta_{\ell-1}^2|\hx_nu^{[\ell-1]}(X)_i|^2+\eta_{\ell-1}^2\sum_{r=0}^{\ell-2}|u_i^{[r]}(X)|^2+F_{\ell-1}(0)^2\Bigr)\\
	&= 4\eta_{\ell-1}^2 \|\hx_nu^{[\ell-1]}(X)\|_2^2+4\eta_{\ell-1}^2\sum_{r=0}^{\ell-2}\|u^{[r]}(X)\|_2^2+4nF_{\ell-1}(0)^2
	\end{align*}
	so that from the Minkowski inequality,
	\begin{align*}
	\|u^{[\ell]}(X)\|_2&\leq 2\eta_{\ell-1}\|\hat X_n\|_2\|u^{[\ell-1]}(X)\|_2+2\eta_{\ell-1}\sum_{r=0}^{\ell-2}\|u^{[r]}(X)\|_2+2n^{1/2}|F_{\ell-1}(0)|\\
	&\leq C\Bigl(\sum_{r=0}^{\ell-1}\|u^{[r]}(X)\|_2+n^{1/2}\Bigr),
	\end{align*}
	where 
	\begin{align*}
	C:=2(\eta_{0}+\cdots+\eta_{k-1})(\|\hx_n\|_2+1)+2\bigl(|F_{0}(0)|+\cdots+|F_{k-1}(0)|\bigr).
	\end{align*}
	If we let $t_\ell:=\|u^{[\ell]}(X)\|+n^{1/2}$ and $C':=1+C$, then the above inequality implies that $$t_\ell\leq C'\sum_{r=0}^{\ell-1}t_r,\,\,\forall 1\leq \ell\leq k.$$ Using induction yields that
	\begin{align*}
	t_\ell&\leq  C'(1+C')^{\ell-1} t_0,\,\,\forall 1\leq \ell\leq k,
	\end{align*}
	which implies that for $\Delta_k(X):=(1+C')^k=(2+C)^k,$
	\begin{align*}
	\|u^{[k]}(X)\|_2\leq C'(1+C')^{k-1} (\|u^{[0]}(X)\|_2+\sqrt{n})\leq \Delta_k(X)(\|u^{[0]}(X)\|_2+\sqrt{n})
	\end{align*}
	and this completes our proof.
\end{proof}

\begin{lemma}\label{add:lem2}
	For any $X,Y\in M_n(\mathbb{R}),$
	\begin{align*}
	\|u^{[k]}(X)-u^{[k]}(Y)\|_2&\leq  \Theta_k(Y)\Delta_{k}(X)\|X_n-Y_n\|_2(\|u^{[0]}(X)\|_2+\sqrt{n})\\
	&+\Theta_k(Y)\|u^{[0]}(X)-u^{[0]}(Y)\|_2,
	\end{align*}
	where $\Delta_k(X)$ is defined in \eqref{prop0:eq3}.
\end{lemma}

\begin{proof}
    From the Lipschitz property of $F_{\ell-1}$,
	\begin{align*}
	\|u^{[\ell]}(X)-u^{[\ell]}(Y)\|_2&\leq \eta_{\ell-1}\Bigl(\|\hx_nu^{[\ell-1]}(X)-\hat{Y}_nu^{[\ell-1]}(Y)\|_2+\sum_{r=0}^{\ell-2}\|u^{[r]}(X)-u^{[r]}(Y)\|_2\Bigr).
	\end{align*}
	Here, for any $1\leq \ell\leq k,$
	\begin{align*}
	\|\hx_nu^{[\ell-1]}(X)-\hat{Y}_nu^{[\ell-1]}(Y)\|_2&\leq \|u^{[\ell-1]}(X)\|_2\|X_n-Y_n\|_2\\
	&+\|\hat{Y}_n\|_2\|u^{[\ell-1]}(X)-u^{[\ell-1]}(Y)\|_2.
	\end{align*}
	If we let 
	\begin{align*}
	C&=(1+\|\hat{Y}_n\|_2)\max(\eta_{0},\ldots,\eta_{k-1}),\\
	D&=\|X_n-Y_n\|_2\max\bigl(\|u^{[0]}(X)\|_2,\ldots, \|u^{[k-1]}(X)\|_2\bigr),
	\end{align*} 
	then
	\begin{align*}
	&\|u^{[\ell]}(X)-u^{[\ell]}(Y)\|_2\leq C\Bigl(D+\sum_{r=0}^{\ell-1}\|u^{[r]}(X)-u^{[r]}(Y)\|_2\Bigr).
	\end{align*}
	If we let  
	$$
	t_\ell:=D+\|u^{[\ell]}(X)-u^{[\ell]}(Y)\|_2,
	$$
	then for $C':=C+1$, $$
	t_\ell\leq C'\sum_{r=0}^{\ell-1}t_r,\,\,1\leq\ell\leq k
	$$
	and by induction, $t_\ell\leq C'(1+C')^{\ell-1} t_0$. Consequently, for $\Theta_k(Y):=(1+C')^k,$
	\begin{align*}
	\|u^{[k]}(X)-u^{[k]}(Y)\|_2&\leq C'(1+C')^{k-1} t_0\leq \Theta_k(Y)t_0
	\end{align*}
	and this completes our proof by noting that $$
	D\leq \|X_n-Y_n\|_2(\|u^{[0]}(X)\|_2+\sqrt{n})\Delta_{k-1}(X)\leq \|X_n-Y_n\|_2(\|u^{[0]}(X)\|_2+\sqrt{n})\Delta_{k}(X).$$
\end{proof}

\begin{proof}[\bf Proof of Proposition \ref{prop0}]
From the Lipschitz property of $\phi$ and Lemmas \ref{add:lem1} and \ref{add:lem2}, our assertion follows immediately.
\end{proof}

\subsection{Some prior bounds}

\begin{lemma}\label{add:lem5}
For any integer $p\geq 1,$ we have that
\begin{align*}
\sup_{n\geq 1}\Bigl(\frac{\e \|u^0\|_2^p}{n^{p/2}},\frac{\e\|Z\|_2^p}{n^p},\e\|A_n\|_2^p,\e\|\ha_n\|_2^{p}\Bigr)<\infty.
\end{align*}
\end{lemma}

\begin{proof}
    Note that $x^{2p}\leq p!e^{x^2}$ and $|x|^p\leq p!e^{|x|}$. The inequality \eqref{add:eq--1} implies that 
	$$
	\sup_{n\geq 1}\frac{\e \|u^0\|_2^p}{n^{p/2}}\leq \sup_{n\geq 1}\Bigl(\e \frac{\|u^0\|_2^{2p}}{n^p}\Bigr)^{1/2}\leq \sup_{n\geq 1}\sqrt{ p!\sigma^{p}\e e^{\|u^0\|_2^2/\sigma n}}<\sqrt{p!\sigma^{p}C(\sigma)}.
	$$ 
	Next, note that the operator norm is no larger than the Frobenius norm. This and the Jensen inequality lead to
	\begin{align*}
	\sup_{n\geq 1}\frac{\e\|Z\|_2^p}{n^p}&\leq \sup_{n\geq 1}\e\Bigl(\frac{\sum_{i,i'=1}^n|z_{ii'}|^2}{n^{2}}\Bigr)^{p/2}\leq \sup_{n\geq 1}\frac{\e\sum_{i,i'=1}^n|z_{ii'}|^{p}}{n^2}\\
	&\leq \sup_{n\geq 1}\frac{p!\sigma^p}{n^2}\sum_{i,i'=1}^n\e e^{|z_{ii'}|/\sigma}\leq p!\sigma^pC(\sigma).
	\end{align*}
	Finally, since the entries of $A$ are independent $\sigma$-subgaussian with zero mean, it is well-known (see, for instance, Corollary 4.4.8 in \cite{RV18}) that $\sup_{n\geq 1}\e\|A_n\|_2^p<\infty$. Putting these bounds together yields the uniform integrability of $\|\ha_n\|_2^p$ and this completes our proof.
\end{proof}

\begin{lemma}\label{add:lem3}
	For any $k\geq 0$, we have that
	\begin{align}\label{add:lem3:eq1}
	\sup_{n\geq 1}\e\Bigl(\frac{\|u^{[k]}(A)\|_2}{\sqrt{n}}\Bigr)^4<\infty
	\end{align}
	and if $\phi\in C(\mathbb{R}^{k+1})$ is Lipschitz,
	\begin{align}\label{add:lem3:eq2}
	\sup_{n\geq 1}\e|\Phi_{k,n}(A)|^4<\infty.
	\end{align}
\end{lemma}

\begin{proof} The proof follows directly from Lemmas \ref{add:lem1} and \ref{add:lem5}.
\end{proof}

\subsection{Proof of Theorem \ref{concentration}}

First of all, we establish a Gaussian concentration inequality for the functional $\Phi_{k,n}.$

\begin{lemma}\label{add:lem4}
Let $u^{[0]}(X)=u^0.$ For any $k\geq 0$, if $\phi\in C(\mathbb{R}^{k+1})$ is Lipschitz, then there exists a constant $c>0$ such that for every $t>0,$
\begin{align*}
\tilde \p\bigl(\bigl|\Phi_{k,n}(G)-\tilde \e[\Phi_{k,n}(G)]\bigr|\geq t-c(\Omega_n+1)e^{-n/c}\bigr)\leq ce^{-nt^2/(c\,\Omega_n^2)}+ce^{-n/c},
\end{align*}
where  $\tilde \p$ and $\tilde \e$ are the probability and expectation conditionally on $u^{0}$ and $Z$, and 
\begin{align}\label{add:eq---3}
\Omega_n:=\Bigl(1+\frac{\|u^{0}\|_2}{\sqrt{n}}\Bigr)\Bigl(1+\frac{\|Z\|_2}{n}\Bigr)^k.
\end{align}
\end{lemma}	 

\begin{proof}
From $u^{[0]}(X)=u^{[0]}(Y)=u^0$ and $\|\hx_n\|_2\leq \|X_n\|_2+\|Z\|_2/n$, Proposition \ref{prop0} implies
\begin{align*}
|\Phi_{k,n}(X)-\Phi_{k,n}(Y)|&\leq  c_0\Omega_n\|X_n-Y_n\|_2\sum_{\ell=1}^k\Bigl(1+\|X_n\|_2\Bigr)^\ell\Bigl(1+\|Y_n\|_2\Bigr)^\ell,
\end{align*}
where $\Omega_n$ is defined in \eqref{add:eq---3} and $c_0$ is  a constant independent of $n.$ Observe that for any $M>0,$ if $\|X_n\|_2,\|Y_n\|_2\leq M$, then
\begin{align*}
\Phi_{k,n}(X)&\leq \Phi_{k,n}(Y)+c_0\Omega_n(\|X_n-Y_n\|_2\wedge (2M))\sum_{\ell=1}^k(1+\|X_n\|_2\wedge M)^\ell(1+\|Y_n\|_2)^\ell).
\end{align*}
This implies that if 
\begin{align*}
T(X)&:=\inf_{Y\in M_n(\mathbb{R}):\|Y_n\|_2\leq M}\Bigl(\Phi_{k,n}(Y)+c_0\Omega_n(\|X_n-Y_n\|_2\wedge (2M))\\
&\qquad\qquad\qquad\qquad \cdot\sum_{\ell=1}^k(1+\|X_n\|_2\wedge M)^\ell(1+\|Y_n\|_2)^\ell)\Bigr),
\end{align*}
then $T(X)\geq \Phi_{k,n}(X)$ if $\|X_n\|_2\leq M$ and consequently, $T(X)=\Phi_{k,n}(X)$ if $\|X_n\|_2\leq M.$ 
Next, note that for any $Y_n\in M_n(\mathbb{R})$ with $\|Y_n\|_2 \leq M$ and $X,X'\in M_n(\mathbb{R})$, 
\begin{align*}
\|X_n-Y_n\|_2\wedge (2M)&\leq (\|X_n-X_n'\|_2+\|X_n'-Y_n\|_2)\wedge (2M)\\
&\leq \|X_n-X_n'\|_2\wedge (2M)+\|X_n'-Y_n\|_2\wedge (2M)\\
&\leq \|X_n-X_n'\|_2+\|X_n'-Y_n\|_2\wedge (2M)
\end{align*}
and 
\begin{align*}
(1+\|X_n\|_2\wedge M)^\ell
&\leq (1+\|X_n-X_n'\|_2\wedge M+\|X_n'\|\wedge M)^\ell\\
&=(1+\|X_n'\|_2\wedge M)^\ell+\sum_{a=1}^{\ell}{\ell\choose a}\bigl(\|X_n-X_n'\|_2\wedge M\bigr)^a(1+\|X_n'\|_2\wedge M)^{\ell-a}\\
&\leq (1+\|X_n'\|_2\wedge M)^\ell+\sum_{a=1}^{\ell}{\ell\choose a}\|X_n-X_n'\|_2M^{a-1}(1+M)^{\ell-a}\\
&=(1+\|X_n'\|_2\wedge M)^\ell+\|X_n-X_n'\|_2\sum_{a=1}^{\ell}{\ell\choose a}M^{a-1}(1+M)^{\ell-a}.
\end{align*}
From these and noting that the $\ell_2$-operator norm of a matrix is less than its Frobenius norm, we see that
$T(X)$ is Lipschitz with respect to the Frobenius norm with Lipschitz constant
$
c_1\Omega_n/n^{1/2}
$
for some constant $c_1$ independent of $n.$ Hence, the usual Gaussian concentration inequality for Lipschitz functions implies that
\begin{align*}
\tilde \p\bigl(\bigl|T(G)-\tilde \e[T(G)]\bigr|\geq t\bigr)\leq 2e^{-nt^2/(4{c_1}^2\Omega_n^2)},\,\,\forall t>0.
\end{align*}
Now note that as long as we fix $M$ large enough at the beginning, 
\begin{align*}
\tilde \p(|T(G)-\tilde{\e} T(G)|\geq t)&\geq \tilde \p\bigl(|\Phi_{k,n}(G)-\tilde \e T(G)|\geq t,\|G_n\|_2\leq M\bigr)\\
&\geq \tilde \p\bigl(|\Phi_{k,n}(G)-\tilde \e T(G)|\geq t\bigr)-\p\bigl(\|G_n\|_2\geq M\bigr)\\
&\geq \tilde \p\bigl(|\Phi_{k,n}(G)-\tilde \e T(G)|\geq t\bigr)-c_2e^{-n(M-c_3)^2/c_2},
\end{align*}
where the last inequality used the well-known bound that the largest eigenvalue of $G$ is concentrated around its mean with exponential tail bound, which follows by the Borell-TIS inequality and $c_2,c_3$ are two constants independent of $n$ and $M.$ On the other hand, 
\begin{align*}
\tilde \e T(G)&=\tilde \e [\Phi_{k,n}(G);\|G_n\|_2\leq M]+\tilde \e [T(G);\|G_n\|_2\geq M]\\
&=\tilde \e [\Phi_{k,n}(G)]+\tilde \e [-\Phi_{k,n}(G)+T(G);\|G_n\|_2\geq M].
\end{align*}
Here, 
\begin{align*}
\bigl|\tilde \e[-\Phi_{k,n}(G)+T(G);\|G_n\|_2\geq M]\bigr|&\leq \bigl(\tilde \e (|\Phi_{k,n}(G)|+|T(G)|)^2\bigr)^{1/2}\p\bigl(\|G_n\|_2\geq M\bigr)^{1/2}\\
&\leq c_4(\Omega_n+1)e^{-n(M-c_3)^2/2c_2}
\end{align*}
for some $c_4$ independent of $n$ and $M.$ From these,
\begin{align*}
&\tilde \p\bigl(\bigl|\Phi_{k,n}(G)-\tilde \e[\Phi_{k,n}(G)]\bigr|\geq t-c_4(\Omega_n+1)e^{-n(M-c_3)^2/2c_2}\bigr)\\
&\leq 2e^{-nt^2/(4{c_1}^2\Omega_n^2)}+c_2e^{-n(M-c_3)^2/c_2}.
\end{align*}
This completes our proof.
\end{proof}

\begin{proof}
	[\bf Proof of Theorem \ref{concentration}] From Lemmas \ref{add:lem5} and \ref{add:lem4} and the Markov inequality, in probability $\p$,
	\begin{align*}
	\lim_{n\to\infty}\bigl|\Phi_{k,n}(G)-\tilde \e[\Phi_{k,n}(G)]\bigr|=0.
	\end{align*}
	In addition, from
	\begin{align*}
	\bigl(\e\bigl|\Phi_{k,n}(G)-\tilde \e[\Phi_{k,n}(G)]\bigr|^4\bigr)^{1/4}&\leq 	\bigl(\e\Phi_{k,n}(G)^4\bigr)^{1/4}+\bigl(\e \bigl(\tilde \e[\Phi_{k,n}(G)]\bigr)^4\bigr)^{1/4}
    \leq 2\bigl(\e\Phi_{k,n}(G)^4\bigr)^{1/4},
	\end{align*}
	the uniform upper bound \eqref{add:lem3:eq2} gives
	$$
	\sup_{n\geq 1}\bigl(\e\bigl|\Phi_{k,n}(G)-\tilde \e[\Phi_{k,n}(G)]\bigr|^4\bigr)^{1/4}<\infty.
	$$
	Hence, the assertion follows.
\end{proof}

\section{Smooth approximation}
\label{sec:sec3.5}

Recall that the functions $F_k$ in Definition \ref{def1} and $\phi$ in Theorem~\ref{thm0} are Lipschitz. In this section, we show that to prove Theorem \ref{thm0}, it suffices to assume that these functions are smooth and their derivatives of any nonzero orders are uniformly bounded. 

\begin{proposition}\label{add:prop1}
	For any $k\geq 0$ and $\varepsilon>0,$ there exist a constant $C$ independent of $n$ and some functions $\bar\phi\in C^\infty(\mathbb{R}^{k+1})$ and $\bar F_\ell\in C^\infty (\mathbb{R}^{\ell+1})$ for $0\leq \ell\leq {k-1}$, whose partial derivatives of any nonzero orders are uniformly bounded  such that
	\begin{align}\label{add:prop1:eq1}
	\bigl|\Phi_{k,n}(X)-\bar\Phi_{k,n}(X)\bigr|\leq \varepsilon C \sum_{\ell=0}^{k-1}\|\hx_n\|_2^\ell,
	\end{align} 
	where 
	$$
	\bar\Phi_{k,n}(X)=\frac{1}{n}\sum_{i=1}^n\bar\phi\bigl(\bar u_i^{[k]}(X),\bar u_i^{[k-1]}(X),\ldots, \bar u_i^{[0]}(X)\bigr)
	$$	
	and $\bar u^{[k]}$ is the $k$-th AMP orbit in Definition \ref{thm0} associated to the functions $\bar F_0,\ldots,\bar F_{k-1}$ and the initial condition $\bar u^{[0]}(X)=u^{[0]}(X).$
\end{proposition}

\begin{proof}
	Denote by $\eta_\ell$ the Lipschitz constant of $F_\ell$. Let $\varepsilon>0$ be fixed. Assume that $\zeta_\ell\in C^\infty(\mathbb{R}^{\ell+1})$ is a mollifier with $\zeta_\ell \geq 0$ and $\int \zeta_\ell dx=1$ and it is supported on the unit ball $\{x\in\mathbb{R}^{\ell +1}:\|x\|_2\leq 1\}.$ Define $\zeta_{\ell ,\varepsilon}(x)=\varepsilon^{-(\ell +1)}\zeta_\ell(x/\varepsilon).$ Set $$
	\bar F_{\ell ,\varepsilon}(x)=F_\ell *\zeta_{\ell,\varepsilon} (x)=\int \zeta_{\ell,\varepsilon}(x-y)F_\ell(y)dy.
	$$ Note that for any $\varepsilon>0$ and $x\in \mathbb{R}^{\ell +1},$ 
	\begin{align*}
	|\bar F_{\ell}(x)-F_\ell (x)|&=\Bigl|\int \zeta_\ell (z)(F_\ell (x-\varepsilon z)-F_\ell (x))dz\Bigr|\leq \eta_\ell \varepsilon \int \|z\|_2\zeta_\ell (z)dz\leq \eta_\ell' \varepsilon
	\end{align*}
	for some constant $\eta_\ell'>0.$
	In addition, for any index $\alpha=(\alpha_\ell ,\ldots,\alpha_0)\in (\{0\}\cup\mathbb{N})^{\ell +1}$ with $|\alpha|:=\sum_{r=0}^\ell \alpha_r\geq 1$, if $\alpha_{r_0}\geq 1$ for some $0\leq r_0\leq \ell $, then 
	\begin{align*}
	\partial^\alpha \bar F_{\ell}(x)&=\frac{1}{\varepsilon^{\ell +|\alpha|}}\int  \partial^{\alpha'}\zeta_\ell \Bigl(\frac{x-y}{\varepsilon}\Bigr)\partial_{y_{r_0}}F_\ell (y)dy=\varepsilon^{1-|\alpha|}\int  \partial^{\alpha'}\zeta_\ell (z)\partial_{y_{r_0}}F_\ell (x-\varepsilon z)dz,
	\end{align*}
	where
	$$\alpha':=(\alpha_\ell ,\ldots,\alpha_{r_0+1},\alpha_{r_0}-1,\alpha_{r_0-1},\ldots,\alpha_0).$$
	Since $F_\ell $ is Lipschitz and $\zeta_\ell $ is supported on the unit ball, it follows that the partial derivatives of all nonzero orders of $\bar F_\ell$ are uniformly bounded. In particular, $\sup_{x} \|\nabla \bar F_{\ell}(x)\|_2\leq \eta_\ell'',$ independent of $\varepsilon$. Let $\eta=\max_{1\leq j\leq \ell} \{\eta_j',\eta_j''\}$. To show \eqref{add:prop1:eq1}, note that
	\begin{align*}
	\|u^{[\ell+1]}-\bar u^{[\ell+1]}\|_2&\leq \bigl\|F_\ell(\hx_nu^{[\ell]},u^{[\ell-1]},\ldots,u^{[0]})-\bar F_{\ell}(\hx_n\bar u^{[\ell]},\bar u^{[\ell-1]},\ldots,\bar u^{[0]})\bigr\|_2\\
	&\leq \bigl\|F_\ell(\hx_nu^{[\ell]},u^{[\ell-1]},\ldots,u^{[0]})-\bar F_{\ell}(\hx_nu^{[\ell]}, u^{[\ell-1]},\ldots, u^{[0]})\bigr\|_2\\
	&+\bigl\|\bar F_{\ell}(\hx_nu^{[\ell]}, u^{[\ell-1]},\ldots, u^{[0]})-\bar F_{\ell}(\hx_n\bar u^{[\ell]},\bar u^{[\ell-1]},\ldots,\bar u^{[0]})\bigr\|_2\\
	&\leq \eta\varepsilon+\eta\Bigl(\|\hx_n\|_2\|u^{[\ell]}-\bar u^{[\ell]}\|_2+\sum_{r=0}^{\ell-1}\|u^{[r]}-\bar u^{[r]}\|\Bigr).
	\end{align*} 
	Since $\bar u^{[0]}=u^{[0]}$, an induction argument implies that
	\begin{align}\label{add:eq0}
	\|u^{[\ell]}(X)-\bar u^{[\ell]}(X)\|_2&\leq \varepsilon C \sum_{j=0}^{\ell-1}\|\hx_n\|_2^j,
	\end{align}
	where $C$ is a constant depending only on $\ell$ and $\eta$. Finally, by the same argument, for any $\varepsilon>0,$ there exists a $\bar\phi \in C^{\infty}(\mathbb{R}^{k+1})$ with uniformly bounded partial derivatives of any nonzero orders such that $\|\phi-\bar\phi\|_\infty<\varepsilon.$ From \eqref{add:eq0} and the Lipschitz property of $ \phi$, our proof is completed.
\end{proof}

\section{Bounding the derivatives}\label{sec:sec4}

We establish uniform moment controls on the partial derivatives of the generalized AMP orbit in Definition \ref{def1}. For $\sigma>0$, recall the random vector $u^0$ and the random matrix $Z$ from \eqref{add:eq--1}. Let $\mathcal{G}_n(\sigma)$ be the collection of $n\times n$ symmetric random matrices $A=(a_{ii'})_{i,i'\in[n]}$, whose entries are centered independent and each of them are  $\sigma'$-subgaussian for some $0\leq \sigma'\leq \sigma$. We also assume that $\mathcal{G}_n(\sigma)$ is independent of $u^0$ and $Z.$ 

For any $m\in \mathbb{N},$ denote by $[m]=\{1,\ldots,m\}$. For any $n\geq 2,$ let $\mathcal{T}_n$ be the collection of all sequences $(i_r,i_r')_{r\geq 1}\subset [n]^2$ with $i_r< i_r'$ for all $r\geq 1.$ Let $P$ be an arbitrary finite subset of $\mathbb{N}$  and let $m=|P|.$ For $h\in C^m(M_n(\mathbb{R}))$ and $(i_r,i_r')_{r\geq 1}\in \mathcal{T}_n$, denote by 
$$\partial_{P}h(X)\in \mathbb{R}$$
the partial derivative of $h$ with respect to the variables $x_{i_ri_r'}$ for all $r\in P$ counting multiplicities. For a vector-valued function $H=(h_1,\ldots,h_n)$ for $h_1,\ldots,h_n\in C^m(M_n(\mathbb{R}))$, we also set the partial derivative of $H$ by $$
\partial_PH(X)=(\partial_Ph_1(X),\ldots,\partial_Ph_n(X))\in \mathbb{R}^n
$$
and denote
$$
\partial_PH(X)_i=\partial_Ph_i(X),\,\,1\leq i\leq n.
$$
For any $n\geq 2,$ $k\geq 0$, $p\geq 1$, and $m\geq 0,$  denote by $\mathcal{B}_n(k,p,m)$ the collection of all $$
\bigl(P,(i_r,i_r')_{r\geq 1},A,i\bigr)
$$
for $P\subset \mathbb{N}$ with $|P|=m$, $(i_r,i_r')\in \mathcal{T}_n,$ $A\in \mathcal{G}_n(\sigma),$ and $i\in [n].$
The following is our main estimate.

\begin{proposition}\label{prop1}
	 Consider the AMP orbit $(u^{[k]}(X))_{k\geq 0}$ in Definition \ref{def1} with $u^{[0]}(X)=u^0$. Assume that the functions $F_k$ in Definition \ref{def1} satisfy the following assumption:
	  \begin{align}\label{ass}
	 \mbox{$F_k\in C^\infty(\mathbb{R}^{k+1})$ and its partial deriatives of all nonzero orders are uniformly bounded.}
	 \end{align}
	 Let $k\geq 0$, $p\geq 1,$ and $m\geq 0$. Let $U\in C^{\infty}(\mathbb{R}^{2(k+1)})$. Assume that its partial derivatives of nonzero orders are uniformly bounded.
	  Define a vector-valued random function on $M_n(\mathbb{R})$ by
	$$U(X)=U(\hat X_nu^{[k]}(X),\ldots,\hat X_nu^{[0]}(X),u^{[k]}(X),\ldots, u^{[0]}(X))\in \mathbb{R}^n.$$
    There exists a universal constant $\Gamma_{k,p,m}^{U}$ such that 
	\begin{align*}
	\sup_{\mathcal{B}_n(k,p,m)}\bigl(\e\bigl|\partial_{P}U(A)_i\bigr|^p\bigr)^{1/p}&\leq \frac{\Gamma_{k,p,m}^{U}}{n^{m/2}},\,\,\forall n\geq 2.
	\end{align*}
\end{proposition}

As we shall see, this bound will be used to control the Gaussian interpolation between the first two moments of $\Phi_{k,n}(A)$ and $\Phi_{k,n}(G).$ 
For the rest of this section, we establish this proposition in three subsections. First of all, we derive explicit formulas for the derivatives of the AMP orbit. Next, we show that Proposition \ref{prop1} is valid if $U$ depends only on the marginal variables. The general case is treated in the last subsection.

\subsection{Some auxiliary lemmas}

Let $k  \geq 0$ and $m\geq 1$ be fixed. Let $v_0,\ldots,v_k\in C^{m}(M_n(\mathbb{R}))$ and $F\in C^m (\mathbb{R}^{k  +1}).$ Set $$
V(X)=F(v_k  (X),\ldots,v_{0}(X))\in \mathbb{R},\,\,\forall X\in M_n(\mathbb{R}).
$$
Let $(i_r,i_r')_{r\geq 1}\in \mathcal{T}_n.$ Let $P$ be a finite subset of $\mathbb{N}$ with $|P|=m.$ For any $1\leq r\leq m,$ set $\mathcal{J}_r(k  )=\{0,\ldots,k  \}^r$ and set $\mathcal{P}_{r}(P) $ the collection of all partitions $\mathcal{P}=\{P_1,\ldots,P_r\}$ of $P$ into $r$ nonempty subsets. For $J=(j_1,\ldots,j_r)\in \mathcal{J}_r(k  )$, set
\begin{align*}
\partial_{J}F(y_k  ,\ldots,y_0)=\partial_{y_{j_r}\cdots y_{j_1}}F(y_k  ,\ldots,y_0).
\end{align*}

\begin{lemma}\label{lem1}
	We have that
	\begin{align*}
	\partial_{P}V&=\sum_{1\leq r\leq m}\sum_{J \in \mathcal{J}_r(k  )}\sum_{\mathcal{P}\in \mathcal{P}_{r}(P) }\partial_J F(v_k  ,\ldots,v_0)\partial_{P_{1}}v_{j_{1}}\cdots \partial_{P_{r}}v_{j_ {r}}.
	\end{align*}
\end{lemma}

\begin{proof}
	We argue by induction on the size of the set $P$. The case $|P| = 1$ is obvious. Suppose that the conclusion holds for some $m\geq 1$ and all $P\subset \mathbb{N}$ with $|P| = m$. Without loss of generality, it suffices to show that the conclusion holds for $P = [m+1]$. From induction hypothesis, we compute directly to get
	\begin{align*}
	\partial_{[m+1]}V =\partial_{\{m+1\}}\bigl(\partial_{[m]}V\bigr)&=\sum_{1\leq r\leq m}\sum_{J \in \mathcal{J }_r(k)}\sum_{\mathcal{P}\in \mathcal{P}_{r}([m])}\Bigl(\partial_JF \partial_{\{m+1\}}\bigl(\partial_{P_1}v_{j_ 1} \cdots \partial_{P_r}v_{j_ r} \bigr)\\
	&\qquad\qquad\qquad+\Bigl(\sum_{j=0}^{k  }\partial_{j,J}F \partial_{\{m+1\}}v_j \Bigr)\partial_{P_1}v_{j_ 1} \cdots \partial_{P_r}v_{j_ r} \Bigr).
	\end{align*}
	To handle the first summation, note that
	\begin{align*}
	&\partial_{\{m+1\}}\bigl(\partial_{P_1}v_{j_ 1}\cdots \partial_{P_r}v_{j_ r}\bigr)
	=\sum_{s=1}^r(\partial_{P_1}v_{j_ 1})\cdots(\partial_{P_{s-1}}v_{j_ {s-1}})( \partial_{\{m+1\}\cup P_s}v_{j_ s})(\partial_{P_{s+1}}v_{j_ {s+1}})\cdots(\partial_{P_r}v_{j_ r}),
	\end{align*}
	which implies that
	\begin{align}
	\begin{split}\label{eq2}
	&\sum_{1\leq r\leq m}\sum_{J \in \mathcal{J }_r(k)}\sum_{\mathcal{P}\in \mathcal{P}_{r}([m])} \partial_J F\partial_{\{m+1\}}\bigl(\partial_{P_1}v_{j_ 1}\cdots \partial_{P_r}v_{j_ r}\bigr)\\
	&=\sum_{1\leq r\leq m}\sum_{J \in \mathcal{J }_r(k)}\sum_{\mathcal{P}\in \mathcal{P}_r([m+1]):\{m+1\}\notin \mathcal{P}}\partial_JF\bigl(\partial_{P_1}v_{j_ 1}\cdots \partial_{P_r}v_{j_ r}\bigr).
	\end{split}
	\end{align}
	On the other hand, since
	\begin{align*}
	\Bigl(\sum_{j=0}^{k  }\partial_{j,J}F\partial_{\{m+1\}}v_j\Bigr)\partial_{P_1}v_{j_ 1}\cdots \partial_{P_r}v_{j_ r}&=\sum_{j_ {r+1}=0}^{k  }\partial_{j_{r+1},J}F\partial_{P_1}v_{j_ 1}\cdots \partial_{P_r}v_{j_ {r}} \partial_{\{m+1\}}v_{j_ {r+1}},
	\end{align*}
	it follows that
	\begin{align*}
	&\sum_{1\leq r\leq m}\sum_{J \in \mathcal{J }_r(k)}\sum_{\mathcal{P}\in \mathcal{P}_{r}([m])}\Bigl(\sum_{j=0}^{k  }\partial_{j,J}F\partial_{\{m+1\}}v_j\Bigr)\partial_{P_1}v_{j_ 1}\cdots \partial_{P_r}v_{j_ r}\\
	&=\sum_{1\leq r\leq m}\sum_{J \in \mathcal{J }_{r+1}(k)}\sum_{\mathcal{P}\in \mathcal{P}_{r+1}([m+1]):\{m+1\}\in \mathcal{P}}\partial_{J }F\partial_{P_1}v_{j_ 1}\cdots \partial_{P_{r+1}}v_{j_ {r+1}}.
	\end{align*}
	To simplify this summation, we write
	\begin{align*}
	&\sum_{1\leq r\leq m}\sum_{J \in \mathcal{J }_{r+1}(k)}\sum_{\mathcal{P}\in \mathcal{P}_{r+1}([m+1]):\{m+1\}\in \mathcal{P}}\partial_{J }F\partial_{P_1}v_{j_ 1}\cdots \partial_{P_{r+1}}v_{j_ {r+1}}\\
	&=\sum_{2\leq r\leq m+1}\sum_{J \in \mathcal{J }_{r}(k)}\sum_{\mathcal{P}\in \mathcal{P}_r([m+1]):\{m+1\}\in \mathcal{P}}\partial_{J }F\partial_{P_1}v_{j_ 1}\cdots \partial_{P_{r}}v_{j_ {r}}\\
	&=\sum_{1\leq r\leq m}\sum_{J \in \mathcal{J }_{r}(k)}\sum_{\mathcal{P}\in \mathcal{P}_r([m+1]):\{m+1\}\in \mathcal{P}}\partial_{J }F\partial_{P_1}v_{j_ 1}\cdots \partial_{P_{r}}v_{j_ {r}}\\
	&+\sum_{J \in \mathcal{J }_{m+1}(k)}\sum_{\mathcal{P}\in \mathcal{P}_{m+1}([m+1])}\partial_{J }F\partial_{P_1}v_{j_ 1}\cdots \partial_{P_{m+1}}v_{j_ {m+1}},
	\end{align*}
	where in the first equality we changed the variable $r+1\to r$, while in the second equality we divide $2\leq r\leq m+1$ into $2\leq r\leq m$ and $r=m+1$ and use the observation that $\mathcal{P}_1([m+1])$ contains no element $\mathcal{P}$ so that $\{m+1\}\in \mathcal{P}.$ Combining this summation with \eqref{eq2} yields the desired formula.
\end{proof}

\begin{lemma}\label{lem2}
	For any $H=(h_1,\ldots,h_n)$ for $h_1,\ldots,h_n\in C^{m}(M_n(\mathbb{R})),$ we have that
	\begin{align*}
\partial_{P}\bigl(\hx_nH(X)\bigr)&=\frac{1}{\sqrt{n}}\sum_{r\in P}E_r\partial_{P\setminus\{r\}}H(X)+\hx_n\partial_{P}H(X),
\end{align*}
where $E_r\in M_n(\mathbb{R})$, whose entries are equal to $1$ at $(i_r,i_r')$ and $(i_r',i_r)$ and are zero otherwise.
	
\end{lemma}

\begin{proof}
  It suffices to assume that $P=[m]$. If $m=1,$ then $\partial_{\{1\}}\bigl(\hx_nH(X)\bigr)=n^{-1/2}E_1H(X)+\hx_n\partial_{\{1\}}H(X).$ Assume that the assertion is valid for $m\geq 1.$ Then
	\begin{align*}
\partial_{[m+1]}\bigl(\hx_nH(X)\bigr)
&=\partial_{\{m+1\}}\Bigl(\frac{1}{\sqrt{n}}\sum_{r\in [m]}E_r\partial_{[m]\setminus\{r\}}H(X)+\hx_n\partial_{[m]}H(X)\Bigr)\\
&=\frac{1}{\sqrt{n}}\sum_{r\in [m]}E_r\partial_{[m+1]\setminus\{r\}}H(X)+\partial_{\{m+1\}}\bigl(\hx_n\partial_{[m]}H(X)\bigr)\\
&=\frac{1}{\sqrt{n}}\sum_{r\in [m+1]}E_r\partial_{[m+1]\setminus\{r\}}H(X)+\hx_n\partial_{[m+1]}H(X).
\end{align*}
\end{proof}

\subsection{Moment control}

The most crucial ingredient of this paper lies on the following proposition, which establishes two special cases of Proposition \ref{prop1}.

\begin{proposition}\label{prop2}
		Consider the AMP orbits in Definition \ref{def1} with the initialization $u^{[0]}(X) = u^{0}$ and assume that $(F_k)_{k\geq 0}$ satisfies \eqref{ass}. For any $k\geq 0$, $p\geq 1,$ and $m\geq 0,$ there exist constants $\Gamma_{k,p,m}$ and $\Gamma_{k,p,m}'$ such that for any $n\geq 2$,
	\begin{align}\label{prop2:eq1}
	\sup_{\mathcal{B}_n(k,p,m)}\bigl(\e \bigl|\partial_{P}u^{[k]}(A)_i\bigr|^p\bigr)^{1/p}&\leq \frac{\Gamma_{k,p,m}}{n^{m/2}}
	\end{align}
	and
	\begin{align}\label{prop2:eq2}
	\sup_{\mathcal{B}_n(k,p,m)}\bigl(\e \bigl|\partial_{P}\bigl(\ha_nu^{[k]}(A)\bigr)_i\bigr|^p\bigr)^{1/p}&\leq \frac{\Gamma_{k,p,m}'}{n^{m/2}}.
	\end{align}
\end{proposition}

\begin{proof} We argue by induction on $k\geq 0.$ First assume that $k=0.$ We aim to show that \eqref{prop2:eq1} and \eqref{prop2:eq2} are valid for all $p\geq 1$ and $m\geq 0.$  For any $P$ with $|P|=m$, since $\partial_Pu^{[0]}(X)=u^0$ if $m=0$ and $\partial_Pu^{[0]}(X)=0$ if $m\geq 1$, \eqref{add:eq--1} obviously implies \eqref{prop2:eq1}. To show \eqref{prop2:eq2}, note that for any $P$ with $|P|=m,$
	\begin{align*}
	\partial_P(\hx_nu^{[0]}(X))=\left\{
	\begin{array}{ll}
	\frac{1}{\sqrt{n}}\hx_nu^0,&\mbox{if $m=0$},\\
	\frac{1}{\sqrt{n}} E_ru^0,&\mbox{if $m=1$ and $P=\{r\}$ for some $r\geq 1$},\\
	0&\mbox{if $m\geq 2$},
	\end{array}\right.
	\end{align*}
	where $E_{r}\in M_n(\mathbb{R})$ is equal to $1$ on the entries $(i_r,i_r')$ and $(i_r',i_r)$ and is zero elsewhere. To control the first case, note that $u^0$ is independent of $A$ and the entries in $A$ are independent. From the subgaussianity of $a_{ij}$ and \eqref{add:eq--1}, there exist positive constants $\lambda(\sigma)$ and $D(\sigma)$ such that for any $n\geq 1$ and $\lambda\in[-\lambda(\sigma),\lambda(\sigma)]$,
	\begin{align*}
	\e e^{\lambda n^{-1/2}\sum_{j=1}^na_{ij}u_j^0}&\leq \e e^{\lambda^2 \sigma^2\|u^0\|^2/2n}\leq \e e^{\lambda(\sigma)^2 \sigma^2\|u^0\|^2/2n}\leq D(\sigma).
	\end{align*}
	Consequently, from $x^{2p}\leq (2p)!\cosh x$ and the Jensen inequality,
	\begin{align*}
	\e\Bigl|\frac{\sum_{j=1}^na_{ij}u_j^0}{\sqrt{n}}\Bigr|^p&\leq \frac{\sqrt{(2p)!D(\sigma)}}{\lambda(\sigma)^p}.
	\end{align*} 
	On the other hand, the Cauchy-Schwarz and Jensen inequalities imply that
	\begin{align*}
	\e \Bigl| \sum_{j=1}^n  \frac{z_{ij}}{n} u_j^0 \Bigr|^p &\leq \e \Bigl| \frac{\sum_{j=1}^n z_{ij}^2}{n}\Bigr|^{p/2} \Bigl(\frac{\|u^0\|_2}{\sqrt{n}}\Bigr)^p\\
	& \leq \Bigl(\e \Bigl| \frac{\sum_{j=1}^n z_{ij}^2}{n}\Bigr|^{p}\Bigr)^{1/2}\Bigl(\e \Bigl(\frac{\|u_0\|_2^2}{n}\Bigr)^{p}\Bigr)^{1/2}\\
	&\leq \Bigl(\e \Bigl| \frac{\sum_{j=1}^n z_{ij}^{2p}}{n}\Bigr|\Bigr)^{1/2}\Bigl(\e \Bigl(\frac{\|u_0\|_2^2}{n}\Bigr)^{p}\Bigr)^{1/2}\leq D'(\sigma),
	\end{align*}
	where $D'(\sigma)$ is a constant independent of $n$ and is guaranteed by the moment assumption \eqref{add:eq--1}. Combining these two inequalities validates \eqref{prop2:eq2}  for $(k,p,m)=(0,p,0)$. In the second case, since $E_ru^0$ has only two nonzero entries and they are  $u_{i_r}^0$ and $u_{i_r'}^0$, the bound \eqref{add:eq--1} implies \eqref{prop2:eq2} for $(k,p,m)=(0,p,1)$. The third case is evident. In conclusion, \eqref{prop2:eq2} holds for $k=0$, $p\geq 1$, and $m\geq 0.$

Next, we assume that there exists some $k_0\geq 0$ such that \eqref{prop2:eq1} and \eqref{prop2:eq2} are valid for all $0\leq k\leq k_0$, $p\geq 1,$ and $m\geq 0.$ Our goal is to show that they are also valid for $k=k_0+1,$ $p\geq 1,$ and $m\geq 0.$ Denote by 
$$
v^{[k_0]}(X)=\hx_nu^{[k_0]}(X),v^{[k_0-1]}(X)=u^{[k_0-1]}(X),\ldots, v^{[0]}(X)=u^{[0]}(X).
$$
First we verify \eqref{prop2:eq1}. Assume that $m=0$. From the Lipschitz property of $F_{k_0}$, 
\begin{align*}
\bigl(\e|u^{[k_0+1]}(A)_i|^p \bigr)^{1/p}&= \bigl(\e|F_{k_0}(v^{[k_0]}(A),\ldots, v^{[0]}(A))_i|^p\bigr)^{1/p}\\
&\leq   \eta_{k_0}\sum_{\ell=0}^{k_0}\bigl(\e|v^{[\ell]}(A)_i|^p\bigr)^{1/p}+ |F_{k_0}(0)|,
\end{align*}
where  $\eta_{k_0}$ is the Lipschitz constant of $F_{k_0}.$
By induction hypothesis, \eqref{prop2:eq1} follows when $(k,p,m)=(k_0+1,p,0)$ for all $p\geq 1.$ Now suppose that $m\geq 1$. For $n\geq 2,$ consider an arbitrary $P$ with $|P|=m$, $(i_r,i_r')_{r\geq 1}\in \mathcal{T}_n$, $A\in \mathcal{G}_n(\sigma),$ and $i\in [n].$ From Lemma \ref{lem1}, 
\begin{align*}
\partial_{P}u^{[k_0+1]}(X)_i&=\sum_{r\in P}\sum_{J\in \mathcal{J}_r(k_0),\mathcal{P}\in \mathcal{P}_r(P)}\partial_JF_{k_0}(v_i^{[k_0]}(X),\ldots,v_i^{[0]}(X))\\
&\qquad\qquad\qquad\qquad\qquad\partial_{P_1}v_i^{[j_1]}(X)\cdots \partial_{P_r}v_i^{[j_r]}(X).
\end{align*}
From this, Minkowski's inequality, H\"older's inequality, and the boundedness of the partial derivatives of $F_{k_0}$, there exist constants $\eta_{k_0,J}$'s such that
\begin{align*}
\bigl(\e|\partial_{P}u^{[k_0+1]}(A)_i|^p\bigr)^{1/p}&\leq \sum_{1\leq r\leq m}\sum_{J\in \mathcal{J}_r(k_0),\mathcal{P}\in \mathcal{P}_r(P)}\eta_{k_0,J}\prod_{l=1}^r\bigl(\e|\partial_{P_l}v^{[j_l]}(A)_i|^{rp}\bigr)^{1/rp}.
\end{align*}
Since each $v^{[j]}(A)$ is either $\ha_nu^{[j]}(A)$ or $u^{[j]}(A)$, this implies that
\begin{align*}
\bigl(\e|\partial_{P}u^{[k_0+1]}(A)_i|^p\bigr)^{1/p}&\leq \sum_{r\in P}\sum_{J\in \mathcal{J}_r(k_0),\mathcal{P}\in \mathcal{P}_r(P)}\prod_{l=1}^r\frac{\max(\Gamma_{k_0,rp,|P_l|},\Gamma_{k_0,rp,|P_l|}')}{n^{|P_l|/2}}\\
&=\frac{1}{n^{m/2}}\sum_{r\in P}\sum_{J\in \mathcal{J}_r(k_0),\mathcal{P}\in \mathcal{P}_r(P)}\prod_{l=1}^r\max(\Gamma_{k_0,rp,|P_l|},\Gamma_{k_0,rp,|P_l|}')\\
&=:\frac{1}{n^{m/2}}\Gamma_{k_0+1,p,m}.
\end{align*}
Hence, \eqref{prop2:eq1} is valid for $k=k_0+1$, $p\geq 1$, and $m\geq 1.$ Putting these two cases together yields the validity of \eqref{prop2:eq1} for $k=k_0+1$, $p\geq 1$, and $m\geq 0.$

Now we verify \eqref{prop2:eq2}. Let $P\subset \mathbb{N}$ with $|P|=m$, $(i_r,i_r')_{r\geq 1}\in \mathcal{T}_n,$ $A\in \mathcal{G}_n(\sigma),$ and $i\in [n]$. Note that from Lemma \ref{lem2},
\begin{align}\label{add:eq--2}
\partial_{P}\bigl(\hx_nu^{[k_0+1]}(X)\bigr)&=\frac{1}{\sqrt{n}}\sum_{r\in P}E_r\partial_{P\setminus\{r\}}u^{[k_0+1]}(X)+\hx_n\partial_Pu^{[k_0+1]}(X).
\end{align}
The first term can be controlled by 
\begin{align}\begin{split}\notag
&\Bigl(\e \Bigl|\frac{1}{\sqrt{n}}\sum_{r\in P}E_r\partial_{P\setminus\{r\}}u^{[k_0+1]}(A)_i\Bigr|^p\Bigr)^{1/p}\\
&\leq \frac{1}{\sqrt{n}}\sum_{r\in P}\bigl(\e\bigl|\delta_{i,i_r}\partial_{P\setminus\{r\}}u^{[k_0+1]}(A)_{i_r'}+\delta_{i,i_{r}'}\partial_{P\setminus\{r\}}u^{[k_0+1]}(A)_{i_r}\bigr|^p \bigr)^{1/p}\\
&\leq \frac{1}{\sqrt{n}}\sum_{r\in P}\Bigl(\bigl(\e|\partial_{P\setminus\{r\}}u^{[k_0+1]}(A)_{i_r'}|^p\bigr)^{1/p}+\bigl(\e|\partial_{P\setminus\{r\}}u^{[k_0+1]}(A)_{i_r}|^p\bigr)^{1/p}\Bigr)
\end{split}\\
\begin{split}\label{add:eq--3}
&\leq \frac{2m\Gamma_{k_0+1,p,m-1}}{n^{m/2}},
\end{split}
\end{align}
where $\delta_{i,i'}=1$ if $i=i'$ and it is zero if $i\neq i'.$
As for the second term, note that for any $u\in \mathbb{R}^n,$
\begin{align*}
|(\hx_nu)_i|&\leq |(X_nu)_i|+\frac{1}{n}\Bigl(\sum_{j=1}^n z_{ij}^2\Bigr)^{1/2}\|u\|_2.
\end{align*}
This implies that
\begin{align*}
&\bigl(\e\bigl|\bigl(\ha_n\partial_Pu^{[k_0+1]}(A)\bigr)_i\bigr|^p\bigr)^{1/p}\\
&\leq \bigl(\e\bigl|\bigl(A_n\partial_Pu^{[k_0+1]}(A)\bigr)_i\bigr|^p\bigr)^{1/p}+\frac{1}{n}\Bigl(\e\Bigl[\Bigl(\sum_{j=1}^n z_{ij}^2\Bigr)^{p/2}\bigl\|\partial_Pu^{[k_0+1]}(A)\bigr\|_2^p\Bigr]\Bigr)^{1/p}\\
&\leq \bigl(\e\bigl|\bigl(A_n\partial_Pu^{[k_0+1]}(A)\bigr)_i\bigr|^p\bigr)^{1/p}
+\Bigl(\e\frac{\bigl(\sum_{j=1}^n z_{ij}^2\bigr)^{p}}{n^p}\Bigr)^{1/2p}\Bigl(\e \frac{\|\partial_Pu^{[k_0+1]}(A)\bigr\|_2^{2p}}{n^{p}}\Bigr)^{1/2p}.
\end{align*}
Here, from \eqref{add:eq--1},
$
n^{-p}\e\bigl(\sum_{j=1}^n z_{ij}^2\bigr)^{p}
$
is bounded above by a constant independent of $n$. From the validity of \eqref{prop2:eq1} for $k=k_0+1,$ $p\geq 1$, and $m\geq 0$ that we established above, we also have that
\begin{align*}
\Bigl(\e \frac{\|\partial_Pu^{[k_0+1]}(A)\bigr\|_2^{2p}}{n^{p}}\Bigr)^{1/2p}&\leq \Bigl(\frac{\e \sum_{i=1}^n|\partial_Pu^{[k_0+1]}(A)_i|^{2p}}{n}\Bigr)^{1/2p}\leq \frac{\Gamma_{k_0+1,2p,m}}{n^{m/2}}
\end{align*}
and from Lemma~\ref{lem-1} below,
\begin{align*}
\bigl(\e\bigl|\bigl(A_n\partial_Pu^{[k_0+1]}(A)\bigr)_i\bigr|^p\bigr)^{1/p}&\leq \frac{\Upsilon_{k_0+1,p,m}}{n^{m/2}},
\end{align*}
where $\Upsilon_{k_0+1,p,m}$ is a universal constant independent of $n.$ Therefore, we arrive at
\begin{align*}
\bigl(\e\bigl|\bigl(\ha_n\partial_Pu^{[k_0+1]}(A)\bigr)_i\bigr|^p\bigr)^{1/p}&\leq \frac{C}{n^{m/2}},
\end{align*}
where $C$ is a constant independent of $n.$ Plugging this inequality and \eqref{add:eq--3} into \eqref{add:eq--2} yields \eqref{prop2:eq2} for $k=k_0+1,$ $p\geq 1,$ and $m\geq 0.$ This completes our proof.
\end{proof}

At the end of this subsection, we establish the following lemma used in the above proof.

\begin{lemma}\label{lem-1}
	Let $k\geq 1$. Assume that for any $p\geq 1$ and $m\geq 0$, there exists a constant $\Gamma_{k,p,m}$ such that 
	\begin{align}\label{lem1:eq1}
	\sup_{\mathcal{B}_n(k,p,m)}\bigl(\e \bigl|\partial_{P}u^{[k]}(A)_i\bigr|^p\bigr)^{1/p}&\leq \frac{\Gamma_{k,p,m}}{n^{m/2}},\,\,\forall n\geq 2.
	\end{align}
	Then for any $p\geq 1$ and $m\geq 0$, there exists a constant $\Upsilon_{k,p,m}$ such that
	\begin{align*}
	\sup_{\mathcal{B}_n(k,p,m)}\bigl(\e \bigl|\bigl(A_n\partial_{P}u^{[k]}(A)\bigr)_i\bigr|^p\bigr)^{1/p}&\leq \frac{\Upsilon_{k,p,m}}{n^{m/2}},\,\,\forall n\geq 2.
	\end{align*}
	
\end{lemma}
\begin{proof} Our idea is to use the Taylor expansion to track the dependence of 
	$$
	\bigl(A_n\partial_{P}u^{[k]}(A)\bigr)^p
	$$
	on each variable $a_{ii'}$ in each iteration.
	By Jensen's inequality, it suffices to assume that $p$ is even. Let $m\geq 0$ be fixed. Let $P\subset \mathbb{N}$ with $|P|=m$, $(i_r,i_r')_{r\geq 1}\in \mathcal{T}_n,$ $A\in \mathcal{G}_n(\sigma),$ and $i\in [n]$. Denote $$
	V(X)=\partial_{P}u^{[k]}(X).
	$$
	For any $D\subseteq [p]$, let $\mathcal{I}_D$ be the collection of all $I=(\iota_1,\ldots,\iota_p)\in [n]^p$ such that $\iota_s$  are distinct for $s\in D$ and 
	\begin{align}\label{lem:proof:eq1}
	\{\iota_s:s\in D^c\}\subseteq \{\iota_s:s\in D\}.
	\end{align} 
	For $I\in \mathcal{I}_D$ and $X\in M_n(\mathbb{R})$, let $X^I\in M_n(\mathbb{R})$, in which each entry of $X^I$ is equal to that of $X$ except that it vanishes on the sites $(i,\iota_s)$ and $(\iota_s,i)$ for all $s\in D.$ For $\iota\in [n]$ and $0\leq t\leq 1$, define  
	\begin{align*}
	f_\iota(t)&=V_\iota(X^I(t))
	\end{align*}
	for $$X^I(t):=tX+(1-t)X^I.$$
	Here, $V_\iota$ is the $\iota$-th entry of $V$. Write by Taylor's theorem,
	\begin{align*}
	V_\iota(X)&=f_\iota(1)=\sum_{a=0}^{p-1}\frac{f_\iota^{(a)}(0)}{a!}+\frac{1}{(p-1)!}\int_0^1(1-t)^{p-1}f_\iota^{(p)}(t)dt=:M_\iota^I(X)+N_\iota^I(X).
	\end{align*}
	Note that here
	\begin{align*}
	f_\iota^{(a)}(0)&=\sum_{s_1,\ldots,s_a\in D}\frac{\partial^aV_\iota(X^I)}{\partial{x_{i\iota_{s_a}}}\cdots \partial {x_{i\iota_{s_1}}}}x_{i\iota_{s_1}}\cdots x_{i\iota_{s_a}}.
	\end{align*}
	Also, note that $\mathcal{I}_D\cap \mathcal{I}_{D'}=\emptyset$ for distinct $D$ and $D'$ and that $[n]^p=\cup_{D\subseteq [p]}\mathcal{I}_D.$ Write
	\begin{align*}
	\e \bigl(A_nV(A)\bigr)_i^p&=\frac{1}{n^{p/2}}\sum_{D\subseteq [p]}\sum_{I\in \mathcal{I}_D}\e a_{i\iota_1}\cdots a_{i\iota_p}V_{\iota_1}(A)\cdots V_{\iota_p}(A)\\
	&=\frac{1}{n^{p/2}}\sum_{D\subseteq [p]}\sum_{I\in \mathcal{I}_D}\sum_{S\subseteq [p]}\e\Bigl(\prod_{l\in [p]}a_{i\iota_l}\Bigr)\Bigl(\prod_{l\in S}M_{\iota_l}^I(A)\Bigr)\Bigl(\prod_{l\in S^c}N_{\iota_l}^I(A)\Bigr) \\
	&=\frac{1}{n^{p/2}}\sum_{D\subseteq [p]}\sum_{I\in \mathcal{I}_D} \sum_{S\subsetneq [p]}\e\Bigl(\prod_{l\in [p]}a_{i\iota_l}\Bigr)\Bigl(\prod_{l\in S}M_{\iota_l}^I(A)\Bigr)\Bigl(\prod_{l\in S^c}N_{\iota_l}^I(A)\Bigr)\\
	&+\frac{1}{n^{p/2}}\sum_{D\subseteq [p]}\sum_{I\in \mathcal{I}_D} \e\Bigl(\prod_{l\in [p]}a_{i\iota_l}\Bigr)\Bigl(\prod_{l\in [p]}M_{\iota_l}^I(A)\Bigr)\\
	&=:\Delta_{n,1}+\Delta_{n,2}.
	\end{align*}
	
	To control these two terms, note that from $x^{2p}\leq (2p)!\cosh x$ and the $\sigma$-subgaussianity, we have the bound
	\begin{align*}
	\sup_{i,j\in [n]}\bigl(\e|a_{ij}|^{p}\bigr)^{1/p}\leq \sup_{i,j\in [n]}\bigl(\e|a_{ij}|^{2p}\bigr)^{1/2p}\leq \xi_p:=(2p)!e^{\sigma^2/2},\,\,\forall p\geq 1.
	\end{align*}
    First we handle $\Delta_{n,1}$. From the given assumption,
	\begin{align*}
	&\Bigl(\e\Bigl|\frac{\partial^aV_{\iota_l}(A^I)}{\partial{x_{i\iota_{s_a}}}\cdots \partial {x_{i\iota_{s_1}}}}a_{i\iota_{s_1}}\cdots a_{i\iota_{s_a}}\Bigr|^{2p}\Bigr)^{1/2p}\\
	&\leq \Bigl(\e\Bigl|\frac{\partial^aV_{\iota_l}(A^I)}{\partial{x_{i\iota_{s_a}}}\cdots \partial {x_{i\iota_{s_1}}}}\Bigr|^{4p}\Bigr)^{1/4p}\Bigl(\e |a_{i\iota_{s_1}}\cdots a_{i\iota_{s_a}}|^{4p}\Bigr)^{1/4p} \leq \frac{\Gamma_{k,4p,a+m}\xi_{4ap}^a}{n^{(a+m)/2}}.
	\end{align*}
	Using the Minkowski inequality, this inequality, and \eqref{lem1:eq1} yields that after dropping $1/a!$,
	\begin{align*}
	\prod_{l\in S}\bigl(\e|M_{\iota_l}^I(A)|^{2p}\bigr)^{1/2p}
	&\leq \prod_{l\in S}\sum_{a=0}^{p-1} \sum_{s_1,\ldots,s_a\in D}\frac{\Gamma_{k,4p,a+m}\xi_{4ap}^{a}}{n^{(a+m)/2}}\\
	&\leq \frac{1}{n^{|S|m/2}}\Bigl(\sum_{a=0}^{p-1} |D|^a\Gamma_{k,4p,a+m}\xi_{4ap}^{a}\Bigr)^{|S|}=:\frac{C_{D,S}}{n^{|S|m/2}}.
	\end{align*}
	Also, since
	\begin{align*}
	\e \Bigl|\int_0^1(1-t)^{p-1}f_\iota^{(p)}(t)dt\Bigr|^{2p}&\leq \e \Bigl(\int_0^1|f_\iota^{(p)}(t)|dt\Bigr)^{2p}\leq \int_0^1\e |f_\iota^{(p)}(t)|^{2p}dt,
	\end{align*}
	we have, by dropping $1/(p-1)!$, that 
	\begin{align*}
	\prod_{l\in S^c}\bigl(\e|N_{\iota_l}^I(A)|^{2p}\bigr)^{1/2p}
	&\leq 	\prod_{l\in S^c}\sum_{s_1,\ldots,s_a\in D}\Bigl(\int_0^1\e\Bigl|\frac{\partial^pV_{\iota_l}(A^I(t))}{\partial{x_{i\iota_{s_p}}}\cdots \partial {x_{i\iota_{s_1}}}}a_{i\iota_{s_1}}\cdots a_{i\iota_{s_p}}\Bigr|^{2p}dt\Bigr)^{1/2p}\\
	&\leq\prod_{l\in S^c}\frac{|D|^p\Gamma_{k,4p,p+m} \xi_{4p^2}^{p}}{n^{(p+m)/2}}\\
	&=\frac{1}{n^{(p+m)|S^c|/2}}\prod_{l\in S^c}|D|^p\Gamma_{k,4p,p+m} \xi_{4p^2}^{p}\\
	&=:\frac{C_{D,S}'}{n^{(p+m)|S^c|/2}}.
	\end{align*}
	Combining these together leads to
	\begin{align*}
	|\Delta_{n,1}|&\leq \frac{1}{n^{p/2}}\sum_{D\subseteq [p]}\sum_{I\in \mathcal{I}_D} \sum_{S\subsetneq [p]} \prod_{l\in [p]}(\e |a_{i\iota_l}|^{2p})^{1/2p}\prod_{l\in S}\bigl(\e|M_{\iota_l}^I(A)|^{2p}\bigr)^{1/2p}\prod_{l\in S^c}\bigl(\e|N_{\iota_l}^I(A)|^{2p}\bigr)^{1/2p}\\
	&\leq \frac{1}{n^{p/2}} n^{p}\xi_{2p}^p\sum_{D\subseteq [p]}\sum_{S\subsetneq [p]}\frac{C_{D,S}}{n^{|S|m/2}}\frac{C_{D,S}'}{n^{(p+m)|S^c|/2}}\\
	&=\xi_{2p}^p \sum_{D\subseteq [p]}\sum_{S\subsetneq [p]}\frac{C_{D,S}C_{D,S}'}{n^{(p(|S^c|-1)+pm)/2}}\\
	&\leq \frac{\xi_{2p}^p}{n^{pm/2}} \sum_{D\subseteq [p]}\sum_{S\subsetneq [p]}C_{D,S}C_{D,S}',
	\end{align*}
	where the last inequality used the fact that $|S^c|\geq 1$ since $S\subsetneq [p].$
	
	Next we turn to the control of $\Delta_{n,2},$ which requires more steps. Let $D\subseteq [n]$ and $I\in \mathcal{I}_D.$ Write
	\begin{align*}
	\prod_{l\in [p]}M_{\iota_l}^I(A)&=\sum_{\bar a}\frac{1}{\bar a!}\sum_{\bar s_{a_1}^1,\cdots, \bar s_{a_p}^p}\Bigl(\prod_{r\in [p]}\prod_{b=1}^{a_r}a_{i\iota_{s_b^r}}\Bigr)\Bigl(\prod_{r\in [p]}\frac{\partial^{a_r}V_{\iota_r}(A^I)}{\partial{x_{i\iota_{s_{1}^r}}}\cdots \partial {x_{i\iota_{s_{a_r}^r}}}}\Bigr),
	\end{align*}
	where the first summation is over all $\bar a=(a_1,\ldots,a_p)\in \{0,\ldots,p-1\}^p$ and $\bar a!:=a_1!\cdots a_p!$, while the second summation is over all $\bar s_{a_r}^r=(s_1^r,\ldots,s_{a_r}^r)\in D^{a_r}$ for $1\leq r\leq p.$ Set $\mathcal{S}_{\bar a}=D^{a_1}\times\cdots \times D^{a_p}.$
	Note that from our construction of $A^I$, its entries at $(i,\iota_s)$ and $(\iota_s,i)$ are all zero for all $s\in D$ and consequently, \eqref{lem:proof:eq1} implies that $A^I$ is independent of $a_{i\iota_1},\ldots,a_{i\iota_p}.$ It follows that 
	\begin{align*}
	&\Delta_{n,2}=
	\frac{1}{n^{p/2}}\sum_{\bar a}\frac{1}{\bar a!}\sum_{D\subseteq [p]}\sum_{I\in \mathcal{I}_D}\sum_{\mathcal{S}_{\bar a}}\mathcal{AL}(I,\bar s_{a_1}^1,\ldots, \bar s_{a_p}^p),
	\end{align*}
	where
	\begin{align*}
	\mathcal{A}(I,\bar s_{a_1}^1,\ldots, \bar s_{a_p}^p)&:=\e\Bigl(a_{i\iota_1}\cdots a_{i\iota_p}\prod_{r\in [p]}\prod_{b=1}^{a_r}a_{i\iota_{s_b^r}}\Bigr),\\
	\mathcal{L}(I,\bar s_{a_1}^1,\ldots, \bar s_{a_p}^p)&:=\e\Bigl(\prod_{r\in [p]}\frac{\partial^{a_r}V_{\iota_r}(A^I)}{\partial{x_{i\iota_{s_{1}^r}}}\cdots \partial {x_{i\iota_{s_{a_r}^r}}}}\Bigr),\\
	\mathcal{AL}(I,\bar s_{a_1}^1,\ldots, \bar s_{a_p}^p)&:=\mathcal{A}(I,\bar s_{a_1}^1,\ldots, \bar s_{a_p}^p)\mathcal{L}(I,\bar s_{a_1}^1,\ldots, \bar s_{a_p}^p).
	\end{align*}
	To control the right-hand side, note that 
	\begin{align*}
	\bigl|\mathcal{A}(I,\bar s_{a_1}^1,\ldots, \bar s_{a_p}^p)\bigr|\leq \xi_{p+|\bar a|}^{p+|\bar a|}
	\end{align*}
	and from \eqref{lem1:eq1},
	\begin{align*}
	\bigl|\mathcal{L}(I,\bar s_{a_1}^1,\ldots, \bar s_{a_p}^p)\bigr|&\leq \prod_{r\in [p]}\Bigl(\e\Bigl|\frac{\partial^{a_r}V_{\iota_r}(A^I)}{\partial{x_{i\iota_{s_{1}^r}}}\cdots \partial {x_{i\iota_{s_{a_r}^r}}}}\Bigr|^p\Bigr)^{1/p}\\
	&\leq \prod_{r\in [p]}\frac{\Gamma_{k,p,m+a_r}}{n^{(m+a_r)/2}}=\frac{1}{n^{(pm+|\bar a|)/2}}\prod_{r\in [p]}\Gamma_{k,p,m+a_r},
	\end{align*}
	where $|\bar a|:=a_1+\cdots+a_p.$ For any fixed $D\subseteq [p]$ and $0\leq b\leq |D|$, let $\mathcal{I}_{D,b}$ be the collection of all $I\in \mathcal{I}_D$ such that there are exactly $b$ many entries that appear once in $I.$ For any $I\in \mathcal{I}_{D,b}$, let $\mathcal{S}_{\bar a,I,b}$ be the collection of all $(\bar s_{a_1}^1,\ldots, \bar s_{a_p}^p)\in \mathcal{S}_{\bar a}$ such that all entries that appear exactly once in $I$ also appear in $$\{\iota_{s_1^1},\ldots,\iota_{s_{a_1}^1},\iota_{s_1^2},\ldots,\iota_{s_{a_2}^2},\ldots,\iota_{s_{1}^p},\ldots,\iota_{s_{a_p}^p}\}.$$ Note that when $b=0,$ every entry in $I\in \mathcal{I}_{D,0}$ must appear at least twice in $I$ and hence,  $\mathcal{S}_{\bar a,I,b}=\mathcal{S}_{\bar a},$ or equivalently, $\mathcal{S}_{\bar a,I,b}^c=\emptyset.$ Also, note that $\mathcal{S}_{\bar a,I,b}$ is nonempty only if $|\bar a|\geq b.$ From these, we can write
	\begin{align*}
	\sum_{\bar a}\sum_{I\in \mathcal{I}_D}\sum_{\mathcal{S}_{\bar a}}&=\sum_{\bar a}\sum_{b=0}^{|D|}\sum_{I\in \mathcal{I}_{D,b}}\sum_{\mathcal{S}_{\bar a,I,b}}+\sum_{\bar a}\sum_{b=0}^{|D|}\sum_{I\in \mathcal{I}_{D,b}}\sum_{\mathcal{S}_{\bar a,I,b}^c}\\
	&=\sum_{\bar a}\sum_{b=0}^{|D|\wedge |\bar a|}\sum_{I\in \mathcal{I}_{D,b}}\sum_{\mathcal{S}_{\bar a,I,b}}+\sum_{\bar a}\sum_{b=1}^{|D|}\sum_{I\in \mathcal{I}_{D,b}}\sum_{\mathcal{S}_{\bar a,I,b}^c}.
	\end{align*}
	To control the first summation, note that for any $I\in \mathcal{I}_{D,b},$ there are exactly $b$ many entries in $I$ that appear once and the other entries are repeated. This implies that $$|\mathcal{I}_{D,b}|\leq C_{D,b}n^{b+\lfloor (p-b)/2\rfloor},$$
	where $C_{D,b}$ is a universal constant independent of $n.$ Plugging this inequality to the above equation yields that
	\begin{align}
	\begin{split}\notag
	&\frac{1}{n^{p/2}}\sum_{\bar a}\sum_{b=0}^{|D|\wedge |\bar a|}\sum_{I\in \mathcal{I}_{D,b}}\sum_{\mathcal{S}_{\bar a,I,b}}\bigl|\mathcal{AL}(I,\bar s_{a_1}^1,\ldots, \bar s_{a_p}^p)\bigr|\\
	&\leq \frac{1}{n^{p/2}}\sum_{\bar a}\sum_{b=0}^{|D|\wedge |\bar a|}\frac{|\mathcal{I}_{D,b}||D|^{|\bar a|}\xi_{p+|\bar a|}^{p+|\bar a|}}{n^{(pm+|\bar a|)/2}}\prod_{r\in [p]}\Gamma_{k,p,m+a_r}\\
	&\leq\sum_{\bar a}\sum_{b=0}^{|D|\wedge |\bar a|}\frac{1}{n^{pm/2+|\bar a|/2+p/2-b-\lfloor (p-b)/2\rfloor}}C_{D,b}|D|^{|\bar a|}\xi_{p+|\bar a|}^{p+|\bar a|}\prod_{r\in [p]}\Gamma_{k,p,m+a_r}\\
	&\leq \sum_{\bar a}\sum_{b=0}^{|D|\wedge |\bar a|}\frac{1}{n^{pm/2+(p-b)/2-\lfloor (p-b)/2\rfloor}}C_{D,b}|D|^{|\bar a|}\xi_{p+|\bar a|}^{p+|\bar a|}\prod_{r\in [p]}\Gamma_{k+,p,m+a_r}
	\end{split}\\ 
	\begin{split}\label{lem:proof:eq2}
	&\leq \frac{1}{n^{pm/2}}\sum_{\bar a}\sum_{b=0}^{|D|\wedge |\bar a|}C_{D,b}|D|^{|\bar a|}\xi_{p+|\bar a|}^{p+|\bar a|}\prod_{r\in [p]}\Gamma_{k,p,m+a_r},
	\end{split}
	\end{align}
	where the third inequality used $|\bar a|\geq b.$ As for the second summation, observe that for $1\leq b\leq |D|$, if $I\in \mathcal{I}_{D,b}$ and $(\bar s_{a_1}^1,\ldots, \bar s_{a_p}^p)\in \mathcal{S}_{\bar a,I,b}^c$, then there exists one entry, say $\iota_\ell$, in $I$ that appears only once in $I$ and it does not appear in the entries of $\bar s_{a_1}^1,\cdots, \bar s_{a_p}^p.$
	Hence, $a_{i\iota_\ell}$ is independent of 
	$$
	a_{i\iota_1}\cdots a_{i\iota_{\ell-1}}a_{i\iota_{\ell+1}}\cdots a_{i\iota_p}\prod_{r\in [p]}\prod_{b=1}^{a_r}a_{i\iota_{s_b^r}},
	$$  
	which results in 
	$\e\mathcal{A}(I,\bar s_{a_1}^1,\ldots, \bar s_{a_p}^p)=0.$	Therefore,
	\begin{align}\label{lem:proof:eq3}
	\frac{1}{n^{p/2}}\sum_{\bar a}\sum_{b=1}^{|D|}\sum_{I\in \mathcal{I}_{D,b}}\sum_{\mathcal{S}_{\bar a,I,b}^c}\mathcal{AL}(I,\bar s_{a_1}^1,\ldots, \bar s_{a_p}^p)=0.
	\end{align}
	Finally, combining \eqref{lem:proof:eq2} and \eqref{lem:proof:eq3} together and dropping $1/{\bar a!}$ lead to
	\begin{align*}
	|\Delta_{n,2}|&\leq \frac{1}{n^{pm/2}}\sum_{D\subseteq [p]}\sum_{\bar a}\sum_{b=0}^{|D|\wedge |\bar a|}C_{D,b}|D|^{|\bar a|}\xi_{p+|\bar a|}^{p+|\bar a|}\prod_{r\in [p]}\Gamma_{k,p,m+a_r}.
	\end{align*}
	This completes our proof.
\end{proof}

\subsection{Proof of Proposition \ref{prop1}}

Let $m=|P|.$ From Lemma \ref{lem1}, the Minkowski inequality, and the Cauchy-Schwarz inequality, we can compute the partial derivatives of $U(X)_i$ to see that $\bigl(\e\bigl|\partial_{P}U(A)_i\bigr|^p\bigr)^{1/p}$ is bounded above by a sum, in which each summand is of the form 
\begin{align*}
\bigl(\e\bigl|\partial_{P_1}v_{j_1}(A)\cdots \partial_{P_r}v_{j_r}(A)\bigr|^{p}\bigr)^{1/p}
\end{align*} 
for some $1\leq r\leq m$ and $(j_1,\ldots, j_r) \in\mathcal{J}_r(2(k+1))$. More importantly, $\mathcal{P}=\{P_1,\ldots,P_r\}\in \mathcal{P}_r(P)$ and each term $v_{j_s}(A)$ is equal to either $(\hx_nu^{[\ell]}(A))_i$ or $u^{[\ell]}(A)_i$ for some $0\leq \ell\leq k.$  Now, using the H\"older inequality gives that
\begin{align*}
\bigl(\e\bigl|\partial_{P_1}v_{j_1}(A)\cdots \partial_{P_r}v_{j_r}(A)\bigr|^{p}\bigr)^{1/p}&\leq \bigl(\e\bigl|\partial_{P_1}v_{j_1}(A)\bigr|^{rp}\bigr)^{1/rp}\cdots \bigl(\e\bigl|\partial_{P_r}v_{j_r}(A)\bigr|^{rp}\bigr)^{1/rp}.
\end{align*}
Here, from Proposition \ref{prop2}, each term on the right-hand side is bounded above by a term of order $1/n^{|P_s|/2}$ and  they together yield that a bound of order $1/n^{|P|/2}$ since $|P_1|+\cdots+|P_r|=|P|.$ This completes our proof.

\section{Proof of Theorem \ref{thm0}}\label{sec:sec5}

We establish  universality for the generalized AMP in Definition \ref{def1}. Recall that in Theorem \ref{concentration}, we use $\tilde{\p}$ and $\tilde{\e}$ to denote the probability and expectation conditionally on $u^0,Z$. The following proposition shows that conditionally on $u^0,Z,$ the first two moments of the AMP orbits between $A$ and $G$ asymptotically match each other. 

\begin{proposition}\label{thm1}
	Consider the AMP orbit in Definition \ref{def1} with the initialization $u^{[0]}(X) = u^{0}$ and assume that $(F_k)_{k\geq 0}$ satisfies \eqref{ass}. Let $k\geq 0.$ Assume that $\phi\in C^\infty(\mathbb{R}^{k+1})$ has uniformly bounded partial derivatives of any nonzero orders. There exists a universal constant $C$ independent of $n$ such that for any $n\geq 2,$
	\begin{align}\label{thm1:eq1}
	\e\bigl|\tilde\e\Phi_{k,n}(A)-\tilde\e\Phi_{k,n}(G)\bigr|\leq \frac{C}{\sqrt{n}}
	\end{align}
	and
		\begin{align}\label{thm1:eq2}
	\e\bigl|\tilde\e \Phi_{k,n}(A)^2-\tilde\e\Phi_{k,n}(G)^2\bigr|\leq \frac{C}{\sqrt{n}},
	\end{align}
	where $\Phi_{k,n}(X)$ is defined in \eqref{add:eq2}.
\end{proposition}

The proof of Theorem \ref{thm0} is argued as follows. From Proposition \ref{add:prop1}, it suffices to assume that the functions $\phi$ and $F_k$'s in Definition \ref{def1} have uniformly bounded partial derivatives of any nonzero orders. First of all, we claim that
\begin{align*}
\lim_{n\to\infty}\e\bigl(\Phi_{k,n}(A)-\tilde{\e}\Phi_{k,n}(A)\bigr)^2&=\lim_{n\to\infty}\e\bigl(\Phi_{k,n}(G)-\tilde{\e}\Phi_{k,n}(G)\bigr)^2=0.
\end{align*}Note that from Theorem \ref{concentration}, 
\begin{align*}
\lim_{n\to\infty}\e|\Phi_{k,n}(G)-\tilde{\e}\Phi_{k,n}(G)|^2=0.
\end{align*}
Now from Proposition \ref{thm1},
\begin{align*}
\lim_{n\to\infty}\bigl|\e\Phi_{k,n}(A)^2-\e\Phi_{k,n}(G)^2\bigr|\leq \lim_{n\to\infty}\e\bigl|\tilde\e\Phi_{k,n}(A)^2-\tilde\e\Phi_{k,n}(G)^2\bigr|=0.
\end{align*}
Since
\begin{align*}
&\bigl|\e\bigl(\tilde\e\Phi_{k,n}(A)\bigr)^2-\e\bigl(\tilde\e\Phi_{k,n}(G)\bigr)^2\bigr|\\
&= \e\bigl(\bigl|\tilde\e\Phi_{k,n}(A)-\tilde\e\Phi_{k,n}(G)\bigr|\bigl|\tilde\e\Phi_{k,n}(A)+\tilde\e\Phi_{k,n}(G)\bigr|\bigr)\\
&\leq \bigl(\e\bigl(\bigl|\tilde\e\Phi_{k,n}(A)-\tilde\e\Phi_{k,n}(G)\bigr|^2\bigr)^{1/2}\bigl(\e\bigl|\tilde\e\Phi_{k,n}(A)+\tilde\e\Phi_{k,n}(G)\bigr|^2\bigr)^{1/2}\\
&\leq \bigl(\e\bigl(\bigl|\tilde\e\Phi_{k,n}(A)-\tilde\e\Phi_{k,n}(G)\bigr|^2\bigr)^{1/2}\bigl(\bigl(\e\Phi_{k,n}(A)^2\bigr)^{1/2}+\bigl(\e\Phi_{k,n}(G)^2\bigr)^{1/2}\bigr),
\end{align*}
it follows from Proposition \ref{thm1} and the moment control in \eqref{add:lem3:eq2}, 
\begin{align*}
\lim_{n\to\infty}\bigl|\e\bigl(\tilde\e\Phi_{k,n}(A)\bigr)^2-\e\bigl(\tilde\e\Phi_{k,n}(G)\bigr)^2\bigr|=0.
\end{align*}
Putting these limits together yields the claim. Consequently, the proof of Theorem \ref{thm0} follows by our claim and Proposition \ref{thm1},
\begin{align*}
\e \bigl(\Phi_{k,n}(A)-\Phi_{k,n}(G)\bigr)^2
&\leq 9\e\bigl(\Phi_{k,n}(A)-\tilde{\e}\Phi_{k,n}(A)\bigr)^2\\
&+9\e\bigl(\tilde{\e} \Phi_{k,n}(A)-\tilde\e\Phi_{k,n}(G)\bigr)^2\\
&+9\e\bigl(\Phi_{k,n}(G)-\tilde{\e}\Phi_{k,n}(G)\bigr)^2\rightarrow 0.
\end{align*}

For the rest of this section, we establish  Proposition \ref{thm1} in five subsections using the Gaussian interpolation and approximate Gaussian integration by parts. In doing these,  Proposition \ref{prop1} will be of great use in tracking the error terms. Subsection \ref{sec6.1} shows that to prove Proposition~\ref{thm1}, it suffices to assume that the main diagonals of $A,G,Z$ are all equal to zero. The Gaussian interpolation between $\Phi_{k,n}(A)$ and $\Phi_{k,n}(G)$ is introduced in Subsection~\ref{sec:interpolation} and the control of its derivative  is handled in Subsection~\ref{sec6.3}. Finally, the proofs of \eqref{thm1:eq1} and \eqref{thm1:eq2} are established in Subsections \ref{sec6.4} and \ref{sec6.5}, respectively.

\subsection{Deletion of the main diagonal}\label{sec6.1}
By the virtue of Proposition \ref{prop0}, it suffices to assume that the main diagonals in $A$ and $Z$ are zero. To see this, recall $\Phi_{k,n}(X),\Delta_k(X)$, and $\Theta_k(X)$ from Proposition \ref{prop0}. Assume that $A'$ is equal to $A$ except that the main diagonal  vanishes. Note that from Lemma \ref{add:lem5}, $\|A_n\|_2$ and $\|A_n'\|_2$ are of order $O(1)$.  
On the other hand, since
$$\|A_n-A_n'\|_2=\frac{1}{\sqrt{n}}\max_{1\leq i\leq n}|a_{ii}|$$
and
\begin{align*}
\p\Bigl(\frac{1}{\sqrt{n}}\max_{1\leq i\leq n}|a_{ii}|\geq t \Bigr)\leq \sum_{i\in[n]}P(|a_{ii}|\geq t\sqrt{n})\leq ne^{-n\sigma^2 t^2/2},
\end{align*}
these imply that in probability,
\begin{align*}
\lim_{n\to\infty}\|A_n-A_n'\|_2=0.
\end{align*}
As a result, Proposition \ref{prop0} implies that the AMP orbits corresponding to $(A,Z)$ and $(A',Z)$ satisfy that in probability,
\begin{align*}
\lim_{n\to\infty}\bigl|\Phi_{k,n}(A)-\Phi_{k,n}(A')\bigr|=0,
\end{align*}
which together with Lemma \ref{add:lem3} gives that 
\begin{align*}
\lim_{n\to\infty}\e\bigl|\Phi_{k,n}(A)-\Phi_{k,n}(A')\bigr|=0.
\end{align*} 
Next one can prove by an almost identical argument to show that the AMP orbits correspond to $(A',Z)$ and $(A',Z')$ are also asymptotically the same under the $L^1(\p)$-distance, where $Z'$ is the same as $Z$ except that its main diagonal is zero. From this, in what follows, we assume that the main diagonals of $A$, $G$, and $Z$ are all equal to zero.

\subsection{Interpolation}\label{sec:interpolation}

Define the Gaussian interpolation between $A$ and $G$ by 
\begin{align*}
A(t)&=(a_{ij}(t))_{i,j\in[n]}=\sqrt{t}A+\sqrt{1-t}G,\,\,0\leq t\leq 1.
\end{align*}
Denote 
\begin{align*}
A_n(t)=\frac{A(t)}{\sqrt{n}}
\end{align*}
and
\begin{align*}
\hat A_n(t)=\frac{A(t)}{\sqrt{n}}+\frac{Z}{n}.
\end{align*} 
For $\phi\in C^\infty(\mathbb{R}^{k+1})$ with uniformly bounded partial derivatives of all nonzero orders, define
\begin{align*}
\Phi_{k,n}(t)&=\tilde\e\Phi_{k,n}(A(t)).
\end{align*}
Note that
\begin{align*}
	\e\bigl|\tilde\e\Phi_{k,n}(A)-\tilde\e\Phi_{k,n}(G)\bigr|=\e\bigl|\Phi_{k,n}(1)-\Phi_{k,n}(0)\bigr|.
\end{align*}
To show \eqref{thm1:eq1}, our goal is to show that 
\begin{align*}
\int_0^1\e |\Phi_{k,n}'(t)|dt\leq \frac{C}{\sqrt{n}}
\end{align*}
for some constant $C$ independent of $n.$
Note that $u^{[0]}(X)=u$. A direct differentiation gives
\begin{align}
\begin{split}\label{diff}
\Phi_{k,n}'(t)&=\sum_{\ell=1}^{k}\frac{1}{n}\sum_{i=1}^n\tilde\e\Bigl(\partial_{y_\ell}\phi\bigl(u^{[k]}(A(t)),\ldots, u^{[0]}(A(t))\bigr)\circ\frac{d}{dt}u^{[\ell]}(A(t))\Bigr)_i,\,\,0<t<1.
\end{split}
\end{align}
Here and thereafter, if $v,v'\in \mathbb{R}^n$, we define $v\circ v'=(v_iv_i')_{i\in [n]}$ as the Hadamard product between $v$ and $v'$. Note that this operation is commutative.

To simplify our notation, denote
\begin{align*}
\dot A_n(t)&=(\dot{a}_{ii'}(t))_{i,i'\in[n]},
\end{align*}
for
$$
\dot{a}_{ii'}(t)=\frac{1}{2}\Bigl(\frac{a_{ii'}}{\sqrt{t}}-\frac{g_{ii'}}{\sqrt{1-t}}\Bigr).
$$
Also, denote
\begin{align*}
u^{[\ell]}(t)&=u^{[\ell]}(A(t)),\\
\partial_{y_r}F_\ell(t)&=\partial_{y_r}F_\ell(\ha_n(t)u^{[\ell]}(t),u^{[\ell-1]}(t),\ldots,u^{[0]}(t)).
\end{align*}
Observe that
\begin{align*}
\frac{d}{dt}u^{[1]}(t) &=\partial_{y_0}F_0(t) \circ (\dot A_n(t) u^{[0]}(t) ),\\
\frac{d}{dt}u^{[2]}(t) &=\partial_{y_1}F_1 (t)\circ (\dot A_n(t) u^{[1]}(t) )\\
&+\partial_{y_1}F_1(t) \circ \bigl(\ha_n \bigl(\partial_{y_0}F_0(t) \circ (\dot A_n(t) u^{[0]}(t) )\bigr)\bigr),
\end{align*}
and
\begin{align*}
\frac{d}{dt}u^{[3]}(t) &=\partial_{y_2}F_2 (t)\circ (\dot A_n(t) u^{[2]}(t) )\\
&+\partial_{y_2}F_2 (t)\circ \bigl(\ha_n(t) \bigl(\partial_{y_1}F_1(t) \circ (\dot A_n(t) u^{[1]} (t))\bigr)\bigr)\\
&+\partial_{y_2}F_2(t) \circ \Bigl(\ha_n(t) \Bigl(\partial_{y_1}F_1(t) \circ \bigl(\ha_n(t) \bigl(\partial_{y_0}F_0(t) \circ (\dot A_n(t) u^{[0]}(t) )\bigr)\bigr)\Bigr)\Bigr)\\
&+\partial_{y_1}F_2(t) \circ \partial_{y_0}F_0(t) \circ (\dot A_n(t) u^{[0]}(t) ).
\end{align*}
For general $1\leq \ell\leq k,$
\begin{align*}
\frac{d}{dt}u^{[\ell]}(t)
&=\partial_{y_{\ell-1} }F_{\ell-1} (t)\Bigl(\dot A_n(t)u^{[\ell-1]}(t)+\ha_n(t)\frac{d}{dt}u^{[\ell-1]}(t)\Bigr)+\sum_{s=1}^{\ell -2}\partial_{y_s}F_{\ell-1} (t)\frac{d}{dt}u^{[s]}(t).
\end{align*}
From these equations, one readily sees that the vector 
$$
\partial_{y_\ell}\phi\bigl(u^{[k]}(t),\ldots, u^{[0]}(t)\bigr)\circ\frac{d}{dt}u^{[\ell]}(A(t))
$$
appearing in \eqref{diff} can be written as a summation of column vectors, in which each summand is of the form $w^r(t)=(w_i^r(t))_{i\in [n]}$ for some $0\leq r\leq \ell-1$ that is defined by an iterative procedure through some functions $L^0,L^1,\ldots,L^{r+1}\in C^{\infty}(\mathbb{R}^{2\ell})$, whose partial derivatives of any nonzero orders are uniformly bounded. More precisely, starting from
\begin{align*}
w^{r,[0]}(t)&=U^{[1]}(A(t))\circ (\dot A_n(t)U^{[0]}(A(t))),
\end{align*}
define  
\begin{align}\label{iteration}
w^{r,[s]}(t)&=U^{[s+1]}(A(t))\circ (\hat A_n(t)w^{r,[s-1]}(t)),\,\,\forall 1\leq s\leq r,
\end{align}
where 
\begin{align}\label{iteration:eq1}
U^{[s]}(X)&:=L^s\bigl(\hat X_nu^{[\ell-1]}(X),\ldots,\hat X_nu^{[0]}(X),u^{[\ell-1]}(X),\ldots,u^{[0]}(X)\bigr),\,\,\forall 0\leq s\leq r
\end{align}
and 
\begin{align}
\begin{split}\label{iteration:eq2}
U^{[r+1]}(X)&:=\partial_{y_\ell}\phi\bigl(u^{[k]}(X),\ldots, u^{[0]}(X)\bigr)\\
&\qquad \circ L^{r+1}\bigl(\hat X_nu^{[\ell-1]}(X),\ldots,\hat X_nu^{[1]}(X),u^{[\ell-1]}(X),\ldots,u^{[0]}(X)\bigr).
\end{split}
\end{align}
Finally, set $w^r(t)=w^{r,[r]}(t).$

\subsection{Bounding the derivative of the interpolation}\label{sec6.3}

For $r\geq 0$, from the iteration \eqref{iteration} and expanding the Hadamard product,
\begin{align}\label{eq-1}
&\frac{1}{n}\sum_{i=1}^n\tilde\e w_{i}^{r}(t)=\frac{1}{n^{1+(r+1)/2}}\sum_{I\in \mathcal{I}_{r}}\tilde\e V_I(t)
\end{align}
for
\begin{align}
\begin{split}\label{V}
V_I(t)&:=\Bigl(\prod_{l=1}^{r} U_{i_{l+1}}^{[l+1]}(A(t))\bigl(a_{i_{l+1},i_{l}}(t)+ n^{-1/2}z_{i_{l+1}, i_l}\bigr)\Bigr)U_{i_1}^{[1]}(A(t))\dot a_{i_1,i_0}(t)U_{i_0}^{[0]}(A(t))\\
&=\Bigl(\prod_{l=0}^{r+1} U_{i_{l}}^{[l]}(A(t))\Bigr) \Bigl(\prod_{l=1}^r\bigl(a_{i_{l+1},i_{l}}(t)+ n^{-1/2}z_{i_{l+1}, i_l}\bigr)\Bigr)\dot a_{i_1,i_0}(t),
\end{split}
\end{align}
where $\mathcal{I}_{r}$ is the collection of all $I=(i_0,i_1,\ldots,i_{r},i_{r+1})\in [n]^{r+2}$ with $i_0\neq i_1\neq \cdots \neq i_r\neq i_{r+1}.$ Here we view each $I$ as a directed graph of length $r+1$ with vertices $(i_s)_{0\leq s\leq r+1}$ and edges $e_I(l)=(i_l,i_{l+1})$ for $0\leq l\leq r.$ For any $I\in \mathcal{I}_r$, disregard the direction, let $\Lambda_I^0$ be the collection of all $0\leq l\leq r$ with $e_I(l)=e_I(0)$. Let $\mathcal{I}_{r}(s)$ be the collection of all graphs in $\mathcal{I}_r$ so that disregard the direction there are exactly $s$ many edges that appear once. 

\begin{proposition}\label{prop3}
	For any $I\in \mathcal{I}_r(s)$, we have that for any $0<t<1,$
	\begin{align*}
	\int_0^1\e\bigl|\tilde\e V_I(t)\bigr|dt&\leq \left\{
	\begin{array}{ll}
	\frac{C_s}{n^{(s+1)/2}},&\mbox{if $|\Lambda_I^0| \leq 2$},\\
	\\
	\frac{C_s}{n^{s/2}},&\mbox{if $|\Lambda_I^0|\geq 3$},
	\end{array}\right.
	\end{align*}
	where $\tilde\e$ is the expectation conditionally on $u^0,Z$ and $C_s$ is a universal constant independent of $n$.
\end{proposition}


To prove this proposition, we first establish a key lemma. For any $0\leq b\leq r+1,$ let $\mathcal{I}_{r}(s,b)$ be the collection of all $I\in \mathcal{I}_r(s)$ with $|\Lambda_I^0|=b.$ Note that when $b=1,$ the set $\mathcal{I}_{r}(s,b)$ is nonempty for all $1\leq s\leq r+1$ and when $2\leq b\leq r+1,$ the set $\mathcal{I}_{r}(s,b)$ is nonempty only if $0\leq s\leq r+1-b.$

\begin{lemma}\label{lem5}
	If $b=1,$ then for any $1\leq s\leq r+1,$
	\begin{align}\label{lem5:eq1}
	|\mathcal{I}_{r}(s,1)|&\leq  C_{r,1,s}n^{\lfloor\frac{r+1-s}{2}\rfloor +s+1}.
	\end{align}
	If $2\leq b\leq r+1,$ then for any $0\leq s\leq r+1-b,$
	\begin{align}\label{lem5:eq2}
	|\mathcal{I}_{r}(s,b)|\leq C_{r,b,s}n^{\lfloor \frac{r+1-b-s}{2}\rfloor +s+2}.
	\end{align}
	Here, $C_{r,b,s}$'s are universal constants independent of $n.$ 
\end{lemma}

\begin{proof}
	For any graph $I\in \mathcal{I}_{r}(s,1)$, let $I'$ be the graph, in which we disregard both multiplicities and directions of the edges. See Figure~\ref{fig:images} for examples. Observe that there are at most $\lfloor (r+1-s)/2\rfloor$ many edges in $I'$ that appear at least twice in $I$ and the total number of edges of $I'$ is at most $\lfloor (r+1-s)/2\rfloor+s.$ This implies that the total number of vertices of $I'$ should be at most $\lfloor (r+1-s)/2\rfloor+s+1$ so there are at most $n^{\lfloor\frac{r+1-s}{2}\rfloor +s+1}$ many such $I'$. 
	
	Since different $I$ can correspond to the same $I'$ (again we refer the reader to Figure~\ref{fig:images} for examples), it remains to show that for each $I'$, there are at most a constant multiple (independent of $n$) of many different $I$ that corresponds to $I'$. First fix such a possible $I'$ and write $E(I')$ for the edge set of $I'$. Each edge of $I'$ may correspond to an edge of multiplicity $1$ in $I$ or to an edge of multiplicity at least $2$ in $I$ (ignoring directions). We choose $s$ edges in $I'$ so that they correspond to the multiplicity $1$ edges in $I$. There are ${|E(I')| \choose s}$ ways to choose such $s$ edges. Now, for the remaining $(|E(I')| - s)$ many edges, the multiplicities are at least $2$ in $I$ and they add up to $r+1-s$. To count how many possibilities there are, it is equivalent to find how many ways $r+1-s$ can be written as sum of $(|E(I')| - s)$ many integers which are at least $2$, which in turn is bounded above by the number of ways to partition the integer $r+1-s$ as sum of $(|E(I')|' - s)$ many positive integers, and it is well-known that the number of ways is ${r-s\choose |E(I')| - s - 1}$. Moreover, each edge has two possible directions in $I$. Finally, each vertex in $I'$ can correspond to several $i_t$ ($0\leq t\leq r+1$) in $I$ (see Figure~\ref{fig:images}). As there are at most $\lfloor (r+1-s)/2\rfloor+s+1$ vertices in $I'$, and each vertex can correspond to at most $r+1$ many $i_t$'s, there are at most $(r+1)^{\lfloor (r+1-s)/2\rfloor+s+1}$ many such correspondence. Therefore, the total number of $I$ that an $I'$ can correspond to is bounded above by
	\[
	(r+1)^{\lfloor (r+1-s)/2\rfloor+s+1}2^{{|E(I')| \choose s} \cdot {r-s \choose |E(I')| - s - 1}} \leq (r+1)^{\lfloor (r+1-s)/2\rfloor+s+1}2^{2^{\lfloor (r+1-s)/2\rfloor}\cdot 2^{r-s}} =: C_{r,1,s}.
	\] 
	This proves \eqref{lem5:eq1}. 
	
	To prove \eqref{lem5:eq2}, note that in this case, the edge $e_I(0)$ has multiplicity $b$, and hence there are at most $\lfloor (r+1-b-s)/2\rfloor+1$ many edges in $I'$ that appear at least twice in $I$. Here, the latest $+1$ comes from $e_I(0)$. The remaining of the proof is similar to that of \eqref{lem5:eq1}, and we omit the detail.
\end{proof}

\begin{figure}[h!]
	\begin{subfigure}{0.50\textwidth}
		\centering
		\includegraphics[trim={0 15cm 0 15cm},clip,width=1\linewidth]{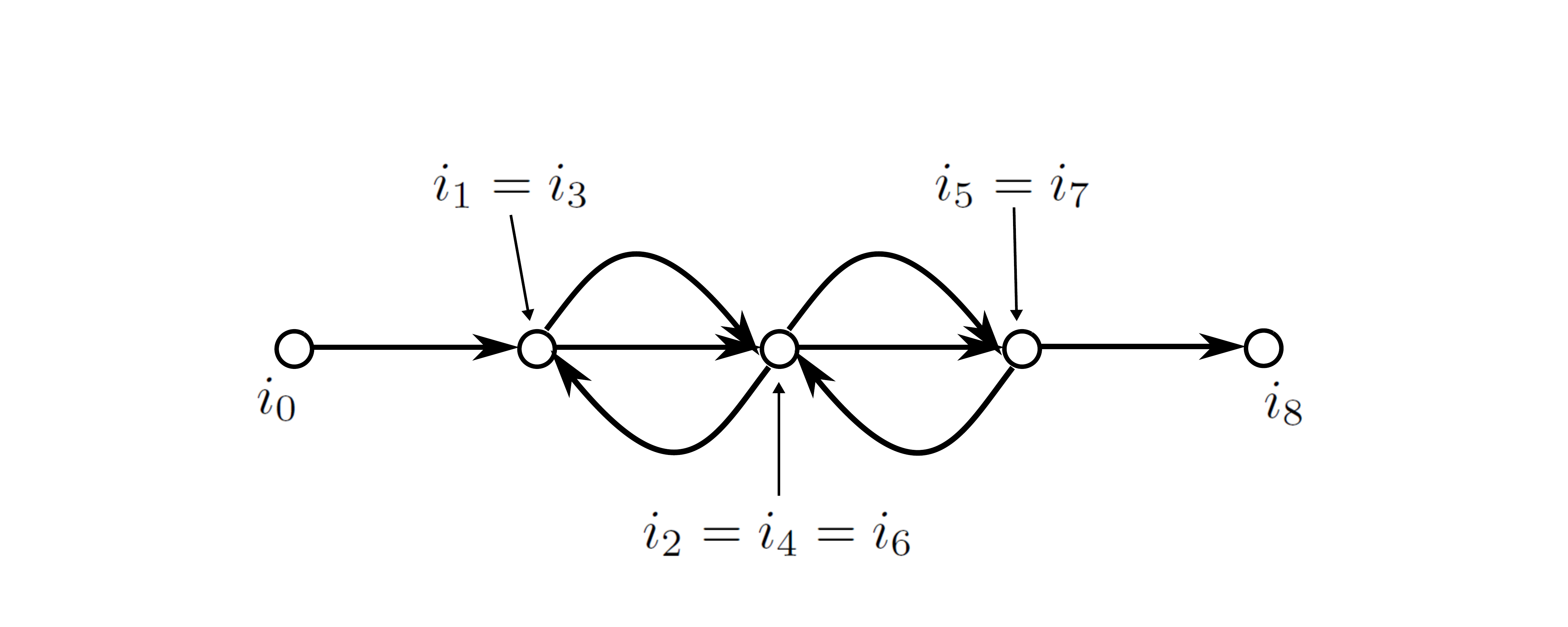} 
		\caption{One example of $I\in \mathcal{I}_{r}(s,1)$}
		\label{fig:I1}
	\end{subfigure}
	~
	\begin{subfigure}{0.50\textwidth}
		\centering
		\includegraphics[trim={0 15cm 0 15cm},clip,width=1\linewidth]{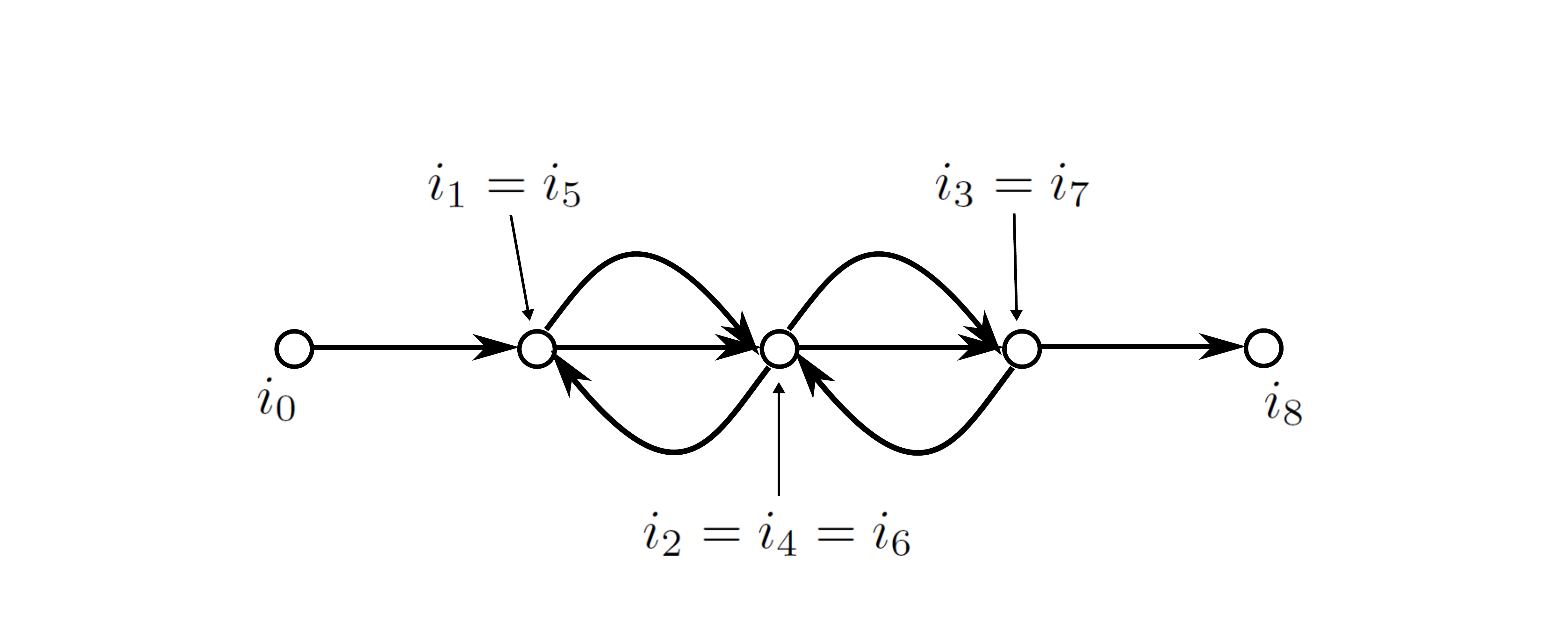}
		\caption{Another example of $I\in \mathcal{I}_{r}(s,1)$}
		\label{fig:I2}
	\end{subfigure}
	~
	\begin{center}
		\begin{subfigure}{0.50\textwidth}
			\centering
			\includegraphics[trim={0 15cm 0 15cm},clip,width=1\linewidth]{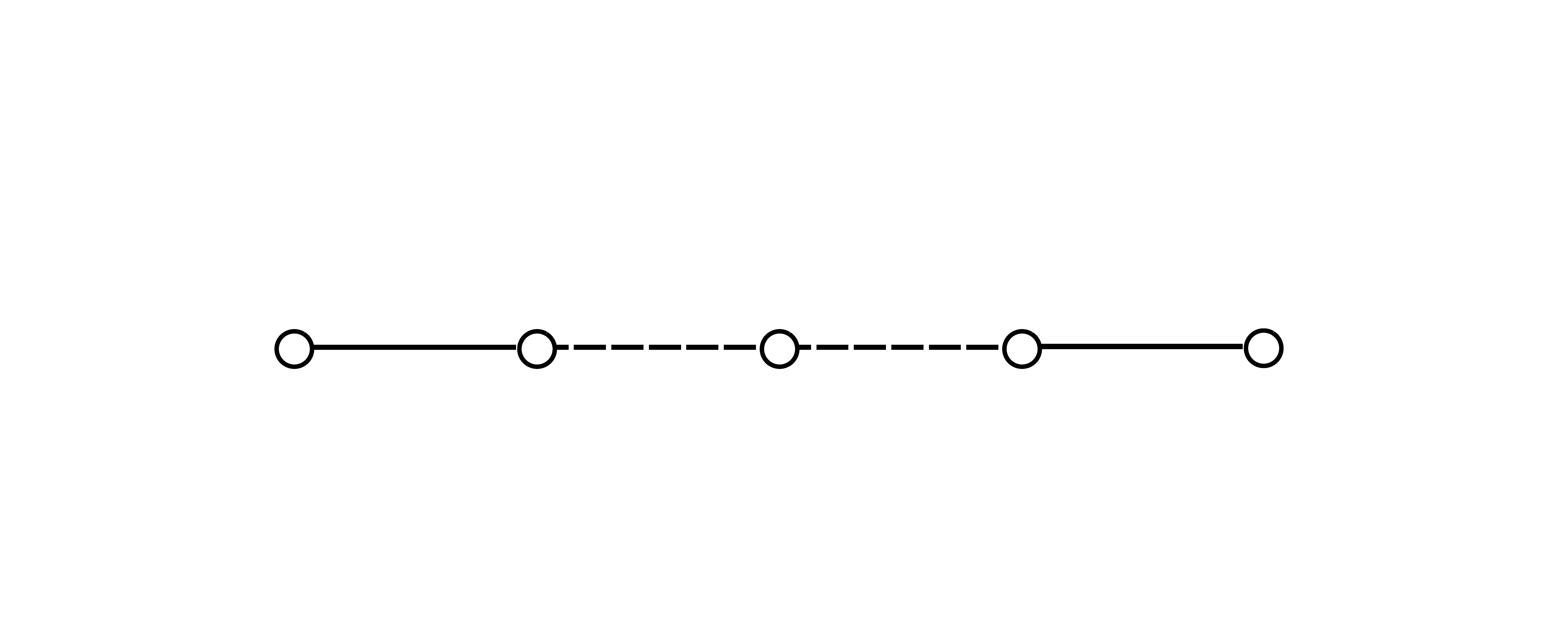}
			\caption{The corresponding undirected graph $I'$}
			\label{fig:I'}
		\end{subfigure}
	\end{center}
	
	\caption{\small (\subref{fig:I1}) and (\subref{fig:I2}) are two different directed graphs in $\mathcal{I}_{r}(s,1)$ for $r=7$ and $s=2$. After we disregard the multiplicities and directions, they correspond to the same $I'$ as shown in (\subref{fig:I'}), where the solid edges correspond to the edges in $I$ that appear only once and the dashed edges correspond to those in $I$ with multiplicity $\geq 2$. }
	\label{fig:images}
\end{figure}


\begin{proof}[\bf Proof of Proposition \ref{prop3}]
	Let $I\in \mathcal{I}_r(s)$ be fixed. Disregard the direction, let $\Lambda_I^1$ be the collection of all $l\notin \Lambda_I^0$ so that $e_I(l)$ appears exactly once in $I$ and $\Lambda_I^2$  be the collection of all $l\notin \Lambda_{I}^0$ so that $e_I(l)$ appears more than once in $I.$ In addition, for any $R\subseteq[r]$, set $\Lambda_I^0(R)=\Lambda_I^0\cap R$, $\Lambda_I^1(R)=\Lambda_I^1\cap R$, and $\Lambda_I^2(R)=\Lambda_I^2\cap R$. Let $\Lambda_I(R):=\Lambda_I^1(R)\cup\{0\}.$ From these and \eqref{V}, after expanding
	$$
	\prod_{l=1}^r\bigl(a_{i_{l+1},i_{l}}(t)+ n^{-1/2}z_{i_{l+1},i_l}\bigr),
	$$
	we can write
	\begin{align*}
	\tilde\e V_I(t)&=\sum_{R\subseteq[r]}\frac{Z_{I,R}}{n^{|R^c|/2}}\tilde\e\Bigl[ U_{I}(A(t))A_{I,R}(t)\Bigl(a_{e_I(0)}(t)^{|\Lambda_I^0(R)|}\dot a_{e_I(0)}(t)\Bigr)\Bigl(\prod_{l\in  \Lambda_I^1(R)}a_{e_I(l)}(t)\Bigr)\Bigr]
	\end{align*}
	for
	\begin{align*}
	U_I(X)&:=\prod_{l=0}^{r+1}U_{i_l}^{[l]}(X),\,\,A_{I,R}(t):=\prod_{l\in \Lambda_I^2(R)}a_{e_I(l)}(t),\,\,
	Z_{I,R}:=\prod_{l\in R^c}z_{i_{l+1}, i_l}.
	\end{align*}
	Note that inside the expectation, the two parentheses are independent of $A_{I,R}(t)$, and each term in the second parentheses appears only once in $I$. From these, we can apply Proposition~\ref{prop1} and Lemmas~\ref{lem3}, \ref{lem4}, and \ref{lem6} in Appendix to control $V_I(t)$. To see this, note that for any $0<t<1$ and $p\geq 1,$
	\begin{align}\label{add:proof:eq1}
	\bigl(\e |a_{e_I(0)}(t)|^p\bigr)^{1/p}\leq C_p
	\end{align}
	and
	\begin{align*}
	\bigl(\e |\dot a_{e_I(0)}(t)|^p\bigr)^{1/p}&\leq C_p\Bigl(\frac{1}{\sqrt{t}}+\frac{1}{\sqrt{1-t}}\Bigr)
	\end{align*}
	for some universal constant $C_p$ independent of $t$.
	Let $\e_{I,R}$ be the expectation only with respect to $a_{e_I(l)}(t)$ for all $l\in \Lambda_I(R)$. Using these bounds and Lemmas~\ref{lem3}, \ref{lem4}, and \ref{lem6}, we get that
	\begin{align*}
	\Bigl|\e_{I,R}\Bigl[U_{I}(A(t))\Bigl(a_{e_I(0)}(t)^{|\Lambda_I^0(R)|}\dot a_{e_I(0)}(t)\Bigr)\Bigl(\prod_{l\in \Lambda_I^1(R)}a_{e_I(l)}(t)\Bigr)\Bigr)\Bigr]\Bigr|
	\end{align*}
	is bounded above, up to an absolute constant independent of $I$ and $n$, by
	\begin{align*}
	\Bigl(\frac{1}{\sqrt{t}}+\frac{1}{\sqrt{1-t}}\Bigr)\sum_{|\alpha|=s_0}\Bigl(\int_0^1 \e_{I,R}\bigl[\bigr|\partial_{I,R}^{\alpha} U_{I}(A(t,\xi)))\bigr|^2\bigr]d\xi\Bigr)^{1/2},
	\end{align*}
	where $A(t,\xi)=(a_{ii'}(t,\xi))_{i,i'\in [n]}$ is defined as $a_{ii'}(t,\xi)=\xi a_{ii'}(t)$ for all $i,i'\in [n]$ satisfying $(i,i')=e_{I}(l)$ or $(i',i)=e_I(l)$ for some $l\in \Lambda_I(R)$ and $a_{ii'}(t,\xi)=a_{ii'}(t)$ otherwise.   
	Here, 
	\begin{align}\label{add:eq4}
	s_0=\left\{
	\begin{array}{ll}
	|\Lambda^1_I(R)|+2,&\mbox{if $|\Lambda_I^0(R)|=0$ by \eqref{lem3:eq3}},\\
	|\Lambda^1_I(R)|+1,&\mbox{if $|\Lambda_I^0(R)|=1$ by \eqref{lem3:eq1}},\\
	|\Lambda^1_I(R)|,&\mbox{if $|\Lambda_I^0(R)|\geq 2$ by \eqref{lem3:eq2}}.
	\end{array}
	\right.
	\end{align}
	The summand in the above bound is over all $\alpha:=(\alpha_l)_{l\in \Lambda_I(R)}\in (\{0\}\cup \mathbb{N})^{s_0}$, $|\alpha|:=\sum_{l\in \Lambda_I(R)}\alpha_l$, and $\partial_{I,R}^{\alpha}$ is the partial derivative with respect to $x_{e_I(l)}$ of order $\alpha_l$ for all $l\in \Lambda_I(R).$ From this inequality, it follows that
	\begin{align}
	\begin{split}\notag
	&\int_0^1\Bigl|\e\Bigl[ U_{I}(A(t))A_{I,R}(t)\Bigl(a_{e_I(0)}(t)^{|\Lambda_I^0(R)|}\dot a_{e_I(0)}(t)\Bigr)\Bigl(\prod_{l\in  \Lambda^1_I(R)}a_{e_I(l)}(t)\Bigr)\Bigr]\Bigr|dt\\
	&\leq C\sum_{|\alpha|=s_0}\int_0^1\e\Bigl[|A_{I,R}(t)|\Bigl(\int_0^1 \e_{I,R}\bigl[\bigr|\partial_{I,R}^\alpha U_{I}(A(t,\xi)))\bigr|^2\bigr]d\xi\Bigr)^{1/2}\Bigr]dt\\
	&\leq C\int_0^1\Bigl(\e A_{I,R}(t)^{2}\Bigr)^{1/2}\sum_{|\alpha|=s_0}\Bigl(\int_0^1 \e\bigr|\partial_{I,R}^\alpha U_{I}(A(t,\xi)))\bigr|^2d\xi \Bigr)^{1/2}dt
	\end{split}\\
	\begin{split}\label{add:eq3}
	&\leq C'\int_0^1\sum_{|\alpha|=s_0}\Bigl(\int_0^1 \e\bigr|\partial_{I,R}^\alpha U_{I}(A(t,\xi)))\bigr|^2d\xi \Bigr)^{1/2}dt,
	\end{split}
	\end{align}
	for some constants $C$ and $C'$ independent of $I$ and $n,$ where the second inequality used the Cauchy-Schwarz inequality, while the third inequality used the H\"older inequality and the bounds \eqref{add:eq--1} and \eqref{add:proof:eq1}. The last inequality can further be  controlled as follows.	
	From the product rule,
	\begin{align*}
	\partial_{I,R}^{\alpha} U_{I}&=\sum_{\beta_l:l\in \Lambda_I(R)}\Bigl(\prod_{l\in \Lambda_I(R)}{\alpha_l\choose \beta_l}\Bigr)\Bigl(\prod_{\ell=1}^{r+1}\partial_{I}^{\beta(\ell)}U_{i_{\ell}}^{[\ell]}\Bigr),
	\end{align*}
	where the summand is over all $\beta_l\in (\{0\}\cup\mathbb{N})^{r+2}$  satisfying that $\sum_{\ell=0}^{r+1}\beta_{l,\ell}=\alpha_l$ for $l\in \Lambda_I(R)$ and $\partial_I^{\beta(\ell)}$ is the partial derivative $\partial_{x_{e_I(l)}}^{\beta_{l,\ell}}$ for all $l\in \Lambda_I(R).$ Now, using the Minkowski and H\"older inequalities leads to
	\begin{align*}
	&\bigl(\e\bigr|\partial_{I,R}^{\alpha} U_{I}( A(t,\xi)))\bigr|^2\bigr)^{1/2}\\
	&\leq \sum_{\beta_l:l\in \Lambda_I(R)}\Bigl(\prod_{l\in \Lambda_I(R)}{\alpha_l\choose \beta_l}\Bigr)\prod_{\ell=1}^{r+1}\Bigl(\e\bigl|\partial_{I,R}^{\beta(\ell)}U_{i_{\ell}}^{[\ell]}(A(t,\xi))\bigr|^{2(r+2)}\Bigr)^{1/2(r+2)}.
	\end{align*}
	From \eqref{iteration:eq1} and \eqref{iteration:eq2}, note  that any nonzero-order partial derivatives of $U^{[\ell]}$'s are uniformly bounded. From Proposition \ref{prop1}, each term on the right-hand side is bounded by
	\begin{align*}
	\Bigl(\e\bigl|\partial_{I,R}^{\beta(\ell)}U_{i_{\ell}}^{[\ell]}(A(t,\xi))\bigr|^{2(r+2)}\Bigr)^{1/2(r+2)}&\leq \frac{\Gamma_\ell}{n^{\sum_{l\in \Lambda_I(R)}\beta_{l,\ell}/2}},
	\end{align*}
	where $\Gamma_\ell$ is a constant independent of $n.$ Consequently,
	\begin{align*}
	\bigl(\e\bigr|\partial_{I,R}^{\alpha} U_{I}(A(t,\xi)))\bigr|^2\bigr)^{1/2}&\leq \sum_{\beta_l:l\in \Lambda_I(R) }\Bigl(\prod_{l\in \Lambda_I(R) }{\alpha_l\choose \beta_l}\Bigr)\frac{\prod_{\ell=0}^{r+1}\Gamma_\ell}{n^{\sum_{\ell=0}^{r+1}\sum_{l\in \Lambda_I(R) }\beta_{l,\ell}/2}}\\
	&=\sum_{\beta_l:l\in \Lambda_I(R) }\Bigl(\prod_{l\in \Lambda_I(R) }{\alpha_l\choose \beta_l}\Bigr)\frac{\prod_{\ell=0}^{r+1}\Gamma_\ell}{n^{\sum_{l\in \Lambda_I(R) }\alpha_l/2}}\\
	&=\frac{1}{n^{s_0/2}}\sum_{\beta_l:l\in \Lambda_I(R) }\Bigl(\prod_{l\in \Lambda_I(R) }{\alpha_l\choose \beta_l}\Bigr)\prod_{\ell=0}^{r+1}\Gamma_\ell.
	\end{align*}
	Plugging this inequality into \eqref{add:eq3} and noting that $Z$ is independent of $A,G$ together with the bound \eqref{add:eq--1} yield that
	\begin{align*}
	\frac{1}{n^{|R^c|/2}}\int_0^1\Bigl|\e\Bigl[ U_{I}(A(t))Z_{I,R}A_{I,R}(t)\Bigl(a_{e_I(0)}(t)^{|\Lambda_I^0(R)|}\dot a_{e_I(0)}(t)\Bigr)\Bigl(\prod_{l\in  \Lambda_I^1(R)}a_{e_I(l)}(t)\Bigr)\Bigr]\Bigr|dt\leq \frac{C''}{n^{|R^c|/2+s_0/2}}.
	\end{align*}
	Here, note that 
	\begin{align*}
	|R^c|+|\Lambda_I^1(R)|&=\bigl(|\Lambda_I^0(R^c)|+|\Lambda_I^1(R^c)|+|\Lambda_I^2(R^c)|\bigr)+|\Lambda_I^1(R)|\\
	&=|\Lambda_I^0(R^c)|+|\Lambda_I^1([r])|+|\Lambda_I^2(R^c)|\\
	&\geq |\Lambda_I^1([r])|.
	\end{align*}
	Also, note that $s$ is the number of edges in $I$ that are crossed once disregard the direction. This implies that $\Lambda_I^1([r])\geq s-1$ and that $\Lambda_I^1([r])=s$ if $|\Lambda_I^0(R)|\geq 1$ since $|\Lambda_I^0|\geq 1+|\Lambda_I^0(R)|\geq 2.$ Recall \eqref{add:eq4}. If $|\Lambda_I^0(R)|=0,$ then 
	$$
	|R^c|+s_0=|R^c|+|\Lambda^1_I(R)|+2  \geq (s-1)+2=s+1;
	$$ if $|\Lambda_I^0(R)|=1,$ then
	$$
	|R^c|+s_0=|R^c|+|\Lambda^1_I(R)|+1\geq s+1;
	$$
	if $|\Lambda_I^0(R)|\geq 2,$ then
	$$
	|R^c|+s_0=|R^c|+|\Lambda^1_I(R)|\geq s.
	$$
	From these, if $|\Lambda_I^0|\leq 2,$ then $|\Lambda_I^0(R)|$ can only be $0$ or $1$ for any $R\subseteq [r]$ and this implies that
	\begin{align*}
	\int_0^1\e\bigl|\tilde\e V_I(t)\bigr|dt\leq \frac{C_s}{n^{(s+1)/2}}
	\end{align*}
	and if $|\Lambda_I^0|\geq 3,$ then $|\Lambda_I^0(R)|$ could be larger than $2$ for some $R\subseteq [r]$ and hence,
		\begin{align*}
	\int_0^1\e\bigl|\tilde\e V_I(t)\bigr|dt\leq \frac{C_s}{n^{s/2}}
	\end{align*}
	for some constant $C_s>0.$ This completes our proof.
	
\end{proof}

\subsection{Proof of Proposition \ref{thm1}: first moment}\label{sec6.4}

Recall \eqref{eq-1}. Our proof will be completed once we establish that
\begin{align}\label{thm1:proof:eq1}
\frac{1}{n^{1+(r+1)/2}}\sum_{I\in \mathcal{I}_{r}}\int_0^1\e|\tilde\e V_I(t)|dt\leq \frac{C}{n^{1/2}},
\end{align}  
where $C$ is an absolute constant independent of $n.$ Recall from the definition of $\mathcal{T}_r(s,b)$ that when $b=1,$ $\mathcal{I}_r(0,b)=\emptyset$ and $\mathcal{I}_{r}(s,b)\neq \emptyset$ for all $1\leq s\leq r+1$ and that  the set $\mathcal{I}_{r}(s,b)$ is nonempty only if $0\leq s\leq r+1-b.$  From these,
\begin{align}
\begin{split}\label{thm1:proof:eq0}
\sum_{I\in \mathcal{I}_{r}}\tilde\e V_I(t)&=\sum_{s=0}^{r+1}\sum_{I\in \mathcal{I}_{r}(s)}\tilde\e V_I(t)\\
&=\sum_{s=1}^{r+1}\sum_{I\in \mathcal{I}_{r}(s,1)}\tilde\e V_I(t)+\sum_{s=1}^{r-1}\sum_{I\in \mathcal{I}_{r}(s,2)}\tilde\e V_I(t)+\sum_{b=3}^{r+1}\sum_{s=0}^{r+1-b}\sum_{I\in \mathcal{I}_{r}(s,b)}\tilde\e V_I(t).
\end{split}
\end{align}
In what follows, we let $C_{r,b,s}$ and $C_{r,b,s}'$ be absolute constants independent of $n.$
Here, from the first case of Proposition \ref{prop3} and Lemma \ref{lem5}, the first two summations can be controlled by 
\begin{align*}
\frac{1}{n^{1+(r+1)/2}}\sum_{I\in \mathcal{I}_{r}(s,1)}\int_0^1\e|\tilde\e V_I(t)|dt&\leq \frac{1}{n^{1+(r+1)/2}}\cdot C_{r,1,s}n^{\lfloor \frac{r+1-s}{2}\rfloor +s+1}\cdot\frac{C_{r,1,s}'}{n^{(s+1)/2}}\\
&=C_{r,1,s}C_{r,1,s}'n^{\lfloor \frac{r+1-s}{2}\rfloor -\frac{(r+1-s)}{2}-\frac{1}{2}}\\
&\leq C_{r,1,s}C_{r,1,s}'n^{-1/2}
\end{align*}
for $1\leq s\leq r+1$
and
\begin{align*}
\frac{1}{n^{1+(r+1)/2}}\sum_{I\in \mathcal{I}_{r}(s,2)}\int_0^1\e|\tilde\e V_I(t)|dt&\leq \frac{1}{n^{1+(r+1)/2}}\cdot C_{r,2,s}n^{\lfloor \frac{r-1-s}{2}\rfloor +s+2}\cdot\frac{C_{r,2,s}'}{n^{(s+1)/2}}\\
&=C_{r,2,s}C_{r,2,s}'n^{\lfloor \frac{r-1-s}{2}\rfloor -\frac{(r-1-s)}{2}-\frac{1}{2}}\\
&\leq C_{r,2,s}C_{r,2,s}'n^{-1/2}
\end{align*}
for $0\leq s\leq r-1$. To control the second summation, from the second case of Proposition \ref{prop3} and Lemma \ref{lem5},
\begin{align*}
\frac{1}{n^{1+(r+1)/2}}\sum_{I\in \mathcal{I}_{r}(s,b)}\int_0^1\e|\tilde\e V_I(t)|dt&\leq \frac{1}{n^{1+(r+1)/2}}\cdot C_{r,b,s}n^{\lfloor\frac{r+1-b-s}{2}\rfloor +s+2}\cdot\frac{C_{r,b,s}'}{n^{s/2}}\\
&= C_{r,b,s}C_{r,b,s}'   n^{\lfloor \frac{r+1-b-s}{2}\rfloor -\frac{r+1-b-s}{2}-(\frac{b}{2}-1)}\\
&\leq C_{r,b,s}C_{r,b,s}'n^{-(\frac{b}{2}-1)}
\end{align*}
for $0\leq s\leq r+1-b.$ Plugging these into \eqref{thm1:proof:eq0} yields \eqref{thm1:proof:eq1} and this completes our proof.

\subsection{Proof of Proposition \ref{thm1}: second moment}\label{sec6.5}

Our approach is the same as that for the first moment. Set 
\begin{align*}
\Psi_{k,n}(t)&=\tilde\e \Phi_{k,n}(A(t))^2,\,\,0<t<1.
\end{align*}
Note that
\begin{align*}
\e \bigl|\tilde\e \Phi_{k,n}(A)^2-\tilde\e \Phi_{k,n}(G)^2\bigr|=\e \bigl|\Psi_{k,n}(1)-\Psi_{k,n}(0)\bigr|.
\end{align*}
To control this expectation, we again consider the derivative
\begin{align*}
\frac{d}{dt}\Psi_{k,n}(t)=2\tilde\e\Phi_{k,n}(A(t))\frac{d}{dt}\Phi_{k,n}(A(t)).
\end{align*} 
Again, our goal would be to show that
\begin{align*}
\int_0^1 \e\bigl|\tilde\e \Psi_{k,n}'(t)\bigr|dt
\end{align*}
is bounded above, up to an absolute constant, by $n^{-1/2}.$ To see this, recall from Section \ref{sec:interpolation} that
$
\Phi_{k,n}'(t)
$
can be written as a summation, in which each summand is of the form $$
\frac{1}{n}\sum_{i=1}^nw_i^r(t)
$$
for some $0\leq r\leq k-1.$ In a similar manner, from \eqref{eq-1} and \eqref{V}, $\e\Psi_{k,n}'(t)$ can also be written as a sum, in which each term is equal to 
\begin{align*}
\tilde\e\Phi_{k,n}(A(t))\Bigl(\frac{1}{n}\sum_{i=1}^nw_i^r(t)\Bigr)=\frac{1}{n^{2+(r+1)/2}}\sum_{i=1}^n\sum_{I\in \mathcal{I}_{r}}\tilde\e \overline V_{i,I}(t),
\end{align*}
where 
\begin{align*}
\overline{V}_{i,I}(t):=\tilde\e U_i^{[r+2]}(A(t))\Bigl(\prod_{l=0}^{r+1} U_{i_{l}}^{[l]}(A(t))\Bigr) \Bigl(\prod_{l=1}^r\bigl(a_{i_{l+1},i_{l}}(t)+n^{-1/2}z_{i_{l+1},i_l}\bigr)\Bigr)\dot{a}_{i_1,i_0}(t)
\end{align*}
for
\begin{align*}
U_{i}^{[r+2]}(A(t)):=\phi(u_i^{[k]}(t),\ldots,u_i^{[0]}(t)),\,\,\forall i\in [n].
\end{align*}
Here,  $\overline{V}_{i,I}(t)$ is essentially the same as $V_I(t)$ (see \eqref{V}) except that it contains one extra term $U_{i}^{[r+2]}(A(t))$. In view of the proof of Proposition~\ref{prop3}, an identical argument implies that for any $i\in [n]$ and $I\in \mathcal{I}_r(s)$, 
\begin{align*}
\int_0^1\e\bigl|\tilde\e \overline{V}_{i,I}(t)\bigr|dt&\leq \left\{
\begin{array}{ll}
\frac{C_s'}{n^{(s+1)/2}},&\mbox{if $|\Lambda_I^0|\leq 2$},\\
\\
\frac{C_s'}{n^{s/2}},&\mbox{if $|\Lambda_I^0|\geq 3$},
\end{array}\right.
\end{align*}
where $C_s'$ is a universal constant independent of $n$ and $i.$ Consequently, as in the proof of \eqref{thm1:proof:eq1}, it follows that
\begin{align*}
\frac{1}{n^{1+(r+1)/2}}\sum_{I\in \mathcal{I}_{r}}\int_0^1\e\bigl|\tilde\e \overline V_{i,I}(t)\bigr|dt\leq \frac{C}{n^{1/2}}
\end{align*}
for some constant $C$ independent of $n$ and $i$ and the same inequality remains valid after taking $n^{-1}\sum_{i=1}^n$,
\begin{align*}
\frac{2}{n^{1+(r+1)/2}}\sum_{i=1}^n\sum_{I\in \mathcal{I}_{r}}\int_0^1\e\bigl|\tilde\e \overline V_{i,I}(t)\bigr|dt\leq \frac{C}{n^{1/2}},
\end{align*}
which completes our proof.

\section{Proof of Theorem \ref{cor1}}\label{sec:sec6}

We establish the proof of Theorem~\ref{cor1}.  For notational convenience, if $a_n$ and $b_n$ are two random variables, we denote 
$$
a_n\asymp b_n
$$
if $|a_n-b_n|\to 0$ in probability; if they are $n$-dimensional random vectors, then this notation means that in probability,
$$
\lim_{n\to\infty}\frac{1}{n}\|a_n-b_n\|_2^2=0.
$$

To begin with, recall from Lemma \ref{add:lem3} and Theorem \ref{concentration} that due to the Lipschitz property of the functions $(F_{k})_{k\geq 0}$, the AMP orbit defined in Definition \ref{def1} is uniformly square-integrable and the average along the Gaussian AMP orbit is concentrated with respect to $\tilde{\e}$. Here, in the setting of Definition \ref{def2}, since both $f_{\ell}$ and its first-order derivatives are Lipschitz, an identical argument also allows to show that $(v^{[k]})_{k\geq 0}$ is uniformly squared-integrable and when $X=G,$ its average along the orbit is self-averaged with respect to $\tilde{\e}$. More precisely, for any Lipschitz $\phi \in C(\mathbb{R}^{k+1})$, 
\begin{align}\label{con}
\frac{1}{n}\sum_{i\in [n]}\phi(v^{[k]}(G),\ldots, v^{[0]}(G))\asymp \frac{1}{n}\tilde\e\sum_{i\in [n]}\phi(v^{[k]}(G),\ldots, v^{[0]}(G)).
\end{align}
Denote by $(v^{G,[k]})_{k\geq 0}$ the AMP orbit in Definition \ref{def2} with the replacement of $b_{k,j}$ by $$
b_{k,j}^G:=\frac{1}{n}\tilde\e\sum_{i\in [n]}\frac{\partial f_k}{\partial u_i^{[j]}}\bigl(v^{[k]}(G),\ldots,v^{[0]}(G)\bigr)
$$
and initialization $v^{G,[0]}=u^0.$

\begin{lemma} \label{add:LEM2}For any Lipschitz $\phi\in C(\mathbb{R}^{k+1}),$
	\begin{align*}
	\frac{1}{n}\sum_{i\in [n]}\phi\bigl(v_i^{G,[k]}(G),\ldots,v_i^{G,[0]}(G)\bigr)\asymp\frac{1}{n}\sum_{i\in [n]}\phi\bigl(v_i^{G,[k]}(A),\ldots,v_i^{G,[0]}(A)\bigr).
	\end{align*}
\end{lemma}

\begin{proof}
	Consider the initialization 
	\begin{align*}
	u^{[0]}(X)&=v^{[0]}(X)=u^0.
	\end{align*}
	Set
	\begin{align*}
	u^{[1]}(X)&=0,\\
	u^{[2]}(X)& =  f_0(u^{[0]}(X)),\\
	u^{[3]} (X)&= \hx_n u^{[2]} (X)= \hx_n f_0(u^{[0]}(X)) 
	\end{align*}
	and for $\ell\geq 1,$ set
	\begin{align*}
	u^{[3\ell+1]} (X)&= 0,\\
	u^{[3\ell+2]}(X) &= f_\ell(u^{[3\ell]}(X), u^{[3(\ell-1)]}(X),\ldots, u^{[3]}(X), u^{[0]}(X)),\\
	u^{[3\ell+3]}(X) &= \hx_n u^{[3\ell+2]} (X)- \sum_{j=1}^\ell b_{\ell,j}^G u^{[3j-1]}(X).
	\end{align*}
	The main feature of this construction is that $u^{[3\ell]}(X)=v^{G,[\ell]}(X)$ for all $\ell\geq 0.$ Note that $b_{\ell,j}^G$ depends only on $u^0,Z$ and it is uniformly bounded. Although Definition \ref{def1} assumes that $F_k$'s are nonrandom, with no essential changes to the proof, Theorem \ref{thm0} indeed extends to randomized $F_k$'s that are dependent only on $u^0,Z$ and the Lipschitz constants of $F_k$'s are bounded by some constants independent of $u^0,Z.$ Hence, the assertion follows by applying Theorem~\ref{thm0} to $(u^{[\ell]})_{\ell\geq 0}$ and $\phi(u^{[3k]},u^{[3(k-1)]},\ldots,u^{[0]}).$ 
\end{proof}

\begin{lemma}\label{add:LEM1}
For any $k\geq 0,$
\begin{align*}
v^{[k]}(G)\asymp v^{G,[k]}(G).
\end{align*}
\end{lemma}

\begin{proof}
We argue by induction. Obviously the assertion is valid for $k=0.$ Assume that there exists some $k'\geq 0$ such that it is also valid for all $0\leq k\leq k'.$ 
From the triangle inequality, 
\begin{align*}
&\bigl\|v^{[k'+1]} (G)- v^{G,[k'+1]}(G)\bigr\|_2 \\
\leq&\; \|\hg_n\|_2 	\bigl\|f_{k'} ({v}^{[k']}(G), \ldots, {v}^{[0 ]}(G)) - f_k (v^{G,[k']}(G), \ldots, v^{G,[0 ]}(G))\bigr\|_2 \\
&+\sum_{j=1}^{k'}  |b_{k' ,j}(G)|	\bigl\|f_{j-1}({v}^{[j-1]}(G),\ldots, {v}^{[0]}(G)) - f_{j-1}(v^{G,[j-1]}(G),\ldots, v^{G,[0]}(G))\bigr\|_2\\
&+\sum_{j=1}^{k'}  |b_{k' ,j}(G)-b_{k' ,j}^G|	\bigl\|f_{j-1}(v^{G,[j-1]}(G),\ldots, v^{G,[0]}(G))\bigr\|_2.
\end{align*}
By induction hypothesis, Lemma \ref{add:lem5}, and noting that the first-order partial derivatives of $f_j$'s are uniformly bounded, the first two terms after dividing by $\sqrt{n}$ converge to zero in probability. As for the last term, note that \eqref{con}  implies that $b_{k',j}^G\asymp b_{k',j}(G)$ for all $1\leq j\leq k'.$ Also, note that from the relation, $u^{[3\ell]}=v^{G,[\ell]}$, in the proof of Lemma \ref{add:LEM2}, the Lipschitz property of $f_{j-1}$ and Lemmas \ref{add:lem5} and \ref{add:lem3} (here, again Lemma \ref{add:lem3} is valid despite of the fact that $F_k$'s are dependent on $Z$ and $u^0$) imply 
\begin{align}
\label{add:EQ12}
\sup_{n\geq 1}\frac{1}{\sqrt{n}}\e\bigl\|f_{j-1}(v^{G,[j-1]}(G),\ldots, v^{G,[0]}(G))\bigr\|_2<\infty.
\end{align} 
Hence, the third term also vanishes in probability and this validates the announced result.
\end{proof}

We are ready to prove Theorem \ref{cor1}, namely, for any $k\geq 0$ and Lipschitz $\phi \in C(\mathbb{R}^{k+1}),$
\begin{align}\label{add:EQ11}
\frac{1}{n}\sum_{i\in [n]}\phi\bigl(v_i^{[k]}(G),\ldots,v_i^{[0]}(G)\bigr)\asymp \frac{1}{n}\sum_{i\in [n]}\phi\bigl(v_i^{[k]}(A),\ldots,v_i^{[0]}(A)\bigr).
\end{align}
We argue by induction. Evidently this is valid for $k=0$. Assume that there exists some $k'\geq 0$ such that it is also valid for all $0\leq k\leq k'.$ 
From Lemmas~\ref{add:LEM2} and \ref{add:LEM1}, 
\begin{align*}
\frac{1}{n}\sum_{i\in [n]}\phi\bigl(v_i^{[k'+1]}(G),\ldots,v_i^{[0]}(G)\bigr)&\asymp \frac{1}{n}\sum_{i\in [n]}\phi\bigl(v_i^{G,[k'+1]}(G),\ldots,v_i^{G,[0]}(G)\bigr)\\
&\asymp\frac{1}{n}\sum_{i\in [n]}\phi\bigl(v_i^{G,[k'+1]}(A),\ldots,v_i^{G,[0]}(A)\bigr).
\end{align*}
We claim that
\begin{align*}
\frac{1}{n}\sum_{i\in [n]}\phi\bigl(v_i^{G,[k'+1]}(A),\ldots,v_i^{G,[0]}(A)\bigr)\asymp \frac{1}{n}\sum_{i\in [n]}\phi\bigl(v_i^{[k'+1]}(A),\ldots,v_i^{[0]}(A)\bigr).
\end{align*}
If this is valid, then \eqref{add:EQ11} is also true for $k+1$ and this would complete our proof. It suffices to show that
\begin{align*}
v^{G,[k]}(A)\asymp v^{[k]}(A),\,\,\forall 0\leq k\leq k'+1.
\end{align*}
Easy to see that this is valid if $k=0.$ Assume that there exists some $0\leq k''\leq k'$ such that this equation holds for all $0\leq k\leq k''.$ Write
\begin{align*}
&\bigl\|v^{[k''+1]}(A)-v^{G,[k''+1]}(A)\bigr\|_2\\
&\leq \|\ha_n\|_2\bigl\|f_{k''}(v^{[k'']}(A),\ldots,v^{[0]}(A))-f_{k''}(v^{G,[k'']}(A),\ldots,v^{G,[0]}(A))\bigr\|\\
&+\sum_{j=1}^{k''} |b_{k'',j}|\bigl\|f_{j-1}(v^{[j-1]}(A),\ldots,v^{[0]}(A))-f_{j-1}(v^{G,[j-1]}(A),\ldots,v^{G,[0]}(A))\bigr\|_2\\
&+\sum_{j=1}^{k''}|b_{k'',j}(A)-b_{k'',j}^G|\bigl\|f_{j-1}(v^{G,[j-1]}(A),\ldots,v^{G,[0]}(A))\bigr\|_2.
\end{align*}
Here, from the induction hypothesis, after dividing by $\sqrt{n},$ the first two lines vanish in probability by using the fact that $b_{k'',j}$ is uniformly bounded and Lemma \ref{add:lem5}. As for the last one, write
\begin{align*}
b_{k'',j}(A)-b_{k'',j}^G=(b_{k'',j}(A)-b_{k'',j}(G))+(b_{k'',j}(G)-b_{k'',j}^G).
\end{align*}
Note that  \eqref{con} implies $b_{k'',j}^G\asymp b_{k'',j}(G)$, while the induction hypothesis of \eqref{add:EQ11} implies that $b_{k'',j}^G\asymp b_{k'',j}^A$.  Hence, $b_{k'',j}(A)\asymp b_{k'',j}^G.$
This and \eqref{add:EQ12} imply that the third line also vanishes in probability. Hence, $v^{[k''+1]}(A)\asymp v^{G,[k''+1]}(A)$ and this completes the proof of our claim.

\section{Proof of Theorem \ref{cor2}}\label{sec:sec7}

We establish the proof of Theorem \ref{cor2}. Our strategy is to approximate the principal eigenvector by the power method. In view of this, it is essentially a special case of the generalized AMP in Definition \ref{def1}. Once this is done, universality would follow by an analogous argument as that for Theorem \ref{cor2}. Again, we adapt the notation $a_n\asymp b_n$ from Section \ref{sec:sec6}.


\subsection{Power method}

The well-known power method states that if the principal eigenvalue stays a gap away from the other eigenvalues, then one can generate the principal eigenvector via an iteration procedure. 

\begin{lemma}[Power method]\label{pm}
	Let $Y\in M_n(\mathbb{R})$ and $y\in \mathbb{R}^n$ with $\|y\|_2=1.$ Let $\lambda_1\geq \cdots\geq \lambda_n$ be the eigenvalues of $Y$ satisfying $\lambda_1\geq \max_{2\leq r\leq n}|\lambda_r|$ and $y^1$ be the normalized eigenvector associated to $\lambda_1$. If $\lambda_1\neq 0$ and $y\not\perp y^1$, then for any $d\geq 1,$
	\begin{align}\label{pm:eq1}
	\Bigl\|\frac{Y^dy}{\|Y^dy\|_2}-\mathrm{sign}\bigl(\la y^1,y\ra \bigr)y^1\Bigr\|_2&\leq \frac{1}{|\la y^1,y\ra|}\max_{2\leq r\leq n}\Bigl|\frac{\lambda_r}{\lambda_1}\Bigr|^{d}.
	\end{align}
\end{lemma}

\begin{proof}
	Let $y^1,\ldots,y^n$ be the orthonormal eigenvectors associated to $\lambda_1,\ldots,\lambda_n.$ 
	Write
	\begin{align*}
	y&=c_1y^1+\cdots+c_ny^n,
	\end{align*}
	where $c=(c_1,\ldots,c_n)\in \mathbb{R}^n$ satisfies $\|c\|_2=\|y\|_2=1.$ Note that
	\begin{align*}
	Y^dy&=c_1\lambda_1^dy^1+\cdots+c_n\lambda_n^dy^n
	\end{align*} 
	and
	\begin{align}\label{add:eq-6}
	\|Y^dy\|_2&=\bigl(c_1^2\lambda_1^{2d}+\cdots+c_n^2\lambda_n^{2d}\bigr)^{1/2}.
	\end{align}
	From these,
	\begin{align*}
	\frac{Y^dy}{\|Y^dy\|_2}&=\frac{c_1\lambda_1^dy^1+\cdots+c_n\lambda_n^dy^n}{\bigl(c_1^2\lambda_1^{2d}+\cdots+c_n^2\lambda_n^{2d}\bigr)^{1/2}}\\
	&=\mbox{sign}(c_1)\frac{|c_1\lambda_1^d|y^1+\mbox{sign}(c_1)\sum_{r=2}^nc_r\lambda_r^dy^r}{\bigl(c_1^2\lambda_1^{2d}+\cdots+c_n^2\lambda_n^{2d}\bigr)^{1/2}}\\
	&=\mbox{sign}(c_1)\frac{y^1+\mbox{sign}(c_1)\sum_{r=2}^n\frac{c_r\lambda_r^d}{|c_1\lambda_1^d|}y^r}{\bigl(1+\sum_{r=2}^n\frac{c_r^2}{c_1^2}\bigl(\frac{\lambda_r}{\lambda_1}\bigr)^{2d}\bigr)^{1/2}}.
	\end{align*}
	If we denote $$
	\Pi=\sum_{r=2}^n\frac{c_r^2}{c_1^2}\Bigl(\frac{\lambda_r}{\lambda_1}\Bigr)^{2d},$$
	then
	\begin{align*}
	\Bigl\|\frac{Y^dy}{\|Y^dy\|_2}-\mbox{sign}(c_1) y^1\Bigr\|_2&=\Bigl(\frac{(1-\sqrt{1+\Pi})^2+\Pi}{1+\Pi}\Bigr)^{1/2}\leq \Bigl(\frac{\Pi^2+\Pi}{1+\Pi}\Bigr)^{1/2}=\Pi^{1/2},
	\end{align*}
	where we used that $\sqrt{1+x}-1\leq x$ for $x\geq 0.$ Now, the assertion follows by
	\begin{align*}
	\Pi&\leq \frac{1}{c_1^2}\max_{2\leq r\leq n}\Bigl|\frac{\lambda_r}{\lambda_1}\Bigr|^{2d}\sum_{r=2}^nc_r^2\leq \frac{1}{|\la y,y^1\ra|^2}\max_{2\leq r\leq n}\Bigl|\frac{\lambda_r}{\lambda_1}\Bigr|^{2d}.
	\end{align*}
\end{proof}

We continue to show that the AMP orbits in Definition~\ref{def1} initialized by the principal eigenvector and the power method can be as close as we want by increasing the power iteration $d.$ Let $\varepsilon\in (0,1).$ For any $d\geq 1$, let
$$
{u}^{A,\varepsilon,d}(X)=\frac{\sqrt{n}\hx_n^du^0}{\tilde\e\|\ha_n^du^0\|_2+\sqrt{n}\varepsilon}
$$
and
$$
{u}^{G,\varepsilon,d}(X)=\frac{\sqrt{n}\hx_n^du^0}{\tilde\e\|\hg_n^du^0\|_2+\sqrt{n}\varepsilon}.
$$
Let $\lambda_1(\hx_n)\geq \cdots\geq \lambda_n(\hx_n)$ be the eigenvalues of $\hx_n$. Let $\psi^1(\hx_n)$ be the principal eigenvector of $\hx_n$ with $\|\psi^1(\hx_n)\|_2 = \sqrt{n}$. Recall the vector $\psi$ defined through \eqref{psi}. Let $v^{[k]}$, $v^{A,\varepsilon,d,[k]},$ and $v^{G,\varepsilon,d,[k]}$ be three AMP orbits that are defined via Definition \ref{def1} associated to the initializations $\psi$, $u^{A,\varepsilon,d}$, and $u^{G,\varepsilon,d}$, respectively. For any Lipschitz $\phi\in C(\mathbb{R}^{k+1})$, denote by $\phi_{k,n}^{\psi}$, $\phi_{k,n}^{u^{A,\varepsilon,d}}$, and $\phi_{k,n}^{u^{G,\varepsilon,d}}$ the averages of $\phi$ over these AMP orbits, respectively. 

\begin{lemma}\label{add:lem6} Assume that \eqref{cor2:eq0} and \eqref{cor2:eq2} are valid.
	Suppose that $\phi\in C(\mathbb{R}^{k+1})$ is Lipschitz. We have that in probability,
	\begin{align*}
	\lim_{\varepsilon\downarrow 0}\lim_{d\to \infty}\lim_{n\to\infty}\bigl|\phi_{k,n}^{\psi}(A)-\phi_{k,n}^{u^{A,\varepsilon,d}}(A)\bigr|&=0,\\
	\lim_{\varepsilon\downarrow 0}\lim_{d\to \infty}\lim_{n\to\infty}	\bigl|\phi_{k,n}^{\psi}(G)-\phi_{k,n}^{u^{G,\varepsilon,d}}(G)\bigr|&=0.
	\end{align*}
\end{lemma}
\begin{proof}
	We only need to establish the first equality. First of all, we claim that in probability
	\begin{align*}
	\lim_{\varepsilon\downarrow 0}\lim_{d\to \infty}\lim_{n\to\infty} \frac{1}{n} \bigl\|\psi(A)-u^{A,\varepsilon,d}(A)\bigr\|_2^2=0.
	\end{align*}
	Define
	$$
	u^{\varepsilon,d}(X):=\frac{\sqrt{n}\hx_n^du^0}{\|\hx_n^du^0\|_2+\varepsilon\sqrt{n}}.
	$$
	From \eqref{cor2:eq0}, \eqref{cor2:eq2}, and \eqref{add:eq-6}, there exist $0<\delta<1$ and $\delta',\delta''>0$ such that
	\begin{align*}
	&\limsup_{n\to\infty}\max_{2\leq r\leq n} \Bigl|\frac{\lambda_r(\ha_n)}{\lambda_1(\ha_n)}\Bigr| \leq \delta,\\
	&\liminf_{n\to\infty}\frac{1}{n}\bigl|\la \psi^1(\ha_n),u^0\ra\bigr|\geq \delta',
	\end{align*}
	and
	\begin{align*}
	\liminf_{n\to\infty}\frac{1}{\sqrt{n}}\|\ha_n^du^0\|_2&\geq \liminf_{n\to\infty}\frac{1}{n}\bigl|\la \psi^1( \ha_n),u^0\ra\bigr| \lambda_1(\ha_n)^{d} \geq \delta'(1+\delta'')^d.
	\end{align*}
	Note that \eqref{pm:eq1} implies that 
	\begin{align*}
	\Bigl\|\frac{u^{\varepsilon,d}(A)}{\sqrt{n}}-\frac{\psi(A)}{\sqrt{n}}\Bigr\|_2&\leq \frac{n}{|\la \psi^1(\ha_n),u^0\ra|}\max_{2\leq r\leq n}\Bigl|\frac{\lambda_r(\ha_n)}{\lambda_1(\ha_n)}\Bigr|^{d} + \frac{\varepsilon}{\frac{\|\ha_n^d u^0\|_2}{\sqrt{n}}+\varepsilon}.
	\end{align*}
	From these inequalities, as long as $d$ is large enough such that
	$$
	\frac{\delta^d}{\delta'}+ \frac{\varepsilon}{\delta'(1+\delta'')^d+\varepsilon}<\varepsilon,
	$$
	we have
	\begin{equation}
	\label{add:eq6}
	\limsup_{n\to\infty} \Bigl\|\frac{u^{\varepsilon,d}(A)}{\sqrt{n}}-\frac{\psi(A)}{\sqrt{n}}\Bigr\|_2 \leq \varepsilon.
	\end{equation}
	Next, note that
	\begin{align*}
	\frac{1}{\sqrt{n}}\bigl\|u^{\varepsilon,d}(A)-{u}^{A,\varepsilon,d}(A)\bigr\|_2
	&=	\frac{\|\ha_n^du^0\|_2}{\sqrt{n}}\Bigl|\frac{\|\ha_n^du^0\|_2/\sqrt{n}-\tilde{\e}\|\ha_n^du^0\|_2/\sqrt{n}}{\bigl(\|\ha_n^du^0\|_2/\sqrt{n}+\varepsilon\bigr)\bigl(\tilde\e\|\ha_n^du^0\|_2/\sqrt{n}+\varepsilon\bigr)}\Bigr|\\
	&\leq  \frac{\|\ha_n^du^0\|_2}{\varepsilon^2\sqrt{n}}\Bigl|\frac{\|\ha_n^du^0\|_2}{\sqrt{n}}-\frac{\tilde{\e}\|\ha_n^du^0\|_2}{\sqrt{n}}\Bigr|.
	\end{align*}
	From Lemmas \ref{add:lem5} and \ref{lem:norm} (established below), in probability,
	\begin{align*}
	\lim_{n\to\infty}\frac{1}{n}\bigl\|{u}^{\varepsilon,d}(A)-u^{A,\varepsilon,d}(A)\bigr\|_2^2=0.
	\end{align*}
	Combining this with \eqref{add:eq6}, our claim follows.
	
	Now to establish the first assertion, note that similar to Proposition \ref{prop0}, the assumption that both $f_k$ and its first-order partial derivative are Lipschitz ensures that the average $\phi_{k,n}(X)$ of $\phi$ along the AMP orbit  $(v^{[k]})_{k\geq 1}$ is Lipschitz with respect to its initialization. The proof of this fact follows directly from the same proof as that of Proposition \ref{prop0}. Hence, it suffices to show that in probability,
\begin{align*}
\lim_{\varepsilon\downarrow 0}\lim_{d\to\infty}\lim_{n\to\infty}\frac{1}{n}\bigl\|v^{[k]}(A)-v^{A,\varepsilon,d,[k]}(A)\bigr\|_2=0,\,\,\forall k\geq 0.
\end{align*}
When $k=0$, this is valid by our claim. Assume that this is also valid for all $0\leq k\leq k'$ for some $k'.$ We prove that this is also valid for $k'+1.$ To see this, we use triangle inequality to write
\begin{align*}
&\bigl\|v^{[k'+1]}(A)-v^{A,\varepsilon,d,[k'+1]}(A)\bigr\|_2\\
&\leq \|\ha_n\|_2\bigl\|f_{k'}(v^{[k']}(A),\ldots,v^{[0]}(A))-f_{k'}(v^{A,\varepsilon,d,[k']}(A),\ldots,v^{A,\varepsilon,d,[0]}(A))\bigr\|\\
&+\sum_{j=1}^{k'} |b_{k',j}(A)|\bigl\|f_{j-1}(v^{[j-1]}(A),\ldots,v^{[0]}(A))-f_{j-1}(v^{A,\varepsilon,d,[j-1]}(A),\ldots,v^{A,\varepsilon,d,[0]}(A))\bigr\|_2\\
&+\sum_{j=1}^{k'}|b_{k',j}(A)-b_{k',j}^{A,\varepsilon,d}(A)|\bigl\|f_{j-1}(v^{A,\varepsilon,d,[j-1]}(A),\ldots,v^{A,\varepsilon,d,[0]}(A))\bigr\|_2,
\end{align*}
where
\begin{align*}
b_{k',j}^{A,\varepsilon,d}(A)&:=\frac{1}{n}\sum_{i\in [n]}\frac{\partial f_{k'}}{\partial v_{i}^{[j]}}\bigl(v_i^{A,\varepsilon,d,[k']}(A),\ldots, v_i^{A,\varepsilon,d,[0]}(A)\bigr).
\end{align*}
Here by Lemma \ref{add:lem5}, Lemma \ref{add:lem3}, and the Lipschitz property of $f_{k}$ and its first-order derivative, these terms vanish in probability from the induction hypothesis.
\end{proof}

At the end of this subsection, we establish the following lemma, which was used in the above proof.

\begin{lemma}
	\label{lem:norm}
	\begin{align}
	\begin{split}
	\label{eq:normconv3}
	\lim_{n\to\infty}\frac{1}{n}\e\bigl|\|\ha_n^du^0\|_2-\|\hg_n^du^0\|_2\bigr|^2&=0,
	\end{split}\\
	\begin{split}	\label{eq:normconv4}
	\lim_{n\to\infty}\frac{1}{n}\e\bigl|\|\hg_n^du^0\|_2-\tilde{\e}\|\hg_n^du^0\|_2\bigr|^2&=0.
	\end{split}
	\end{align}
\end{lemma}

\begin{proof}
	We establish \eqref{eq:normconv3} first. Consider the AMP orbit
	\begin{align*}
	u^{[0]}(X) &= u^0,\\
	u^{[k+1]}(X) &= \hx_n u^{[k]}(X) = \hx_n^k u^0,\quad k\geq 0.
	\end{align*}
	Note that $u^{[d+1]}(G)=\hg_n^du^0$ and $u^{[d+1]}(A)=\ha_n^du^0.$ Recall that Theorem \ref{thm0} implies that if $\phi\in C(\mathbb{R}^{k+1})$ is Lipschitz, then $\lim_{n\to\infty}|\phi_{k,n}(A)-\phi_{k,n}(G)|=0$ in probability. However, we can not apply this result directly to $\|u^{[d+1]}(X)\|_2^2$ since each term inside the summation is not Lipschitz. To this end, we adapt a truncation argument. For any $M>1,$ let $\rho\in C^1(\mathbb{R})$ be uniformly bounded by $2M^2$ and satisfy that $\rho(x)=x^2$ on $[-M,M]$ and $\rho(x)=(M+1)^2$ for $x\notin [-(M+1),M+1]$.  Note that for any $x\in \mathbb{R}$,
	\begin{align*}
	x^2-\rho(x)&=(x^2-\rho(x))1_{\{|x|\in [M,M+1]\}}+(x^2-(M+1)^2)1_{\{|x|\in (M+1,\infty)\}},
	\end{align*}
	which implies that
	\begin{align*}
	|x^2-\rho(x)|^2&\leq 4(x^4+9M^4)1_{\{|x|\geq M\}}.
	\end{align*}
	It follows that from the Jensen  and Cauchy-Schwarz inequalities,
	\begin{align*}
	&\e \Bigl|\frac{1}{n}\sum_{i=1}^n\bigl(u_i^{[d+1]}(G)^2-\rho(u_i^{[d+1]}(G))\bigr)\Bigr|^2\\
	&\leq \frac{4}{n}\e \sum_{i=1}^n\bigl((\e (u_i^{[d+1]}(G))^8)^{1/2}\p(|u_i^{[d+1]}(G)|\geq M)^{1/2}+9M^4\p(|u_i^{[d+1]}(G)|\geq M)\bigr).
	\end{align*}
	Here, from Proposition \ref{prop2}, there exists a constant $C>0$ independent of $n$, $i,$ and $M>1,$
	\begin{align*}
	\e |u_i^{[d+1]}(G)|^8&\leq C
	\end{align*}
	and
	\begin{align*}
	\p(|u_i^{[d+1]}(G)|\geq M)\bigr)&\leq \frac{\e |u_i^{[d+1]}(G)|^8}{M^8}\leq \frac{C}{M^8}.
	\end{align*}
	Consequently,
	\begin{align*}
	\e \Bigl|\frac{1}{n}\sum_{i=1}^n\bigl(u_i^{[d+1]}(G)^2-\rho(u_i^{[d+1]}(G))\bigr)\Bigr|^2&\leq \frac{40C}{M^4}.
	\end{align*}
	Similarly, the same inequality is valid for $A$. Now from Theorem \ref{thm0} and the dominated convergence theorem, $$
	\lim_{n\to\infty}\e \Bigl|\frac{1}{n}\sum_{i=1}^n\bigl(\rho(u_i^{[d+1]}(G))-\rho(u_i^{[d+1]}(A))\bigr)\Bigr|^2=0.
	$$
	Hence, we arrive at
	\begin{align*}
	\limsup_{n\to\infty}\frac{1}{n^2}\e \Bigl|\|u^{[d+1]}(G)\|_2^2-\|u^{[d+1]}(A)\|_2^2\Bigr|^2&\leq \frac{640C}{M^4}.
	\end{align*}
	Since this is valid for all $M>1$, this limit is indeed equal to zero. From this, since both $n^{-1/2}\|u^{[d+1]}(G)\|_2$ and $n^{-1/2}\|u^{[d+1]}(A)\|_2$ are uniformly square-integrable by Lemma~\ref{add:lem3}, the assertion \eqref{eq:normconv3} follows. The proof of \eqref{eq:normconv4} can be established by using Theorem \ref{concentration} and an identical truncation argument. As this part of the proof does not involve additional complications, we omit the details here.
\end{proof}

\subsection{Main argument}

We are ready to establish the proof of Theorem \ref{cor2}. Recall the AMP orbits $(v^{A,\varepsilon,d,[k]})_{k\geq 0}$ and $(v^{G,\varepsilon,d,[k]})_{k\geq 0}$ from last subsection. Define
\begin{align*}
v^{A,\varepsilon,d,[-d]}(X)&=\frac{\sqrt{n}u^0}{\tilde\e\|\ha_n^du^0\|_2+\sqrt{n}\varepsilon},\\
v^{G,\varepsilon,d,[-d]}(X)&=\frac{\sqrt{n}u^0}{\tilde\e\|\hg_n^du^0\|_2+\sqrt{n}\varepsilon}.
\end{align*}
For $-d\leq k\leq -1,$ set
\begin{align*}
v^{A,\varepsilon,d,[k+1]}(X)&=\hx_n v^{A,\varepsilon,d,[k]}(X),\\
v^{G,\varepsilon,d,[k+1]}(X)&=\hx_n v^{G,\varepsilon,d,[k]}(X).
\end{align*}
Note that $v^{A,\varepsilon,d,[0]}(X)=u^{A,\varepsilon,d}(X)$ and $v^{G,\varepsilon,d,[0]}(X)=u^{G,\varepsilon,d}(X)$. This implies that $(v^{A,\varepsilon,d,[k]})_{k\geq -d}$ and $(v^{G,\varepsilon,d,[k]})_{k\geq -d}$ are again AMP orbits with the initializations $v^{A,\varepsilon,d,[-d]}(X)$ and $v^{G,\varepsilon,d,[-d]}(X).$ The key feature of this construction is that the initializations are independent of $X$ and its norm is bounded above by $\|u^0\|_2/\varepsilon.$ Hence, the assumption \eqref{add:eq--1} is satisfied with possibly a larger $\sigma.$ From an identical argument as that of Theorem \ref{cor1}, we see that the AMP orbit $(v^{A,\varepsilon,d,[k]})_{k\geq -d}$ satisfies universality. In particular,  for any Lipschitz $\phi\in C(\mathbb{R}^{k+1}),$ this implies that
\begin{align*}
\phi_{k,n}^{u^{A,\varepsilon,d}}(A)\asymp \phi_{k,n}^{u^{A,\varepsilon,d}}(G).
\end{align*}
Finally, since Lemma \ref{lem:norm} implies that the initialization satisfies
\begin{align*}
v^{A,\varepsilon,d,[-d]}(G)=u^{A,\varepsilon,d}(G)\asymp u^{G,\varepsilon,d}(G)=v^{G,\varepsilon,d,[-d]}(G),
\end{align*}
the argument in the second half of the proof of Lemma \ref{add:lem6}  applies to the present setting and it yields that $$v^{A,\varepsilon,d,[k]}(G)\asymp v^{G,\varepsilon,d,[k]}(G)$$ for all $k\geq -d.$ Consequently,
\begin{align*}
\phi_{k,n}^{u^{A,\varepsilon,d}}(G)\asymp \phi_{k,n}^{u^{G,\varepsilon,d}}(G).
\end{align*}
From this and Lemma \ref{add:lem6}, the announced result follows.

\appendix 

\section{Approximate Gaussian integration by parts}
	This appendix gathers three inequalities of approximate Gaussian integration by parts. Let $s\geq 1$ be fixed. Let $a_1,\ldots,a_s$  be independent random variables with zero mean and unit variance. Suppose that $g_1,\ldots,g_s$ are i.i.d. standard standard normal and are independent of $a_1,\ldots,a_s.$ Set $$
	a_j(t)=\sqrt{t}a_{j}+\sqrt{1-t}g_j
	$$
	for $0\leq t\leq 1$ and $1\leq j\leq s.$ Set $a(t)=(a_1(t),\ldots,a_s(t)).$ In what follows, we denote $\alpha=(\alpha_1,\ldots,\alpha_s)\in (\{0\}\cup\mathbb{N})^s$, $\alpha!=\alpha_1!\cdots\alpha_s!$, and $|\alpha|=\alpha_1+\cdots+\alpha_s.$ Also, $\partial^\alpha=\partial_{x_1}^{\alpha_1}\cdots \partial_{x_s}^{\alpha_s}$ and $x^\alpha=x_1^{\alpha_1}\cdots x_s^{\alpha_s}.$

	\begin{lemma}\label{lem3}
		Let $f\in C^{s}(\mathbb{R}^{s})$. For any $0<t<1,$ we have that
		\begin{align}
		\begin{split}\label{lem3:eq1}
		&\bigl|\e f(a(t))a_1(t)\cdots a_{s}(t)\dot a_1(t)\bigr|\\
		&\leq \frac{\e[\dot a_1(t)^4]^{1/4}}{(s-1)!}\sum_{|\alpha|=s}\prod_{j=1}^s\e\bigl[ a_j(t)^{4(\alpha_j+1)}\bigr]^{1/4}\Bigl(\int_0^1 \e\bigl[\bigr|\partial^\alpha f(\xi a(t))\bigr|^2\bigr]d\xi\Bigr)^{1/2}.
		\end{split}
		\end{align}
	\end{lemma}

\begin{proof}
	From Taylor's theorem, for any $x\in \mathbb{R}^s$, 
	\begin{align*}
	f(x)&=\sum_{|\alpha|\leq s-1}\frac{\partial^\alpha f(0)}{\alpha!} x^\alpha+\frac{1}{(s-1)!}\sum_{|\alpha|=s}x^\alpha\int_0^1 \partial^\alpha f(\xi x)d\xi.
	\end{align*}
	From these, for any $\alpha$ satisfying that $|\alpha|\leq s-1$, we can write
	\begin{align*}
	\e a(t)^{\alpha}a_1(t)\cdots a_n(t)\dot a_1(t)&=\e a_1(t)^{\alpha_1+1}\dot a_1(t)\cdot \e a_2(t)^{\alpha_2+1}\cdots \e a_s(t)^{\alpha_s+1}(t).
	\end{align*}
	If $\alpha_j=0$ for some $2\leq j\leq s,$ then this expectation vanishes. If $\alpha_j\neq 0$ for all $2\leq j\leq s,$ then the condition $|\alpha|\leq s-1$ forces that $\alpha_1=0$ and $\alpha_2=\cdots=\alpha_s=1$ so that
	\begin{align*}
	\e a_1(t)^{\alpha_1+1}\dot a_1(t)\cdot \e a_2(t)^{\alpha_2+1}\cdots \e a_s(t)^{\alpha_s+1}(t)&=\e a_1(t) \dot a(t)\cdot \e a_2(t)^2\cdots \e a_s(t)^2=0 
	\end{align*}
	since
	\begin{align}
	\begin{split}\label{add:eq1}
	\e a_1(t) \dot a_1(t)&=\e (\sqrt{t}a_1+\sqrt{1-t}g_1)\Bigl(\frac{a_1}{\sqrt{t}}-\frac{g_1}{\sqrt{1-t}}\Bigr)\\
	&=\e a_1^2-\frac{\sqrt{t}}{\sqrt{1-t}}\e a_1g_1+\frac{\sqrt{1-t}}{\sqrt{t}}\e a_1g_1-\e g_1^2=0.
	\end{split}
	\end{align}
	From these,
	\begin{align*}
	&\bigl|\e f(a(t))a_1(t)\cdots a_s(t)\dot a_1(t)\bigr|\\
	&=\Bigl|\frac{1}{(s-1)!}\sum_{|\alpha|=s}\int_0^1 \e\bigl[a_1(t)\cdots a_s(t)\dot a_1(t)a(t)^\alpha\partial^\alpha f(\xi a(t))\bigr]d\xi\Bigr|\\
	&\leq \frac{1}{(s-1)!}\sum_{|\alpha|=s}\e\bigl[\bigl(a_1(t)\cdots a_s(t)\dot a_1(t)a(t)^\alpha\bigr)^2\bigr]^{1/2}\Bigl(\int_0^1 \e\bigl[\bigr|\partial^\alpha f(\xi a(t))\bigr|^2\bigr]d\xi\Bigr)^{1/2}.
	\end{align*}
	Finally applying the Cauchy-Schwarz inequality and the independence of $a_1(t),\ldots,a_s(t)$ to the first expectation in the last line ends our proof.
\end{proof}

\begin{lemma}\label{lem4}
	Let $f\in C^{s-1}(\mathbb{R}^{s})$. For any $0<t<1$ and $r\geq 2$, we have that
		\begin{align}
		\begin{split}\label{lem3:eq2}
		&\bigl|\e f(a(t))a_1(t)^r a_2(t)\cdots a_{s}(t)\dot a_1(t)\bigr|\\
		&\leq \frac{\e[\dot a_1(t)^4]^{1/4}}{(s-2)!}\sum_{|\alpha|=s-1}\e\bigl[a_1(t)^{4(\alpha_1+r)}\bigr]^{1/4}\prod_{j=2}^s\e\bigl[a_j(t)^{4(\alpha_j+1)}\bigr]^{1/4}\Bigl(\int_0^1 \e\bigl[\bigr|\partial^\alpha f(\xi a(t))\bigr|^2\bigr]d\xi\Bigr)^{1/2}.
		\end{split}
		\end{align}
	\end{lemma}

\begin{proof}
	Consider Taylor's expansion,
	\begin{align*}
	f(x)&=\sum_{|\alpha|\leq s-2}\frac{\partial^\alpha f(0)}{\alpha!} x^\alpha+\frac{1}{(s-2)!}\sum_{|\alpha|=s-1}x^\alpha\int_0^1 \partial^\alpha f(\xi x)d\xi.
	\end{align*}
	For any $\alpha$ satisfying $|\alpha|\leq s-2,$ we can write
	\begin{align*}
	\e a(t)^\alpha a_1(t)^r\cdots a_s (t)\dot a(t)&=\e a_1(t)^{\alpha_1+r}\dot a(t)\cdot \e a_2(t)^{\alpha_2+1}\cdots \e a_s (t)^{\alpha_s +1}(t).
	\end{align*}
	Observe that if $\alpha_j \geq 1$ for all $2\leq i\leq s,$ then $|\alpha| \geq s-1,$ which is not possible. Thus, one of the $\alpha_2,\ldots,\alpha_s $ must be zero so that this expectation vanishes. Consequently,
	\begin{align*}
	&\bigl| \e f(a(t)) a_1(t)^ra_2(t)\cdots a_s (t)\dot a_1(t)\bigr|\\
	&=\Bigl|\frac{1}{(s-2)!}\sum_{|\alpha|=s-1}\int_0^1 \e\bigl[a_1(t)^ra_2(t)\cdots a_s(t)\dot a_1(t)a(t)^\alpha\partial^\alpha f(\xi a(t))\bigr]d\xi\Bigr|\\
	&\leq \frac{1}{(s-2)!}\sum_{|\alpha|=s-1}\e\bigl[\bigl(a_1(t)^r\cdots a_s(t)\dot a_1(t)a(t)^\alpha\bigr)^2\bigr]^{1/2}\Bigl(\int_0^1 \e\bigl[\bigr|\partial^\alpha f(\xi a(t))\bigr|^2\bigr]d\xi\Bigr)^{1/2},
	\end{align*}
	which leads to the assertion by applying the Cauchy-Schwarz inequality to the first expectation in the last line.
\end{proof}

\begin{lemma}\label{lem6}
	Let $f\in C^{s+1}(\mathbb{R}^{s})$. For any $0<t<1$, we have that
		\begin{align}
		\begin{split}\label{lem3:eq3}
		&\bigl|\e f(a(t))a_2(t)\cdots a_s(t)\dot a_1(t)\bigr|\\
		&\leq \frac{\e[\dot a_1(t)^4]^{1/4}}{s!}\sum_{|\alpha|=s+1}\e \bigl[a_1(t)^{4\alpha_1}\bigr]^{1/4}\prod_{j=2}^{s}\e \bigl[a_j(t)^{4(1+\alpha_j)}\bigr]^{1/4}\Bigl(\int_0^1 \e\bigl[\bigr|\partial^\alpha f(\xi a(t))\bigr|^2\bigr]d\xi\Bigr)^{1/2}.
		\end{split}
		\end{align}
	\end{lemma}
	\begin{proof}
	Write
		\begin{align*}
	f(x)&=\sum_{|\alpha|\leq s}\frac{\partial^\alpha f(0)}{\alpha!} x^\alpha+\frac{1}{s!}\sum_{|\alpha|=s+1}x^\alpha\int_0^1 \partial^\alpha f(\xi x)d\xi,
	\end{align*}
	For $\alpha$ satisfying $|\alpha|\leq s,$ write
	\begin{align*}
	\e a(t)^\alpha a_2(t)\cdots a_s (t)\dot a_1(t)&=\e  a_1(t)^{\alpha_1}\dot a_1(t)\e a_2(t)^{\alpha_2+1}\cdots \e a_s (t)^{\alpha_s +1}.
	\end{align*}
	If $\alpha_j =0$ for some $2\leq j\leq s$, then this expectation is equal to zero. If $\alpha_j >0$ for all $2\leq j\leq s,$ then $\alpha_1=0$ or $\alpha_1=1.$ In the former case,  the expectation vanishes; in the latter case, this expectation is also equal to zero since $\e a_1(t)\dot a(t)=0$ due to \eqref{add:eq1}. Consequently,
	\begin{align*}
	&\bigl|\e f(a(t))a_2(t)\ldots a_s (t)\dot a_1(t)\bigr|\\
&=\Bigl|\frac{1}{s!}\sum_{|\alpha|=s+1}\int_0^1 \e\bigl[a_2(t)\cdots a_s(t)\dot a_1(t)a(t)^\alpha\partial^\alpha f(\xi a(t))\bigr]d\xi\Bigr|\\
&\leq \frac{1}{s!}\sum_{|\alpha|=s+1}\e\bigl[\bigl(a_2(t)\cdots a_s(t)\dot a_1(t)a(t)^\alpha\bigr)^2\bigr]^{1/2}\Bigl(\int_0^1 \e\bigl[\bigr|\partial^\alpha f(\xi a(t))\bigr|^2\bigr]d\xi\Bigr)^{1/2}.
	\end{align*}
	The rest of the proof follows by using the Cauchy-Schwarz inequality and the independence of $a_1(t),\ldots,a_s(t)$ to the first expectation of the last line.
\end{proof}

\bibliographystyle{plain}
{\footnotesize\bibliography{ref}}

\end{document}